\documentclass[11pt]{amsart}
\usepackage[margin=1.2in]{geometry}

\usepackage{mathrsfs}
\usepackage{amsfonts}
\usepackage{latexsym,epsfig}
\usepackage{mathabx}
\usepackage{amsmath, amscd, amsthm,amssymb}
\usepackage[pdftex]{color}
\usepackage{comment}
\usepackage{marginnote}
\usepackage[colorlinks,citecolor=blue,linkcolor=red]{hyperref}
\usepackage[nameinlink]{cleveref}
\usepackage{enumitem}
\usepackage[normalem]{ulem}
\usepackage{tikz-cd}
\tikzcdset{arrow style=math font}
\tikzcdset{every label/.append style = {font = \small}}

\usepackage{todonotes}

\newcommand{\RR}[0]{\mathbb{R}}

\newcommand{\Z}{\mathbb{Z}}

\newcommand{\PP}[0]{\mathcal{P}}

\newcommand{\ZZ}[0]{\mathbb{Z}}

\newcommand{\R}{\mathbb{R}}

\newcommand{\wt}{\widetilde}
\newcommand{\mc}{\mathcal}
\newcommand{\mf}{\mathfrak}
\newcommand{\cone}{\mathrm{cone}}

\newcommand{\ol}{\overline}
\newcommand{\wh}{\widehat}
\newcommand{\wc}{\widecheck}
\newcommand{\uM}{\wt M}

\newcommand{\mr}{\mathring}
\newcommand{\del}{\partial}

\newcommand{\hbs}{\tau^{(2)}}
\newcommand{\onto}{\twoheadrightarrow}

\newcommand{\gr}{\mathrm{gr}}

\usepackage{hyperref}

\newtheorem{thm}{Theorem}

\numberwithin{equation}{section}
\newtheorem{theorem}{Theorem}[section]
\newtheorem{lemma}[theorem]{Lemma}
\newtheorem{proposition}[theorem]{Proposition}
\newtheorem{corollary}[theorem]{Corollary}

\newtheorem{convention}[theorem]{Convention}
\newtheorem{fact}[theorem]{Fact}
\newtheorem{claim}[theorem]{Claim}

\theoremstyle{definition}
\newtheorem{question}{Question}
\newtheorem{remark}[theorem]{Remark}

\newcommand{\FF}{\mathcal{F}}

\newcommand{\T}{\mathcal{T}}

\newcommand{\C}{\mathcal{C}}

\newcommand{\define}[1]{\textbf{#1}}

\newcommand{\boundary}{\partial}

\newcommand\tsim{\kern-.4em\sim}

\newcommand\ep{\epsilon}

\newcommand\ssm{\smallsetminus}

\renewcommand{\hom}{\mathrm{Hom}}

\newcommand{\ray}{\mathrm{ray}}
\newcommand{\ent}{\mathrm{ent}}

\renewcommand{\phi}{\varphi}
\renewcommand{\epsilon}{\varepsilon}

\DeclareMathOperator{\intr}{int}
\DeclareMathOperator{\brloc}{brloc}
\DeclareMathOperator{\coll}{coll}

\begin{document}
\title[Growth rates and the veering polynomial]{Flows, growth rates, and 
the veering polynomial}
\author[M.P. Landry]{Michael P. Landry}
\address{Department of Mathematics and Statistics\\
Washington University in Saint Louis }
\email{\href{mailto:mlandry@wustl.edu}{mlandry@wustl.edu}}
\author[Y.N. Minsky]{Yair N. Minsky}
\address{Department of Mathematics\\ 
Yale University}
\email{\href{mailto:yair.minsky@yale.edu}{yair.minsky@yale.edu}}
\author[S.J. Taylor]{Samuel J. Taylor}
\address{Department of Mathematics\\ 
Temple University}
\email{\href{mailto:samuel.taylor@temple.edu}{samuel.taylor@temple.edu}}
\date{\today}
\thanks{This work was partially supported by the NSF postdoctoral fellowship DMS-2013073, NSF grants DMS-1744551, DMS-2005328,
DMS-2102018, and the Sloan Foundation.}

\begin{abstract}
For a pseudo-Anosov flow $\phi$ without perfect fits on a closed $3$-manifold, 
Agol--Gu\'eritaud produce a veering triangulation $\tau$ on the manifold $M$ obtained by
deleting $\phi$'s singular orbits. We show that $\tau$ can be realized in $M$ so that its
2-skeleton is positively transverse to $\phi$, and that the combinatorially defined flow
graph $\Phi$ embedded in $M$ uniformly codes $\phi$'s orbits in a precise sense. Together with
these facts we use a modified version of the veering polynomial, previously introduced by
the authors, to compute the growth rates of $\phi$'s closed orbits after cutting $M$ along
certain transverse surfaces, thereby generalizing work of McMullen in the fibered setting. These results are new even in the case where the transverse surface represents a class in the boundary of a fibered cone of $M$.

Our work can be used to study the flow $\phi$ on the original closed manifold. Applications include counting growth rates of closed orbits after cutting along closed transverse surfaces, defining a continuous, convex entropy function on the `positive' cone in $H^1$ of the cut-open manifold, and answering a question of Leininger about the closure of the set of all stretch factors arising as monodromies within a single fibered cone of a $3$-manifold. 
This last application connects to the study of 
endperiodic automorphisms of infinite-type surfaces and the
growth rates of their periodic points.
\end{abstract}

\maketitle

\setcounter{tocdepth}{1}
\tableofcontents

%% !TEX root =veering_poly2.tex

\section{Introduction}

In this paper we address the following family of questions which relate dynamics to topology for a pseudo-Anosov flow
$\phi$ in a 3-manifold. Given a properly embedded surface $S$ which is positively transverse to
$\phi$, one can attempt to count orbits with respect to intersection number with
$S$. That is, one can consider the growth rate
$$
\mathrm{gr}_\phi(S) = \lim_{L\to \infty} \#\{\gamma:\gamma\cdot S \le L\}^{1/L}
$$
where $\gamma$ varies over closed orbits of $\phi$. If $S$ is a {\em cross section} of $\phi$
(that is, $S$ intersects every flow line) then $\phi$ is the suspension flow of a
fibration with fiber $S$, and  $\mathrm{gr}_\phi(S)$ is the {\em Teichm\"uller
  dilatation} of the monodromy map (its logarithm is the entropy). If $S$ is not a cross section then this growth rate is
$\infty$, but we can interrogate the finer structure of $\phi$ by considering
$\phi|S$, the flow restricted to the complement of $S$. Growth rates
of closed orbits in $\phi|S$ can be counted with respect to their intersection with
transverse surfaces in the complement of $S$, or more generally with respect to cohomology
classes positive on the closed orbits of $\phi|S$.

Our main tool for studying these questions is the {\em veering triangulation} of
Agol-Gu\'eritaud, which is a canonical ideal triangulation associated to a pseudo-Anosov
flow without perfect fits (see \Cref{sec:flow_setup} for details on this condition and the
Agol-Gu\'eritaud construction).  In previous work \cite{LMT20} we associated to such a
triangulation an invariant called the {\em veering polynomial}, and a transverse graph called
the {\em flow graph}. In this paper we will show that {\em the triangulation parameterizes
transverse surfaces}, {\em the flow graph gives an explicit coding for the flow}, and {\em
  the polynomial computes the growth rates}. 

In the case of a fibered manifold with pseudo-Anosov monodromy, the veering polynomial
recovers McMullen's Teichm\"uller polynomial, and the growth rates correspond to
Teichm\"uller dilatations in the fibered cone of Thurston's norm on homology. 
But even in this case we obtain some new information on the behavior of these dilatations --
see \Cref{th:Stretch} below.

\medskip

What arises from this, we hope, is evidence
that the veering triangulation is an effective
combinatorial tool for studying pseudo-Anosov flows, providing as it does an explicit coding which
is sensitive simultaneously to the dynamics of the flow and the topology of
the 3-manifold.

\subsection{Growth rates}
To summarize our results we introduce the terminology in more detail. Let $\ol M$ be a closed oriented $3$-manifold and let $\phi$ be a pseudo-Anosov flow on
$\ol M$ without perfect fits (see \Cref{sec:flow_setup}). 
We assume throughout that $\phi$ has at least one singular orbit; that is, $\phi$ is not Anosov.
Let $M$ denote $\ol M$ minus the singular orbits of $\phi$. 
Let $\tau$ be the veering
triangulation of $M$ dual to $\phi$ furnished by
the Agol--Gu\'eritaud  construction (\Cref{th:AG}).

The 2-skeleton $\tau^{(2)}$ has the structure of an oriented branched surface and we can consider
surfaces $S$ {\em carried} by it. For such a surface let $M|S$ denote $M$ cut along $S$,
and let $\phi|S$ denote the restricted flow in $M|S$, which is a {\em semiflow} in the sense of
Fenley--Mosher \cite{fenley2001quasigeodesic}. Assume for simplicity that $S$, and hence
$M|S$, is connected.

Let $\mc O_\phi$ denote the closed orbits of $\phi$ and 
$\mc O_\phi|S$ those closed orbits that avoid $S$. We say that a cohomology class $\xi\in
H^1(M|S)$ is \define{positive} if it is positive on orbits in $\mc O_\phi|S$ as well as on certain peripheral ``prong curves" corresponding to the removed singular orbits (see \Cref{sec:growth} for details).

The veering polynomial $V_\tau$ previously defined in \cite{LMT20} is an element of the group ring
$\ZZ[H_1(M;\ZZ)/\mathrm{torsion}]$. 
We will define an adapted polynomial $V_{\phi|S}$ in
$\ZZ[H_1(M|S;\ZZ)/\mathrm{torsion}]$, morally obtained by deleting certain terms from $V_\tau$
(see \Cref{subsec:polynomial counting} for the precise definition). A positive class $\xi\in H^1(M|S)$
gives rise to a {\em specialization} $V_{\phi|S}^\xi(u)$ in the sense of McMullen
(see \Cref{sec:poly}), 
which is a single variable polynomial-like expression.
Our main theorem
about growth rates is the following. 

\begin{thm}[Growth rates of closed orbits]
\label{th:Growth}
Let $S$ be a connected surface carried by $\tau^{(2)}$. Then for any positive class $\xi \in H^1(M|S)$, the growth rate
\begin{align} \label{eq:growth}
\mathrm{gr}_{\phi|S}(\xi) = \lim_{L\to \infty}  \# \{ \gamma \in \mc O_\phi|S : \xi(\gamma) \le L   \} ^{\frac{1}{L}}
\end{align}
exists and is equal to the reciprocal of the smallest positive root of the specialization $V_{\phi |S}^\xi(u)$ of the veering polynomial at $\xi$. 

Moreover, $\mathrm{gr}_{\phi|S}(\xi) >1$ if and only if there are infinitely many closed primitive orbits of $\phi$ that miss $S$.
\end{thm}
\noindent  See \Cref{th:entropy-veering-nonlayered} for the general statement, in
particular allowing disconnected $S$. 

\begin{remark} 
Using \Cref{cor:boundary_fibered} and code written by Parlak, Schleimer, and Segerman \cite{PSS}, Ross Griebenow has found explicit examples of surfaces $S$ carried by $\tau^{(2)}$ missing infinitely many closed primitive orbits of $\phi$ \cite{Griebenow}. Hence, $\mathrm{gr}_{\phi|S}(\xi) >1$ for such examples by \Cref{th:Growth}.

In fact, in the forthcoming paper \cite{LMTspA} we give a construction which shows that such examples are plentiful. Starting with a general type of endperiodic map on an infinite-type surface, we produce a surface $S$ in a fibered manifold $M$ with a pseudo-Anosov suspension flow $\phi$. The infinite-type surface determines a class $\xi\in H^1(M|S)$, and the growth rate $\mathrm{gr}_{\phi|S}(\xi)$ is the `stretch factor' of the original endperiodic map. See \Cref{rmk:endperiodic} for some details on the connection between growth rates and stretch factors. It is then easy to produce endperiodic maps so that the associated stretch factors are greater than 1.
\end{remark}

Let $\mc C^+$ be the cone in $H^1(M|S)$ consisting of positive classes. The associated \define{entropy function} is
\begin{align*}
\ent_{\phi|S} \colon &\mc C^+ \to [0, \infty)\\
			&\xi \mapsto \log(\mathrm{gr}_{\phi|S}(\xi)),
\end{align*}
where $\mathrm{gr}_{\phi|S}$ is given by \Cref{eq:growth}. 
The following result, which is a combination of \Cref{th:entropy2} and
\Cref{th:entropy3},
establishes the essential properties of the entropy function on the cone of positive classes.
In \Cref{sec:entropy}, we define what it means for the restricted
semiflow to be \emph{essentially transitive} and refer the reader there for details.

\begin{thm}[Entropy] \label{th:Entropy}
The entropy function $\ent_{\phi|S} \colon  \mc C^+ \longrightarrow [0,\infty)$ is
  continuous, convex, and homogeneous of degree $-1$.

Moreover, if the semiflow $\phi|S$ is essentially transitive, then $\ent_{\phi|S}$ is real analytic, strictly convex, and blows up at the boundary of $\mc C^+$.
\end{thm}

Throughout this discussion, we have focused on the manifold $M$. However, much of this theory extends to study transverse surfaces in the original closed manifold $\ol M$. See, for example, \Cref{th:growth_rates_closed} which is an analogue of \Cref{th:Growth} for transverse surfaces in $\ol M$.

\subsection{Transversality and coding}
Theorems \ref{th:Growth} and \ref{th:Entropy} rely on the following results which connect the flow to
the combinatorial structure of $\tau$ and its flow graph.

\begin{thm}[Transversality]\label{th:Transversal}
The veering triangulation $\tau$ dual to $\phi$ can be realized in $M$ so that the cooriented branched surface $\tau^{(2)}$ is positively transverse to the flow lines of $\phi$.
\end{thm}
\noindent  While this transversality is automatic in the setting of a suspension flow, the general case requires a surprisingly delicate argument. For a more detailed statement, 
see 
\Cref{thm:flow transverse}.

One important takeaway from \Cref{th:Transversal} is that surfaces that are carried by $\tau^{(2)}$, which are often in plentiful supply, are automatically transverse to the flow $\phi$. For example, by \Cref{th:cones}, any class in $H^1(M)$ that is nonnegative on closed positive transversals of $\tau^{(2)}$ is represented by a surface carried by $\tau^{(2)}$ and such classes form the entire cone over a face of the Thurston norm ball.

In \cite{LMT20}, we used the combinatorial structure of $\tau$ to define a directed graph $\Phi$, called the \emph{flow graph} of $\tau$, and an embedding $\iota \colon \Phi \to M$ which maps edges of $\Phi$ to arcs that are positively transverse to $\tau^{(2)}$. The next result (which is a summary of facts stated in \Cref{th:closed_orbits} and \Cref{prop:flowline_bound}) justifies the name \emph{flow graph} by establishing that $\Phi$ codes the orbits of $\phi$. 

\begin{thm}[Coding $\phi$ with $\Phi$]
\label{th:Coding}
The map $\iota \colon \Phi \to M$ establishes a correspondence between directed lines in $\Phi$ and flow lines in $\ol M$, which is surjective and uniformly bounded-to-one. 

Restricting this correspondence to closed directed cycles, we get a one-to-one correspondence with the exception of finitely many orbits and their positive multiples.
\end{thm}
\noindent In fact, we can say far more about the correspondence between closed directed cycles of $\Phi$ and closed orbits of $\phi$. See \Cref{th:closed_orbits} for the detailed statement. The upshot is that the explicit coding of the flow $\phi$ by the flow graph $\Phi$ allows us to address Theorems \ref{th:Growth} and \ref{th:Entropy} using tools from the study of growth rates of directed cycles of graphs, as in McMullen's work on the
clique polynomial \cite{mcmullen2015entropy}.

\subsection{Fibered faces and stretch factors}
Let us recall some of the theory developed for fibered manifolds by Thurston \cite{thurston1986norm}, Fried \cite{Fri79, fried1982geometry}, and McMullen \cite{mcmullen2000polynomial},
which motivates most of our results.

Thurston defined a norm on the vector space $H^1(\ol M;\R)$ of a 3-manifold whose unit ball
$B$ is a polyhedron, and which organizes the fibrations of $\ol M$ over the circle in the following
sense: Any integral class $\alpha\in H^1(\ol M;\ZZ)$ which is Poincar\'e dual to the fiber of a
fibration must appear in the cone $\R_+{\bf F}$  on an open top-dimensional face ${\bf F}$
of $B$,
and moreover all other integral points of this cone correspond to fibers as well (hence
${\bf F}$ is called a {\em fibered face}, and $\alpha$ a fibered class). 

Further, the suspension flows associated to the various fibers in 
the cone $\R_+{\bf F}$ agree, up to isotopy and reparametrization, and so we identify
them with a single \emph{circular} flow $\phi$. Here a flow is circular if it admits a
cross section and so is up to reparametrization a suspension flow.

The orbit growth rate $\mathrm{gr}_\phi(\alpha)$ defined above, 
can also be interpreted as the stretch factor, or Teichm\"uller dilatation, of the return
map of the flow to a fiber associated to $\alpha$. Its logarithm, the entropy of the
return map, extends to a 
function $h_\phi \colon \R_+\mathrm{int}({\bf F}) \to (0, \infty)$ that is continuous, convex, and blows up at the boundary of $\R_+{\bf F}$ \cite[Theorem E]{fried1982flow}. McMullen extends Fried's result by showing that $h_\phi$ is additionally real analytic and
strictly convex \cite[Corollary 5.4]{mcmullen2000polynomial}. To do so, he introduced a
new polynomial invariant, called the \emph{Teichm\"uller polynomial}, which both packages
growth rates of the flow and detects the fibered cone $\R_+{\bf F}$ in a precise sense.
Since McMullen's work, the Teichm\"uller polynomial has become a central tool in the study of these stretch factors; see e.g. \cite{leininger2013number, hironaka2010small, kin2013minimal, sun2015transcendental}.

The veering polynomial is a direct generalization of the Teichm\"uller polynomial, with
\Cref{th:Growth} extending McMullen's theorem on growth rates and \Cref{th:Entropy}
extending the theorem on the properties of $h_\phi$.

Now suppose $\xi$ is a fibered class, while $S$ is a surface transverse to $\phi$ which
is not a fiber. Then $\xi$ pulls back to a positive class in $H^1(M|S)$ in the sense of
\Cref{th:Growth}, and
$\mathrm{gr}_{\phi|S}(\xi)$ can be interpreted as both the growth rate with respect to
$\xi$ of closed orbits of
$\phi$ that miss $S$ (\Cref{cor:boundary_fibered}) as well as the stretch factor of an
endperiodic homeomorphism of the infinite type surface obtained by `spinning' the fiber
representatives of $\xi$ around $S$ (\Cref{rmk:endperiodic}). In fact, these quantities
all arise as {\em accumulation points} of the set of stretch factors of pseudo-Anosov return
maps to fibers in $\R_+{\bf F}$. 

To be more precise, let $\Lambda_{\bf F}\subset [1,\infty)$ be the set of stretch factors
  of monodromies  associated to fibers in $\R_+{\bf F}$ and let $\ol \Lambda_{\bf F}$ be
  its closure. Denote by $\ol \Lambda_{\bf F}'$ its derived set (i.e. set of limit points)
  and set $\ol \Lambda_{\bf F}^{n+1} = (\ol \Lambda_{\bf F}^{n}) '$. 
The following theorem answers a question of Chris Leininger (see \Cref{q:Lein}):

\begin{thm}[Stretch factors and fibered cones]\label{th:Stretch}
The stretch factor set $\ol \Lambda_{\bf F}$ is compact, well-ordered under $\ge$, and $\ol \Lambda_{\bf F}^n = \{1 \}$ for some $1\le n \le \mathrm{dim}(H^1(M;\R)$.
\end{thm}
A more detailed statement can be found in \Cref{th:structure_stretch}, including the
relation between limit points of $\Lambda_{\bf F}$ and growth rates of the form
$\mathrm{gr}_{\phi|S}(\xi)$.

\subsection{Connections to previous and ongoing work}
Although Agol and Gu\'eritaud's construction of a veering triangulation from a pseudo-Anosov flow without perfect fits is unpublished, there are many established connections between veering triangulations and the topology, geometry, and dynamics of their underlying manifolds. These include links to pseudo-Anosov stretch factors \cite{agol2011ideal}, angle structures \cite{hodgson2011veering, futer2013explicit}, hyperbolic geometry \cite{Gueritaud,hodgson2016non,futer2020random}, and the curve complex \cite{minsky2017fibered, strenner2018fibrations}.

More relevant to this paper is the work of Landry \cite{landry2018taut, landry2019stable,Landry_norm} which studies the surfaces carried by the underlying $2$-skeleton of the veering triangulation. This connects to our previous work \cite{LMT20} introducing the veering polynomial, relating it to the Teichm\"uller polynomial, and laying the combinatorial groundwork for what is done here (although we emphasize that this paper can be read independently of the previous). Also, Parlak has recently introduced and implemented algorithms to compute the veering polynomial and its relatives \cite{parlak2020computation} and demonstrated a connection with the Alexander polynomial \cite{parlak2021taut}, thereby generalizing work of McMullen on the Teichm\"uller polynomial \cite{mcmullen2000polynomial}.

Finally, the Agol--Gu\'eritaud  construction is expected to be reversible in the sense that a veering triangulation should determine a pseudo-Anosov flow and the process of going from one to the other should be inverse operations.
Proving this statement is an ongoing program of Schleimer--Segerman, the first part of which is \cite{schleimer2019veering} where from a veering triangulation a combinatorial `flow space' is reconstructed. There is also forthcoming work of Agol--Tsang \cite{AgolTsang} which produces a pseudo-Anosov flow from a veering triangulation, but without the claim that it is canonical or that it recovers the original flow if the veering triangulation was produced by the Agol--Gu\'eritaud  construction.

\subsection{Outline of paper}
In \Cref{sec:background} we review essential properties of veering triangulations as well as some basic structure we introduced in \cite{LMT20}. This is followed by \Cref{sec:dynamicplanes} which lays out one of our primary combinatorial tools, which we call \emph{dynamic planes}. 

Background on pseudo-Anosovs flows and the construction of Agol--Gu\'eritaud, which builds the dual veering triangulation, is presented in \Cref{sec:flow_setup}. In \Cref{sec:trans} we prove \Cref{th:Transversal} that the veering triangulation can be realized positively transverse to $\phi$, and \Cref{sec:flowandgraph} uses this transversality to prove \Cref{th:Coding} that the flow graph codes $\phi$'s orbits. \Cref{th:Growth} is then a consequence of these results along with connection between dynamic planes and $\phi$'s flow space, as established in \Cref{sec:growth}. 

In \Cref{sec:closed_case}, we prove a version of \Cref{th:Growth} that covers the case of closed surfaces transverse to the flow $\phi$ on the \emph{closed} manifold $\ol M$. In this section the veering triangulation only appears as a tool in the proof. Finally, in \Cref{sec:apps} we give several applications of our main theorems. These include \Cref{th:Entropy} and \Cref{th:Stretch}. 

\subsection*{Acknowledgements} 
We thank Chris Leininger
for illuminating  discussions on the topic
and for asking  \Cref{q:Lein}, Amie Wilkinson for helpful remarks related to
\Cref{sec:trans_flow}, and Chi Cheuk Tsang for comments on an earlier draft. 

%% !TEX root =veering_poly2.tex

\section{The flow graph, the veering polynomial, and carried surfaces}
\label{sec:background}

Here we record some required background and summarize 
results from our previous work \cite{LMT20}. Background on pseudo-Anosov
flows will be deferred until \Cref{sec:flow_setup}.

\subsection{Veering triangulations}
\label{sec:veering_basics}
A veering triangulation of a $3$-manifold $M$ is a taut ideal triangulation together with a coherent assigment of veers to its edges. We begin by explaining each of these terms.

A \define{taut ideal tetrahedron} is an ideal tetrahedron (i.e. a tetrahedron without vertices) along with a coorientation on each face so that it has two inward pointing faces, called its \emph{bottom faces}, and two outward pointing faces, called its \emph{top faces}.
Each of its edges is then assigned either angle $\pi$ or $0$ depending on whether
the coorientations on the adjacent faces agree or disagree, respectively.
 
Following Lackenby \cite{lackenby2000taut}, an ideal triangulation of $M$ is \define{taut} if each of its faces has been cooriented so that each ideal tetrahedron is taut and the angle sum around each edge is $2\pi$. 
The local structure around each edge $\bf{e}$ is as follows: $\bf{e}$ includes as a $\pi$-edge into two tetrahedra. For the other tetrahedra meeting $\bf{e}$, $\bf{e}$ includes as a $0$-edge and these tetrahedra 
are divided into the two sides, called \define{fans} of $\bf{e}$, each of which is linearly ordered by the coorientation on faces. The \textbf{length} of each fan is one less than the degree of $e$ on that side. See \Cref{fig:fans}.

\begin{figure}
\centering
\includegraphics{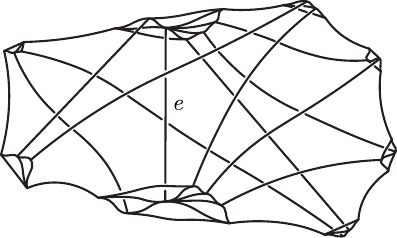}
\caption{The edge $e$ has one fan of length 1 and one fan of length 3.}
\label{fig:fans}
\end{figure}

A \define{veering triangulation} $\tau$ of $M$ is a taut ideal triangulation of $M$ in
which each edge has a consistent \define{veer}. This means that 
each edge is labeled to be either
\emph{right} or \emph{left} veering such that each tetrahedron of $\tau$ admits an
orientation preserving isomorphism to the model veering tetrahedron pictured in
\Cref{fig:veer_tet}, in which the veers of the 0-edges are specified: 
right veering edges have positive slope and left veering edges have negative slope. The
$\pi$-edges can veer either way, as long as adjacent tetrahedra satisfy the same rule. 
If the $\pi$-edges of a tetrahedron have opposite veer, the tetrahedron is said to be \define{hinge}; 
otherwise it is \define{non-hinge}.

\begin{figure}[h]
\begin{center}
\includegraphics{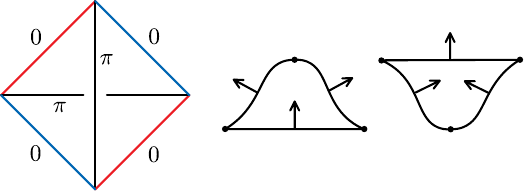}
\caption{A model veering tetrahedron and its cusps with their coorientations.}
\label{fig:veer_tet}
\end{center}
\end{figure}

\subsection{The dual graph, flow graph, and stable branched surface}
The stable branched surface $B^s$ in $M$ associated to the veering triangulation $\tau$, introduced in \cite{schleimer2019veering} as the \emph{upper branched surface in dual position} and in \cite[Section 4]{LMT20}, plays a central role throughout this paper. We refer the reader to \cite{floyd1984incompressible, oertel1984incompressible} for general facts about branched surfaces.

\begin{figure}[h]
\begin{center}
\includegraphics{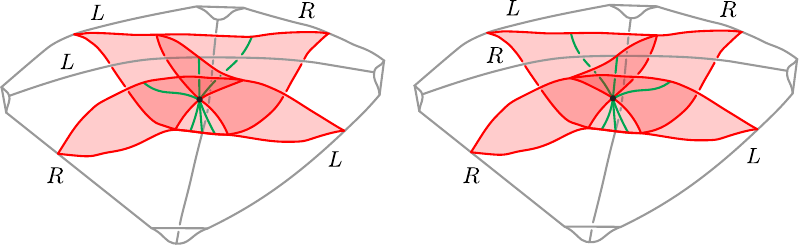}
\caption{The stable branched surface $B^s$ within a single tetrahedron. Its intersection with the flow graph, where edges are directed upward, is shown in green (note: the number of incoming edges at a vertex may vary).}
\label{fig:stablebs}
\end{center}
\end{figure}

Topologically, the \define{stable branched surface} $B^s$ is the dual complex of $\tau$ in $M$, and as such, it is a deformation retract of $M$. 
Note that $B^s$ is $2$ dimensional since $\tau$ has no vertices.
For each tetrahedron $t$ we define a smooth structure on $B^s_t = B^s \cap t$  as follows: if the top edge of $t$ is left veering, then we smooth according to the lefthand side of \Cref{fig:stablebs} and otherwise we smooth according the the righthand side. It is proven in \cite[Lemma 4.3]{LMT20} that this produces a well-defined global smooth structure making $B^s$ into a branched surface.

The stable branched surface contains two directed graphs related to $\tau$ that are also of central importance. The first, is the \define{dual graph} $\Gamma$ of $\tau$ which is defined to be the $1$-skeleton of $B^s$ whose edges are directed by the coorientation on the faces of $\tau$. Alternatively, $\Gamma$ is 
the graph with a vertex interior to each tetrahedron and a directed edge crossing each cooriented face from the vertex in the tetrahedron below the face to the vertex in the tetrahedron above the face. See \Cref{fig:stablebs}. The \emph{directed} cycles of $\Gamma$ are called \define{dual cycles} or $\Gamma$\define{-cycles}. Here and throughout, a directed cycle of a directed graph is an oriented loop determined by a cyclic concatenation of directed edges.

As the $1$-skeleton of the branched surface $B^s$, each turn in the graph $\Gamma$ is either \define{branching}, i.e. realized by a smooth arc in $B^s$, or else what we call \define{anti-branching} (or \define{AB}). In greater detail, a turn of $\Gamma$ is an ordered pair $(e_1,e_2)$ of directed $\Gamma$-edges so that the terminal vertex $v$ of $e_1$ equals the initial vertex of $e_2$. The turn is branching if the arc $e_1 \cup e_2$ is smooth as an arc in the singular locus of $B^s$ and is anti-branching (or AB) otherwise. 
A  directed path, ray, or cycle in $\Gamma$ that makes only branching turns is called a \define{branch} \define{path}, \define{ray}, or \define{cycle}, respectively. Similarly, a directed path, ray, or cycle in $\Gamma$ that makes only AB turns is called an \define{AB} \define{path}, \define{ray}, or \define{cycle}.
We note that since for each vertex of $\Gamma$ each incoming edge is part of exactly one branching turn and one AB turn, 
there are only finitely many branch and AB cycles in $\Gamma$. 

Branching and anti-branching turns of $\Gamma$ can be characterized solely in terms of the veering combinatorics (\cite[Lemma 4.5]{LMT20}) and from this we can deduce a few important properties of the sectors of $B^s$. 

Each sector $\sigma$ of $B^s$ is a topological disk pierced by a single $\tau$-edge, as in \Cref{fig:sectors}. The $\Gamma$-edges bounding $\sigma$ are oriented so that exactly one vertex is a source, which we call the \textbf{bottom} of $\sigma$, and one is a sink, which we call the \textbf{top} of $\sigma$. The top and bottom divide the boundary of $\sigma$ into two oriented $\Gamma$-paths called \textbf{sides}. Each side has at least two $\Gamma$-edges because the $\tau$-edge piercing $\sigma$ has a nonempty fan on each side.
According to the following lemma, which appears as \cite[Lemma 4.6]{LMT20}, if you remove the last edge in any side of any sector of $B^s$, the resulting path is a branch segment, and that the entire side is never a branch segment. See \Cref{fig:sectors}, where the AB turns appear as corners of the sector. We call these vertices the \define{corner} vertices of the sector.

\begin{figure}[h]
\begin{center}
\includegraphics{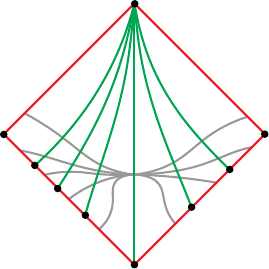}
\caption{A sector of $B^s$ and its intersection with $\Phi$ (green) and $\tau^{(2)}$ (gray). The triple points of $B^s$, which are the vertices of $\Gamma$ and $\Phi$, are in black. All edges are directed upward; the top is the northmost vertex, the bottom is the southmost, and the corners are westmost/eastmost.}
\label{fig:sectors}
\end{center}
\end{figure}

\begin{lemma}[Sectors and turns]\label{lem:sectors_turns}
Let $\sigma$ be a sector of $B^s$ and let $p$ be a side of $\sigma$ considered as a directed path in $\Gamma$ from the bottom to the top of $\sigma$.  The last turn of $p$ is anti-branching, and all other  turns are branching.
\end{lemma}

The second directed graph embedded in $B^s$ is the \define{flow graph} $\Phi$ of $\tau$, which was introduced in \cite[Section 4.3]{LMT20}. The vertices of $\Phi$  are in correspondence with $\tau$-edges, and for each tetrahedron $t$ of $\tau$, there are $\Phi$-edges from the bottom $\tau$-edge of each tetrahedron to its top $\tau$-edge and the two side $\tau$-edges whose veer is \emph{opposite} that of the top $\tau$-edge. 

This defines $\Phi$ as an abstract directed graph, but it also comes equipped with an embedding $\iota \colon \Phi \to B^s$, which was called \define{dual position} in \cite{LMT20}. 
Each $\tau$-edge $e$ is at the bottom of a unique tetrahedron $t_e$ and $\iota$ maps the vertex of $\Phi$ corresponding to $e$ to the vertex of $\Gamma$ contained in $t_e$. Each directed edge of $\Phi$ is then mapped into a single sector of $B^s$ so that it is positively transverse to $\tau^{(2)}$. See \Cref{fig:stablebs}. 
According to \cite[Lemma 4.7]{LMT20}, for each sector $\sigma$ of $B^s$ there is a directed edge of $\iota(\Phi) \cap \sigma$ coming into the top vertex of $\sigma$ from each vertex of $\sigma$ other than its two corner vertices. 
See \Cref{fig:sectors}. This characterizes the flow graph in dual position according to its intersection with each sector of $B^s$.

The \emph{directed} cycles of $\Phi$, along with their images under $\iota$, are called \define{flow cycles} or $\Phi$\define{-cycles}. When convenient, we sometimes identify $\Phi$ with its image under $\iota$.

\subsection{The veering polynomial}
\label{sec:poly}
Fix a finitely generated, free abelian group $G$ and denote its group ring with integer coefficients by $\Z[G]$.
Let $P \in \Z[G]$ and write $P= \sum_{g\in G} a_g  \cdot g$. The \define{support} of $P$ is
\[
\mathrm{supp}(P) = \{g \in G : a_g \neq 0 \}.
\]

For $P \in \Z[G]$ with $P = \sum_{g\in G} a_g \cdot g$  and $\alpha \in \hom(G,\R)$
the \define{specialization of $P$ at $\alpha$} is the single variable expression $P^\alpha$ in  $\Z[u^r : r \in \RR]$ given by
\[
P^\alpha(u) = \sum_{g\in G} a_g \cdot u^{\alpha(g)}.
\]

These generalities will be used in the specific setting of veering polynomials. For this, 
 let $M$ be a $3$-manifold with veering triangulation $\tau$, and set 
$G = H_1(M;\ZZ)/\mathrm{torsion}$. 
In \cite[Section 2]{LMT20}, we defined a polynomial invariant $V_\tau \in \ZZ[G]$, called the \define{veering polynomial} of $\tau$. Here, we recall an alternative characterization of $V_\tau$ in terms of the Perron polynomial of the flow graph $\Phi$. We refer the reader to \cite[Section 4]{LMT20} for additional details.

For a directed graph $D$, let $A$ denote the matrix with entries
\begin{align} \label{eq:adj}
A_{ab} = \sum_{\partial e=b-a} e,
\end{align}
where the sum is over all edges $e$ from the vertex $a$ to the vertex $b$. 
We call $A$ the \define{adjacency matrix} for $D$.
The \define{Perron polynomial} of $D$ is defined to be $P_D = \det(I - A)$. By definition this is an element of  $\ZZ[C_1(D)]$, where $C_1(D)$ is the group of simplicial 1-chains in $D$.
 
Following McMullen \cite{mcmullen2015entropy}, we define the \define{cycle complex} $\C(D)$ of $D$ to be the graph whose vertices are directed simple cycles of $D$ and whose edges correspond to disjoint cycles. We recall that $P_D$ equals the \define{clique polynomial} of $\C(D)$, which in particular shows that $P_D$ is an element of the subring 
$\Z[H_1(D)]$ (see \cite[Theorem 1.4 and Section 3]{mcmullen2015entropy}).  
Here, the clique polynomial associated to $\C(D)$ is
\begin{equation}\label{eq:cliquepoly}
P_D = 1 + \sum_{C} (-1)^{|C|} C \in \Z[H_1(D)],
\end{equation}
where the sum is over nonempty cliques $C$ of the graph $\C(D)$, i.e. over simple \emph{multicycles} of $D$, and  $|C|$ is the number of vertices of $C$, i.e. the number of components of the multicycle. 
Note that the support of $P$ is the set $\mathrm{supp}(P_D) = \{ C\} \subset H_1(D)$ of directed simple multicycles appearing in the expression (\ref{eq:cliquepoly}).

Now let $\iota \colon \Phi \to M$ be 
the flow graph with its embedding into $M$.
This induces a ring homomorphism $\iota_*\colon \ZZ[H_1(\Phi)] \to \ZZ[G]$ and we set
\begin{align*}
V_\tau = \iota_*(P_\Phi),
\end{align*}
where $P_\Phi$ is the Perron polynomial of $\Phi$. According to \cite[Theorem 4.8]{LMT20} this agrees with the original definition of the veering polynomial.

\subsection{Surfaces carried by $\tau$ and cones in (co)homology}
\label{sec:cones}

As noted by Lackenby \cite{lackenby2000taut}, tautness of $\tau$ naturally gives its $2$-skeleton $\tau^{(2)}$ the structure of a transversely oriented branched surface in $M$. 
The smooth structure on $\tau^{(2)}$ can be obtained by, within each tetrahedron, smoothing along the $\pi$-edges and pinching along the $0$-edges, thus giving $\tau^{(2)}$ a well-defined tangent plane field at each of its points.

As a transversely oriented
branched surface, $\tau^{(2)}$ can carry surfaces similarly to the way a train track on a surface can carry curves.  We let $\cone_2(\tau)$ be the closed cone in $H_2(M,\partial M)$ positively generated by classes that are represented by the surfaces that $\tau$ carries.
We call $\cone_2(\tau)$ the \define{cone of carried classes}.

In a bit more detail, the branched surface $\tau^{(2)}$ has a branched surface fibered neighborhood $N= N(\tau^{(2)})$ foliated by intervals such that collapsing $N$ along its $I$-fibers recovers $\tau^{(2)}$. The transverse orientation on the faces of $\tau$ orients the fibers of $N$, and a properly embedded oriented surface $S$ in $M$ is \define{carried} by $\tau^{(2)}$ if it 
is contained in $N$ where it is positively transverse to its $I$-fibers. 
We also say that $S$ is carried by $\tau$.

A carried surface $S$
embedded in $N$ transverse to the fibers 
defines a nonnegative integral weight on each face of $\tau$ given by the number of times the $I$-fibers over that face intersect $S$. These weights satisfy the \define{matching (or switch) conditions} stating that the sum of weights on one side of a edge match the sum of weights on the other side. Conversely, a collection of nonnegative integral weights satisfying the matching conditions gives rise to a surface embedded in $N$ transverse to the fibers in the usual way. 
More generally, any collection of nonnegative weights on faces of $\tau$ satisfying the matching conditions defines a nonnegative relative cycle giving an element of $H_2(M, \partial M ; \RR)$ and we say that a class is \define{carried} by $\tau^{(2)}$ if it can be realized by such a nonnegative cycle. Hence, $\cone_2(\tau)$ is precisely the subset of $H_2(M,\partial M)$ consisting of carried classes.

The following theorem is a summary of results in \cite[Theorem 5.1 and Theorem 5.12]{LMT20}. For its statement, we let $\cone_1(\Gamma) \subset H_1(M; \R)$ denote the cone positively spanned by the direct cycles of the dual graph $\Gamma$. We call $\cone_1(\Gamma)$ the \textbf{cone of homology directions} of $\tau$ and note that it is equal to the cone positively generated by all closed curves which are positively transverse to $\tau^{(2)}$ at each point of intersection. We write $\cone_1^\vee(\Gamma)$ for its dual cone in $H^1(M ; \R)$, which consists of classes that are nonnegative on all dual cycles.

\begin{theorem}[Cones and Thurston norm]
\label{th:cones}
For any veering triangulation $\tau$ of $M$:
\begin{enumerate}
\item The cone of homology directions $\cone_1(\Gamma)$ is positively generated by $\iota(\mathrm{supp}(P_\Phi))$, the image of the support of $P_\Phi$.
\item After identifying $H^1(M;\R) = H_2(M, \partial M; \R)$, $\cone_2(\tau) = \cone_1^\vee(\Gamma)$.
\item There is a cone $\RR_+{\bf F}_\tau$ over a (possibly empty) face ${\bf F}_\tau$ of the Thurston norm ball in $H_2(M, \partial M)$ such that $\cone_2(\tau) = \RR_+{\bf F}_\tau$.
\end{enumerate}
\end{theorem}

So, for example, a class $\alpha \in H_2(M,\partial M)$ is carried by $\tau$ if and only if $\langle \alpha, \iota(c) \rangle \ge 0$ for each simple directed cycles $c$ of $\Phi$.

%%%%%%%%%%%%%%%%%%%
%% !TEX root =veering_poly2.tex

\section{Dynamic planes and flow cycles}
\label{sec:dynamicplanes}

In this section, we introduce and develop the essential features of dynamic planes of the veering triangulation $\tau$. 
A dynamic plane is a combinatorial version of a leaf of the weak stable foliation of a pseudo-Anosov flow but with additional structure coming from its interaction with the dual and flow graphs of $\tau$.
The main results are \Cref{prop:enough_flow}, which says that all but finitely many dual cycles (and their multiples) are homotopic to flow cycles, and \Cref{lem:trans_hom}, which combinatorially characterizes when dual cycles are homotopic within the quotient of a dynamic plane. Both these technical facts will be essential in \Cref{sec:flowandgraph} where we describe precisely how the flow graph codes the orbits of the dual flow.

\subsection{Descending sets and dynamic planes}
\label{sec:width}
For any branched surface $B$, let $N(B)$ denote a regular neighborhood of $B$ foliated in the standard way by intervals. Let 
\[
\coll\colon N(B)\onto B
\]
be the map which collapses all the intervals.

\begin{figure}[h]
\centering
\includegraphics{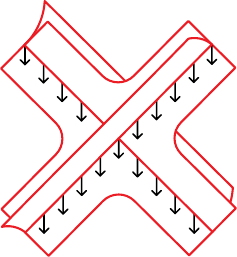}
\caption{The maw vector field}
\label{fig_maw}
\end{figure}

If $B$ is a branched surface with generic branching then we denote its branch locus, i.e. its collection of nonmanifold points, by $\brloc (B)$. The \textbf{maw vector field} is a vector field tangent to $B$ defined on $\brloc (B)$ that always points from the 2-sheeted side to the 1-sheeted side. Note that the maw vector field is defined even at triple points; see \Cref{fig_maw}.

A \textbf{descending path} in $B$ is an oriented immersed curve in $B$ whose tangent vector at each point of intersection with $\brloc(B)$ is equal to the maw vector field at that point.
 
We next consider the stable branched surface $B^s$.  
Note that up to homotopy any closed descending path in $B^s$ is negatively transverse to $\hbs$ (see \Cref{fig:sectors}), and is therefore homotopically nontrivial in $M$ (\cite[Theorem 3.2]{schleimer2020essential}).
Let $\wt B^s$, $\wt \Gamma$, and $\wt \Phi$ be the preimages of $B^s$, $\Gamma$, and $\Phi$, respectively, in the universal cover $\wt M$ of $M$.

Let $\sigma$ be a sector of the branched surface $\wt B^s$. The \textbf{descending set} of $\sigma$, denoted $\Delta(\sigma)$, is defined to be the union of all sectors $\sigma'$ of $\wt B^s$ such that there exists a descending path from $\sigma$ to $\sigma'$. 
Before describing $\Delta(\sigma)$ in detail, recall that by a path, ray, or line in $\Gamma$ or $\Phi$ we always mean a directed path, ray, or line. If $\ell$ is a branch line in $\wt \Gamma$ through a vertex $v$, then the \define{negative subray} of $\ell$ at $v$ is the portion of the branch line $\ell$  that lies below $v$.

\begin{figure}[h]
\centering
\includegraphics{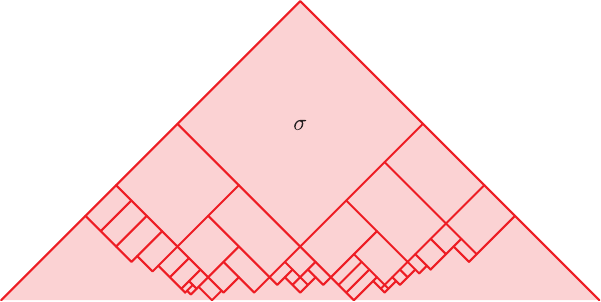}
\caption{The descending set of a sector $\sigma$, and part of its intersection with $\wt\Gamma$.}
\label{fig_descset}
\end{figure}

\begin{lemma}[Structure of $\Delta(\sigma)$]\label{lem_dstructure}
Let $\sigma$ be a $\wt B^s$-sector. 
\begin{enumerate}[label=(\alph*)]
\item The descending set $\Delta(\sigma)$ is diffeomorphic to a closed quarter plane bounded by 
the negative subrays of the two branch lines passing through the top vertex $v$ of $\sigma$.

\item If $w$ is a $\wt \Gamma$-vertex contained in $\Delta(\sigma)$, then any $\wt\Gamma$-ray starting at $w$ intersects $\partial \Delta(\sigma)$.

\item If $w$ is a $\wt \Gamma$-vertex contained in $\intr(\Delta(\sigma))$, there is a unique outgoing $\wt \Phi$-edge incident to $w$ contained in $\Delta(\sigma)$. 
The unique $\wt \Phi$-ray starting at $w$ lying in $\Delta(\sigma)$ terminates on $\partial \Delta(\sigma)$.

\end{enumerate}
\end{lemma}

Before the proof, let us establish a few facts that we will need.
By \cite[Theorem 8.1]{schleimer2019veering}, the branched surface $B^s$ fully carries a unique 2-dimensional lamination $\mc L^s$ without parallel leaves such that $\mc L$ is essential and each leaf of $\mc L^s$ is either a plane, an $\pi_1$-injective annulus, or a $\pi_1$-injective M\"obius band. 
Denote by $\wt {\mc L}^s$ the lamination lifted to the universal cover $\wt M$ whose leaves are planes. Note that since $\wt{\mc L}^s$ is carried by $\wt B^s$, each leaf inherits a tesselation corresponding to the sectors of $\wt B^s$ it traverses.

It is clear from the branching structure of $B^s$ (c.f. \cite[Remark 8.27]{schleimer2019veering}) that if $\ell$ is a leaf carried by $\wt B^s$ such that $\coll(\ell)$ contains $\sigma$, and if there is a descending path from $\sigma$ to another sector $\sigma'$, then $\sigma'$ is also contained in $\coll(\ell)$. Consequently, we have that $\coll(\ell)$ contains the descending set of every sector traversed by $\ell$.

We also observe that each leaf of $\wt{\mc L}^s$ traverses a sector of $\wt B^s$ at most once. 
For if $\ell$ is a leaf traversing a sector $\sigma$ twice, a short segment contained in a regular neighborhood of $\sigma$ connecting two points of $\ell$ identified under $\coll$ may be homotoped to lie entirely in $\ell$. Since the branched surface $B^s$ is laminar (as observed in \cite{schleimer2019veering}) and hence essential, this contradicts  \cite[Theorem 1.d]{GO89} (see also \cite[Lemma 2.7]{GO89}).
We conclude that for any leaf $\ell$ of $\wt{\mc L}^s$, $\coll(\ell)$ is a plane embedded in $\wt B^s$.

\begin{proof}[Proof of \Cref{lem_dstructure}]
We begin by using the above discussion to prove part $(a)$. Let $\ell$ be any leaf of $\wt {\mc L}^s$ that traverses $\sigma$. Then $P = \coll(\ell)$ is a plane tessellated by sectors of $\wt B^s$ that contains the descending set $\Delta(\sigma)$. From the local picture of $P$ around vertices of $\wt \Gamma$ shown in \Cref{fig_vertices}, we see that for each vertex $w$ of $P$, $P$ contains the negative subrays of both branch lines through $w$.
So if $v$ is the vertex at the top of $\sigma$, then the branch lines through $v$ are proper lines contained in $P$ and determine a quarter plane $Q$ as in the statement of $(a)$. Hence, it suffices to show that $Q = \Delta(\sigma)$. 

Clearly, $\Delta(\sigma) \subset Q$ since no descending paths starting at $\sigma$ can cross the branch lines through $v$. 

For the reverse containment, let $S_n$ denote the set of $\wt B^s$-sectors reachable from $\sigma$ by a descending path traversing at most $n$ sectors. Then $\sigma =  S_1\subset S_2\subset S_3\subset\cdots$ is an exhaustion of $\Delta(\sigma)$. 

\begin{claim}\label{claim_bdy}
If $w$ is a vertex in the boundary of $S_n$ then either $w$ is in the interior of $S_{n+1}$, or it lies on one of the two branch lines through $v$ and hence on the boundary of $Q$.
\end{claim}

\begin{proof}[Proof of claim]
The proof is by induction with the case of $S_1 = \sigma$ being by inspection (see \Cref{fig_descset}). 

Now suppose that $w$ is in the boundary of both $S_n$ and $S_{n+1}$. By the inductive hypothesis, we may assume that $w$ is a vertex of a sector $\sigma' \subset S_n \ssm S_{n-1}$.  In particular, $w$ is not the top vertex of $\sigma'$. If $w$ is not joined by a  $\wt\Gamma$-edge to the top of $\sigma'$, then again it is clear from the picture (\Cref{fig_descset}) that $w$ is in the interior of $S_{n+1}$ contradicting our assumption.

Otherwise, $w$ is joined by an edge $e$ to the top vertex $w'$ of $\sigma'$ and we say that $w$ is one of the two side vertices of $\sigma'$. 
Note that $e$ is in the boundary of $S_n$, since otherwise we would again have that $w$ is in the interior of $S_{n+1}$.

It suffices to show that $e$ is an edge of a branch line through the vertex $v$. 
Note that $w'$ is a vertex of some sector $\sigma''$ in $S_{n-1}$ since any descending path from $\sigma$ to $\sigma'$ passes through one of the two top edges of $\sigma'$. Since $w$ is not in the interior of $S_{n+1}$, we must have that 
$\sigma'$ is attached to $\sigma''$ along the edge $e'$, where $e'$
 is the $\wt\Gamma$-edge at the top of $\sigma'$ that is not $e$.
By the induction hypothesis, either $w'$ is in the interior of $S_{n}$ or $w'$ is contained in a branch line through $v$. But if $w'$ is in the interior of $S_n$, then $e$ is also in the interior of $S_n$, a contradiction. Hence, we must have that $w'$ lies along a branch line though $v$. Since $e'$ is in the interior of $S_n$, this branch line continues along $e$, establishing that it contains $w$. This completes the proof of \Cref{claim_bdy}.
\end{proof}

We conclude that $\Delta(\sigma)$ is a subcomplex of the quarter plane $Q$ (with its locally finite tessellation by sectors) and that $\partial \Delta(\sigma) = \partial Q$. It follows easily that $Q =  \Delta(\sigma)$ as required. 

For part $(b)$, again let $S_n$ denote the set of $\wt B^s$-sectors reachable from $\sigma$ by a descending path traversing at most $n$ sectors. 
If $w$ is a vertex of $S_n$, then any $\wt \Gamma$-path in $\Delta(\sigma)$ remains within $S_n$. Since $S_n$ has finitely many vertices and $\wt \Gamma$-rays are simple, each $\wt \Gamma$-ray starting at $w$ eventually meets $\partial \Delta(\sigma)$. 
This proves part $(b)$.

Considering a picture makes the first claim of part $(c)$ clear; see \Cref{fig_vertices}.
\begin{figure}[h]
\centering
\includegraphics{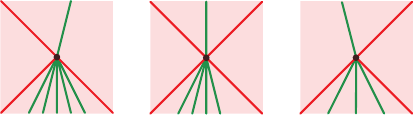}
\caption{Local pictures of vertices in $Q$ and incident $\wt\Gamma$-edges (red) and $\wt \Phi$-edges (green) in the proof of \Cref{lem_dstructure}. All edges are directed upward. Note that there are always incoming $\wt\Phi$-edges (the number may vary) and a unique outgoing $\wt\Phi$-edge.}
\label{fig_vertices}
\end{figure}

The same argument as the one for part $(b)$ shows that the $\wt\Phi$-ray starting from any point in $\Delta(\sigma)$ must meet $\del \Delta(\sigma)$.
This completes the proof of \Cref{lem_dstructure}.
\end{proof}

Next we describe and analyze a canonical set associated to a $\wt \Gamma$--ray. For a vertex $v$ or directed edge $e$ of $\wt \Gamma$, we set $\sigma(v)$ and $\sigma(e)$ to be the sector into which the maw vector field points at $v$ or along the interior of $e$, respectively. So if $v$ is the terminal vertex of the edge $e$ and $\sigma$ is the unique sector of $\wt B^s$ whose top vertex is $v$, then $\sigma(v)= \sigma(e) = \sigma$.

\begin{lemma}\label{lem:pushdown}
Suppose there exists a directed path in $\wt \Gamma$ from $u$ to $v$. Then $\Delta(\sigma(u)) \subset \Delta(\sigma(v))$.
\end{lemma}

\begin{figure}[h]
\centering
\includegraphics{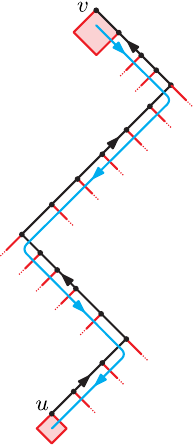}
\caption{If there is a path in $\wt\Gamma$ from $u$ to $v$ (shown in black), then there is a descending path in $\wt B^s$ from $\sigma(v)$ to $\sigma(u)$ (shown in blue).}
\label{fig_nested}
\end{figure}

\begin{proof}
The lemma follows by induction on the length of the path from $u$ to $v$. The inductive step is immediate from the observation that if $e$ is an edge of $\wt \Gamma$ with initial vertex $u$ and terminal vertex $v$, then there is a descending path from $\sigma(v) = \sigma(e)$ to $\sigma(u)$. Hence, $\Delta(\sigma(u)) \subset \Delta(\sigma(v))$.
\end{proof}

\begin{remark}
The descending path that the proof of \Cref{lem:pushdown} produces can be obtained by pushing
 the dual path from $u$ to $v$ slightly in the direction of the maw vector field and reversing orientation, as shown in \Cref{fig_nested}.
\end{remark}

Let $\gamma$ be a $\wt \Gamma$--ray, and set 
\[
D(\gamma)=\bigcup_{v \in\gamma}\Delta(\sigma(v)),
\]
where the unions are taken over all $\wt \Gamma$--vertices $v$ traversed by $\gamma$.
By \Cref{lem:pushdown}, 
it follows that $D(\gamma)$ is a nested union of quarter planes. If $\gamma$ is not eventually a branch ray of $\wt \Gamma$, then $D(\gamma)$ is evidently diffeomorphic to $\R^2$ and we say that $D(\gamma)$ is the \textbf{dynamic plane} associated to $\gamma$. By construction, $D(\gamma)$ is properly embedded in $\wt B^s$.
If $\gamma$ is eventually a branch ray then $D(\gamma)$ is diffeomorphic to a half plane and we say that $D(\gamma)$ is the \textbf{dynamic half plane} associated to $\gamma$. 

We remark that any dynamic plane $D$ is tessellated by sectors of $\wt B^s$ and hence it makes sense to speak of $\wt \Gamma$-paths or $\wt \Phi$-paths in $D$. Moreover, $\wt \Gamma \cup \wt \Phi$ determine a triangulation of $D$.

\begin{remark}\label{rmk:alt_dyn}
Since $\sigma(v) = \sigma(e)$ for an edge $e$ of $\wt \Gamma$ with terminal vertex $v$, we also have
\[
D(\gamma) = \bigcup_{e \in\gamma}\Delta(\sigma(e)),
\]
where the union is over edges traversed by $\gamma$. 
\end{remark}

\begin{proposition}[Basics of dynamic planes]\label{prop:dynamic_plane}
Let $D$ be a dynamic plane.
\begin{enumerate}[label=(\alph*)]
\item For any edge $e$ of $D$, $\Delta(\sigma(e)) \subset D$.
\item If $\gamma$ is any $\wt \Gamma$-ray contained in $D$ that is not eventually a branch ray, then $D = D(\gamma)$. 
\item The stabilizer of $D$ is either infinite cyclic or trivial. 
\end{enumerate}
\end{proposition}

\begin{proof}
First note that since $D$ is a plane properly embedded in $\wt B^s$, if $D$ contains $e$, then $D$ contains $\sigma(e)$. This follows since $\sigma(e)$ is the sector on the $1$-sheeted side of $e$. Next, suppose that $D = D(\psi)$ for some $\wt \Gamma$-ray $\psi$ that is not eventually a branch ray. If $\sigma \subset D(\psi)$, then directly from the definitions we have that $\Delta(\sigma) \subset D(\psi)$. Taken together, these two facts prove $(a)$.

For $(b)$, note that $(a)$ implies that $D(\gamma) \subset D$ (see \Cref{rmk:alt_dyn}). 
Since these are each planes properly embedded in $\wt B^s$, equality also holds.

For $(c)$, we appeal to the discussion preceding the proof of \Cref{lem_dstructure}. Since the dynamic plane $D$ is carried by $B^s$ it determines a unique leaf $\wt \ell$ of $\wt {\mc L}^s$ such that $D = \coll (\wt \ell)$. Hence, the stabilizer of $D$ is equal to the stabilizer of $\wt \ell$ and the claim follows from the fact that the image of $\wt \ell$ in $M$ is either a plane, annulus, or M\"obius band.
\end{proof}

\begin{figure}[h]
\centering
\includegraphics{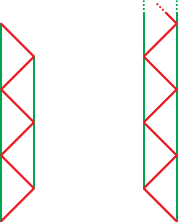}
\caption{An AB strip (left) and an infinite $AB$ strip (right), with $\wt \Gamma$-edges in red and $\wt \Phi$-edges in green. Edges are directed upward.}
\label{fig_abstrip}
\end{figure}

Recall that if $\gamma$ is a $\wt \Gamma$-path or ray contained in $D$ which makes only AB turns, we say $\gamma$ is an \textbf{AB path} or \textbf{AB ray}. Then each two consecutive $\wt \Gamma$-edges of $\gamma$ determine a triangle in the triangulation of $D$ by edges of $\wt \Phi$ and $\wt \Gamma$, and the third edge of this triangle is a $\wt \Phi$-edge. The union of all these triangles is a subset $S$ of $D$ diffeomorphic to $[0,1]\times [0,1]$ or $[0,1]\times [0,\infty)$ called an \textbf{AB strip} or \textbf{infinite AB strip}, respectively. See \Cref{fig_abstrip}.

The following lemma essentially says that $\wt \Phi$-rays in a dynamic plane either converge or are separated by AB strips.

\begin{lemma}[Dynamics of dynamic planes]
\label{lem_rayconvergence}
Let $D$ be a dynamic plane.
If $\alpha$ and $\beta$ are $\wt\Phi$-rays contained in $D$, then either $\alpha$ and $\beta$ eventually coincide or both eventually lie on the boundaries of infinite AB strips.
\end{lemma}

\begin{figure}[h]
\centering
\includegraphics{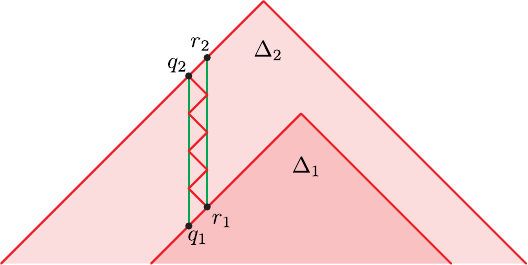}
\caption{AB strips are precisely the obstruction to contraction under ``flowing" forward in a dynamic plane.}
\label{fig_rayconvergence}
\end{figure}

\begin{proof}
Suppose that $D = D(\gamma)$, where $\gamma$ is a $\wt \Gamma$-ray which is not eventually a branch ray, and
let $p_1,p_2,\dots$ be the sequence of vertices where $\gamma$ makes AB turns. Define $\Delta_i=\Delta(\sigma(p_i))$, so that $\Delta_1\subset\Delta_2\subset \Delta_3\subset\cdots$ is the exhaustion of $D(\gamma)$ by descending sets.
Let $a$ and $b$ be vertices of $\alpha$ and $\beta$, respectively.
By truncating and reindexing the exhaustion $\{\Delta_i\}$, we can assume that $a$ and $b$ lie in $\Delta_1$. Hence by \Cref{lem_dstructure}(b) there exist vertices $a_1$ of $\alpha$ and  $b_1$ of  $\beta$ lying on the boundary of $\Delta_1$, and $a_1$ and $b_1$ are a finite distance apart in the combinatorial metric on $\partial \Delta_1$.

Let $q_1,r_1\in \partial \Delta_1$ be vertices that are connected by an edge from $q_1$ to $r_1$ with $\wt \Phi$-rays intersecting $\partial \Delta_2$ in $q_2$ and $r_2$ respectively. We claim that $
q_2\ne r_2$ if and only if the corresponding ray segments cobound an AB-strip (see \Cref{fig_rayconvergence}). Indeed, if $s$ is the sector in $D$ above the $\wt \Gamma$-edge connecting $q_1$ and $r_1$, then the $\wt \Phi$-rays from $q_1$ and $r_1$ immediately converge unless $r_1$ is a corner vertex of $s$ as in \Cref{fig:collide}. Applying this analysis repeatedly proves the claim.
\begin{figure}[h]
\begin{center}
\includegraphics{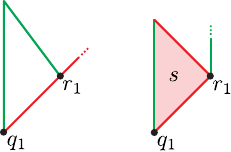}
\caption{The flow rays from $q_1$ and $r_1$ immediately collide (left) unless $r_1$ is a corner vertex of the sector $s$ above $q_1$ (right).}
\label{fig:collide}
\end{center}
\end{figure}
In other words, ``flowing" forwards in $D$ weakly contracts distance in $\partial \Delta_i$, with equality if and only if the flow segments are separated by a union of AB strips.
This implies that either the rays from $a_1$ and $b_1$ eventually coincide, or they both eventually meet AB rays.
\end{proof}

We next require a basic lemma about the structure of $\wt B^s$. Recall that each $\wt B^s$-sector has two \textbf{sides}, a \textbf{top vertex} and \textbf{bottom vertex}, and two \textbf{corner vertices}. 
Each side of a sector is composed of two branch segments, one which begins at the bottom vertex and terminates at a corner vertex, and one which consists of precisely one $\wt \Gamma$-edge which begins at the corner vertex and terminates at the top vertex. Each of these vertices corresponds to a triple point of $\wt B^s$, and has a right or left veer as shown in \Cref{fig_branchtype}. This veer agrees with the veer of the edge atop the unique tetrahedron containing the triple point (compare with \Cref{fig:stablebs}).

\begin{figure}[h]
\centering
\includegraphics{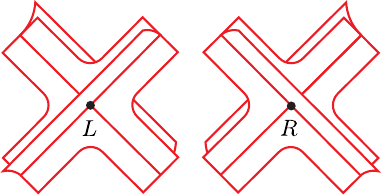}
\caption{Each triple point in $B^s$ comes with a veer.}
\label{fig_branchtype}
\end{figure}

In this language we have the following lemma, which is a reformulation of \cite[Fact 1]{LMT20}. For an illustration of the behavior described in the lemma see \Cref{fig_sectorveer}.

\begin{lemma}\label{lem:sectorveer}
Let $A$ be a $B^s$-sector. The bottom vertex and two corner vertices have identical veer. All other vertices in $\partial A$, except possibly the top vertex, have the opposite veer.
\end{lemma}

\begin{figure}[h]
\centering
\includegraphics{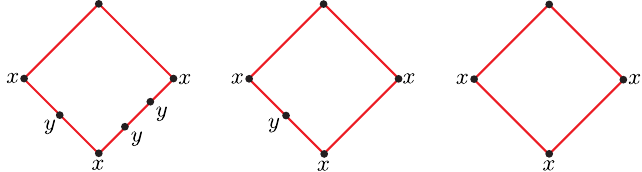}
\caption{Some situations allowed by \Cref{lem:sectorveer}, where $x$ and $y$ denote opposite veers. Note that the veers of top vertices are not governed by the lemma.}
\label{fig_sectorveer}
\end{figure}

Recall from above 
that an infinite AB strip is a subset of a dynamic plane homeomorphic to $[0,1]\times [0,\infty)$, determined by an AB ray. The subsets of an infinite AB strip corresponding to $\{0,1\}\times [0,\infty)$ are $\wt \Phi$-rays. Similarly, we define a \textbf{bi-infinite AB strip} to be a subset of a dynamic plane homeomorphic to $[0,1]\times \R$ determined by an AB \emph{line}. The boundary components of a bi-infinite AB strip are $\wt \Phi$-lines.

\begin{lemma}\label{lem:ABstrips}
Let $D$ be a dynamic plane. The following are equivalent:
\begin{enumerate}[label=(\roman*)]
\item $D$ contains an infinite AB strip
\item $D$ contains a bi-infinite AB strip
\item $D$ contains the lift of an AB-cycle.
\end{enumerate}
\end{lemma}
\begin{proof}
For the equivalence of conditions (i)-(ii), note that 
every vertex in $D$ has a unique backward AB ray (\Cref{fig_vertices}), so
every infinite AB strip is part of a bi-infinite AB strip.
Since every AB line covers an AB-cycle, the other implications are clear.
\end{proof}

Recall that each $\wt\tau$-edge $e$ has two \textbf{fans}, consisting of the tetrahedra for which $e$ is a $0$-edge lying on a particular side of $e$. The length of a fan in $\wt \tau$ is the number of tetrahedra it contains. 
We can also define fans and their lengths for edges of $\tau$ by lifting to $\wt \tau$. 

We denote the length of the longest fan in $\tau$ by $\delta_\tau$.

\begin{proposition}[The AB region]\label{lem:ABregion}
Suppose that $D$ contains a bi-infinite AB strip. Then 
\begin{enumerate}[label=(\arabic*)]
\item the number of bi-infinite AB strips in $D$ is less than $\delta_\tau$, and 
\item the union $D_{\text{AB}}$ of all bi-infinite AB strips in $D$ is diffeomorphic to $[0,1]\times \R$.
\end{enumerate} 
\end{proposition}
The union $D_{\text{AB}}$ of all bi-infinite AB strips in $D$ as in \Cref{lem:ABregion} will be called the \textbf{AB region} of $D$.

\begin{proof}
By \Cref{lem_rayconvergence}, if $D$ contains $n$ or more infinite AB strips then there is a subset $S$ of $D$ that is obtained by gluing $n$ infinite AB strips along their $[0,\infty)$ boundaries. See \Cref{fig:ABregion}. By \Cref{lem:sectorveer} all the vertices in the interior of $S$ have identical veer, so if $v$ is a vertex in the interior of $S$, then $v$ lives in a tetrahedron whose top and bottom edges have the same veer. In other words, every vertex in the interior of $S$ lies in a non-hinge tetrahedron. Any branch line traversing $S$ therefore must pass through $n-1$ consecutive non-hinge tetrahedra.

\begin{figure}[h]
\centering
\includegraphics{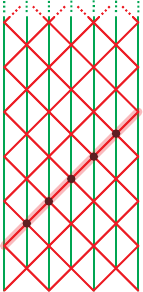}
\caption{A picture of the set $S$ in the proof of \Cref{lem:ABregion}. The $\wt \Phi$-edges are shown in green. By \Cref{lem:sectorveer}, each vertex in the interior of $S$ has the same veer. The highlighted branching segment passes through 5 consecutive nonhinge tetrahedra, corresponding to the vertices colored black.}
\label{fig:ABregion}
\end{figure}

We now need a combinatorial fact about veering triangulations.

\begin{claim}\label{claim:branching}
If a branch line $\gamma$ passes through consecutive non-hinge tetrahedra $T_1,\dots, T_k$, then the $T_i$ all lie in the fan of a single edge $e$. 
\end{claim}

\begin{proof}
Suppose without loss of generality that the top and bottom edges of $T_1$ are right veering. The branching line $\gamma$ passes through two faces  of $T_1$ which meet along a left veering edge $E$ of $T_1$ by \cite[Lemma 4.5]{LMT20}. We call these faces $f_1$ and $f_2$ where $f_1$ is a bottom face for $T_1$ and $f_2$ is a top face for $T_1$. If $f_3$ is the next face passed through by $\gamma$, \cite[Lemma 4.5]{LMT20} again implies that $f_2$ and $f_3$ meet along a left veering edge of $T_2$. Since $e$ is the only left veering edge of $f_2$, we conclude that $f_1, f_2, f_3$ are all incident to $e$ and that $T_1$ and $T_2$ lie in the same fan of $e$. Continuing in this way shows that each $T_i$ is in this fan.
This completes the proof of \Cref{claim:branching}.
\end{proof}

Returning to the proof of \Cref{lem:ABregion}: since any fan of length $>1$ containing non-hinge tetrahedra contains two distinct hinge tetrahedra and any non-hinge tetrahedron is part of such a fan (see e.g. \cite[Observation 2.6]{futer2013explicit}), \Cref{claim:branching} implies that there exists a fan of $\tau$ with size  at least $n+1$. Hence the number of bi-infinite AB strips in $D$ is less than the length $\delta_\tau$ of the longest fan in $\tau$, proving (i).

Now suppose that $D$ contains exactly $n$ bi-infinite AB strips, and let $g$ be the generator of the stabilizer of $D$. Let $s$ and $s'$ be two adjacent such strips in the sense that there are no strips between them. It is clear that $g$ permutes the set of bi-infinite AB strips in $D$; let $g'$ be a power of $g$ preserving $s$ and $s'$. By \Cref{lem_rayconvergence}, two of the boundary $\wt \Phi$-lines of $s$ and $s'$ eventually coincide. Since $s\cup s'$ is $g$-invariant, these boundary $\wt\Phi$-lines must be equal. Applying this argument $n-1$ times establishes (ii).
\end{proof}

\begin{remark}
If a dynamic plane $D$ contains at least two bi-infinite AB strips, then it corresponds to a region of the triangulation that Agol and Tsang call a \emph{wall} in their work-in-progress \cite{AgolTsang}. A key property, which they point out, is that these regions prevent $\Phi$ from being strongly connected. 
From our perspective, this can be seen by noting that
when there are at least two AB strips in $D$, there will be at least one component of $\wt \Phi\cap D$ which is a properly embedded line. Such a line descends to a $\Phi$-cycle in $M$ that is a circular source in the sense that it has no other incoming $\Phi$-edges.
\end{remark}

We say that two $\wt\Phi$-rays are \textbf{asymptotic} if they eventually agree. It is clear that asymptoticity is an equivalence relation on $\wt \Phi$-rays. We define the \textbf{width} of a dynamic plane $D$, denoted $w(D)$ to be the number of asymptotic classes of $\Phi$-rays contained in $D$. \Cref{lem_rayconvergence} and \Cref{lem:ABregion} imply the following:

\begin{corollary}[Width of dynamic planes] \label{cor:width}
Let $D$ be a dynamic plane. The width $w(D)$ of $D$ satisfies
\begin{align*}
w(D)&=1+(\text{\# of bi-infinite AB strips in $D$})\\
&\le \delta_\tau.
\end{align*}
\end{corollary}

The following lemma characterizes when the quotient of a dynamic plane is an annulus in terms of the veering combinatorics. 

\begin{lemma} \label{lem:dynamic_plane_orient}
Let $\gamma$ be a $\Gamma$-cycle which is not a branch curve. Let $\wt \gamma$ be a lift to $\wt M$, and let $g \in \pi_1(M)$ generate the stabilizer of $\wt \gamma$.
Then $L = D(\wt \gamma)/ \langle g \rangle$ is an annulus if and only if $\gamma$ has an even number of AB-turns.
\end{lemma}

\begin{proof}
In \cite[Lemma 5.6]{LMT20}, it is shown that
$\gamma$ has an even number of AB--turns if and only if the pullback of the tangent bundle over $B^s$ is orientable. 
However, the immersion $\gamma \to B^s$ factors through the immersion $L \to B^s$, and so $\gamma$ has an even number of AB--turns if and only if $L$ is orientable. 
From this, the lemma easily follows.
\end{proof}

The following proposition is a key technical result of this section.

\begin{proposition}[$\Phi$ sees most $\Gamma$-cycles] \label{prop:enough_flow}
Let $\gamma$ be a $\Gamma$-cycle.
Then $\gamma$ is either homotopic to a $\Phi$-cycle or to an AB-cycle of odd length.
\end{proposition}

In particular, the dual cycles that are not homotopic to flow cycles form a finite set of homotopy classes, up to positive multiples.

\begin{proof}
In the proof of \cite[Proposition 5.7]{LMT20} it is explained that every branch curve is homotopic to a $\Phi$-cycle. Hence we can assume that $\gamma$ is not a branch curve. It follows that any lift of $\gamma$ to $\wt M$ determines a dynamic plane.

Let $\wt\gamma$ be a lift of $\gamma$ to $\wt M$ and let $D = D(\wt \gamma)$. Let $g$ be the deck transformation of $\wt M$ that generates the stabilizer of $\wt \gamma$ and translates $\wt\gamma$ in the positive direction. Then $g D = D$, so $\gamma$ lifts to the core of $L = D / \langle g \rangle$, which is either an open annulus or open M\"obius band. We abuse notation slightly by referring to the images of $\wt \Phi$ and $\wt \Gamma$ in $L$ as $\Phi$ and $\Gamma$.

If the width $w(D)$ is equal to $1$,
and $\wt\rho$ is any $\wt\Phi$-ray contained in $D$, then $\wt\rho$ and $g\cdot \wt\rho$ eventually coincide. This follows from \Cref{lem_rayconvergence} and the fact if $w(D) = 1$ then $D$ has no infinite $AB$ strips (\Cref{lem:ABstrips}).
It follows that $\wt\rho$ is eventually $g$-periodic, and projects to a $\Phi$-cycle $\rho$ homotopic to the core of $L$. Hence $\gamma$ is homotopic to 
$\rho$, proving the claim in this case.

If $w(D)>1$, 
then $D$ has a nonempty AB region $D_{\text{AB}}$. Let $L_{\text{AB}}$ denote the image of this AB region in $L$.
Note that since $D_{\text{AB}}$ is $g$-invariant, $L_{\text{AB}}$ is an annulus or M\"obius 
band if and only if $L$ is an annulus or M\"obius band, respectively. We finish the proof by considering three cases.
\begin{itemize}
\item If $L_{\text{\text{AB}}}$ is an annulus, then there are $w(D)$ parallel $\Phi$-cycles in $L_{\text{\text{AB}}}$ homotopic to the core of $L$, so $\gamma$ is homotopic to a $\Phi$-cycle.
\item If $L_{\text{\text{AB}}}$ is a M\"obius band and $w(D)$ is odd, then there is a single $\wt \Phi$-line bisecting the AB region of $D$ which projects to a $\Phi$-cycle in $L$ and which is homotopic to the core of $L$, so $\gamma$ is homotopic to a $\Phi$-cycle.
\item Finally, if $L_{\text{\text{AB}}}$ is a M\"obius band and $w(D)$ is even, then there is a bi-infinite AB strip bisecting the AB region of $D$ whose core AB cycle projects to a $\Gamma$-cycle in $L$ homotopic to the core of $L$. 
By \Cref{lem:dynamic_plane_orient}, this AB cycle has odd length, so $\gamma$ is homotopic to an odd AB cycle.
\qedhere
\end{itemize}
\end{proof}

\subsection{Homotopy in dynamic planes}

We conclude this section with an additional fact, \Cref{lem:trans_hom} about dynamic planes that will be necessary in \Cref{sec:flowandgraph}. In its proof we will use the following lemma.

\begin{lemma}\label{lem:nogoingback}
If $\gamma_1$ and $\gamma_2$ are distinct directed paths in $\wt \Gamma$ with common endpoints, then each $\gamma_i$ contains an anti-branching turn.
\end{lemma}

\begin{proof}
Let $u$ and $v$ be the initial and terminal vertices of $\gamma_1$ and $\gamma_2$, respectively. By shortening the paths, we may assume $u$ and $v$ are the only common vertices of $\gamma_1$ and $\gamma_2$. Let $D$ be a dynamic plane containing $\gamma_1$ and $\gamma_2$, and let $A$ be the disk component of $D\setminus (\gamma_1\cup \gamma_2)$. Considering the local structure of the branch locus of $\wt B^s$, we see that the maw vector field must point into $A$ along the terminal edges of $\gamma_1$ and $\gamma_2$ and out of $A$ along the initial edges of $\gamma_1$ and $\gamma_2$. Since the maw vector field switches between pointing inward and outward at exactly the anti-branching turns, we conclude that each $\gamma_i$ contains an odd, and in particular nonzero, number of anti-branching turns.
\end{proof}

As a consequence, a path in $\wt \Gamma$ that deviates from a branch line can never return to that branch line.

Let $D$ be a dynamic plane stabilized by some $g \in \pi_1(M)$. As before, let $L$ denote the quotient $D/ \langle g \rangle$. Consider a $\Gamma$-cycle $\gamma$ contained in $L$. If there is a sector $\sigma$ of $L$ such that $\gamma$ runs along a side of $\sigma$ from its bottom vertex to its top vertex, then we may perform a homotopy of $\gamma$, supported on $\sigma$, that pushes $\gamma$ from one side of $\sigma$ to the other side of $\sigma$. We refer to this homotopy as \define{sweeping across the sector} $\sigma$. See \Cref{fig:sidemove}, where it is shown that sweeping across a sector is a homotopy through curves that are transverse to $\tau^{(2)}$.

\begin{lemma}
\label{lem:trans_hom}
Let $D$ be a dynamic plane stabilized by $g \in \pi_1(M) \ssm \{1\}$. 
Let $\wt \gamma_1$ and $\wt \gamma_2$ be $g$--invariant $\wt \Gamma$-lines contained in $D$ and assume that neither is a branch line. 
Then $\gamma_1 = \wt \gamma_1 /\langle g \rangle$ is homotopic to $\gamma_2 = \wt \gamma_2 /\langle g \rangle$ in $L = D/ \langle g \rangle$ by a homotopy that sweeps across sectors.
\end{lemma}

\begin{proof}
The embedded dual cycles $\gamma_1, \gamma_2$ in $L$ either intersect or not. If they intersect and are distinct, there is at least one connected component $U$ of $L-(\gamma_1\cup \gamma_2)$ with closure homeomorphic to a disk. Let $p_1$ and $p_2$ be the two segments of $\gamma_1$ and $\gamma_2$ which cobound $U$. 
By \Cref{lem:nogoingback}, each of $p_1$ and $p_2$ contains an anti-branching turn.

Let $(d_1, d_2)$ be the first anti-branching turn of, say, $p_1$. Since $U$ is tiled by sectors, it must be the case that $p_1$ traverses an entire side of $\sigma(d_2)$. When we sweep $\gamma_1$ across $\sigma(d_2)$, we shrink the region $U$ by $1$ sector. It follows that after sweeping across finitely many sectors we can homotope $\gamma_1$ to $\gamma_2$.

Next, suppose that $\gamma_1$ and $\gamma_2$ do not intersect. 
Note that this is not possible if $L$ is a M\"obius band and so we may assume that $L$ is an annulus.
Then there is a unique component of $L-(\gamma_1\cup \gamma_2)$ with compact closure, which we also call $U$; note that $U$ is an annulus with boundary components $\gamma_1$ and $\gamma_2$.

\begin{figure}[h]
    \centering
    \includegraphics{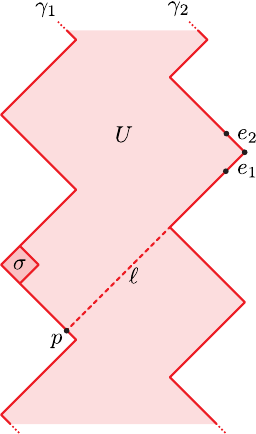}
    \caption{Notation from the proof of \Cref{lem:trans_hom}.}
    \label{fig:annuluscase}
\end{figure}

Let $(e_1,e_2)$ be an AB turn of $\gamma_2$ such that $\sigma(e_2)\subset U$, i.e. the maw vector field points into $U$ along $e_2$. Such a turn exists since $\gamma_2$ has a nonzero even number of anti-branching turns by \Cref{lem:dynamic_plane_orient}.
Let $\ell$ be the branch line through $e_1$. 
We claim that $\ell$ intersects $\gamma_1$ in addition to $\gamma_2$. To see this, first note that the negative subray of $\ell$ from $e_2$ (which is entirely contained in $L$) is not entirely contained in $U$, since $U$ is tiled by finitely many sectors. Further, this negative subray cannot return to $\gamma_2$ by \Cref{lem:nogoingback}.

Therefore the negative subray of $\ell$ from $e_1$ must intersect $\gamma_1$ as in \Cref{fig:annuluscase}. Let $p$ be the first vertex of intersection between this negative ray and $\gamma_1$. Let $(f_1,f_2)$ be the first anti-branching turn of $\gamma_1$ after $p$, and let $\sigma=\sigma(f_2)$. If $\ell'$ is the branching line through $f_1$, then the bottom of $\sigma$ can lie no lower on $\ell'$ than $p$. Hence $\gamma_1$ traverses the entire left side of $\sigma$ and can be homotoped across $\sigma$, shrinking the size of $U$ by one sector. After applying this argument finitely many times, we are finished.
\end{proof}

%%%%%%%%%%%%%%%

%% !TEX root =veering_poly2.tex

%%%%%%%%%%%%%%
\section{Pseudo-Anosov flows and veering triangulations}
\label{sec:flow_setup}

In this and subsequent sections, we will be primarily interested in a pseudo-Anosov flow $\phi$ without perfect fits on a closed $3$-manifold $\ol M$, as well as the manifold $M$ obtained by removing from $\ol M$ the singular orbits of $\phi$. In this setting, the construction of Agol--Gu\'eritaud  (\Cref{th:AG}) produces a veering triangulation $\tau$ on a manifold $N$ that is homeomorphic to $M$. Here we review some necessary background and terminology, with the essential properties of $\phi$ summarized in \Cref{lem:flow_space_prop}.
In \Cref{sec:trans}, we will show
$\tau$ can be realized as a triangulation of $M$ such that flow lines of $\varphi$ are positively transverse to $\tau^{(2)}$. 

First, let $\phi$ be a pseudo-Anosov flow on the closed $3$-manifold $\ol M$. 
We refer the reader to \cite[Section 4]{fenley2001quasigeodesic} for the precise definition, and informally summarize $\phi$'s features as follows:
\begin{itemize}
\item $\phi$ has finitely many (and at least one) singular periodic orbits where the return map on a transverse disk is locally modeled on a pseudo-Anosov surface homeomorphism near an $(n\ge 3)$-pronged singularity, 
\item  the orbits of the flow are $C^1$ and $\phi$ is smooth away from its singular orbits, 
\item there is a pair of mutually transverse $2$-dimensional singular foliations, called the \define{stable} and \define{unstable} foliations, whose leaves intersect in exactly the orbits of $\phi$, such that orbits in a leaf of the stable foliations are
exponentially contracted under $\phi$ and the orbits in a leaf of the unstable foliation are exponentially expanded.
\end{itemize}

A nonsingular closed orbit $\gamma$ of $\phi$ is \define{orientable} or \define{untwisted} if the stable leaf containing it is homeomorphic in the path topology to an annulus. Otherwise, the orbit is called \define{nonorientable} or \define{twisted} and the stable leaf is a M\"obius band.
\smallskip

Let $\mc Q$ denote the flow space of $\phi$ for $\ol M$, i.e. the space obtained by lifting to the universal cover $\wt{ \ol M}$ and collapsing flow lines of the lifted flow. According to Fenley--Mosher \cite[Proposition 4.1]{fenley2001quasigeodesic}, $\mc Q$ is homeomorphic to the plane and the lifts of $\phi$'s stable/unstable foliations project to a pair of transverse singular foliations $\mc F^{s/u}$ on $\mc Q$. The points of $\mc Q$ that are the images of (lifted) singular orbits of $\phi$ are called the \define{singularities} of $\mc Q$.
Note that there is a natural action $\pi_1(\ol M) \curvearrowright \mc Q$ by orientation preserving homeomorphisms, where the orientation on $\mc Q$ is induced by the fixed orientation on $\ol M$ and the orientation on flow lines.

Similarly, 
we let $\mr {\mc P}$ denote the flow space of $M$, defined by the same procedure, which
can also be obtained by taking the universal cover of $\mc Q$ minus its singularities. From this, we see that $\mr {\mc P}$ is also homeomorphic to the plane.
Moreover, this perspective allows us to define the \define{completed flow space} $\mc P$ of $M$ as the corresponding branched cover $\mc P \to \mc Q$ infinitely branched over the singularities of $\mc Q$. 
We also call the branch points of this map the \define{singularities} of $\mc P$. Since singularities of $\mc Q$ are discrete, so are the singularities of $\mc P$. Throughout, we extend terminology for $\mc Q$ to $\mc P$ by lifting. For example, we continue to denote the lifted singular foliations on $\mc P$ by $\mc F^{s/u}$. 
There is also an orientation preserving action $\pi_1(M) \curvearrowright \mc P$ by homeomorphisms that makes the branched cover $\mc P \to \mc Q$ equivariant with respect to the homomorphism $\pi_1(M) \to \pi_1(\ol M)$.
The projections to the flow space $\wt M \to \mc {\mr P}$ and $\wt{\ol M} \to \mc Q$ are oriented line bundles over the plane.

A \define{rectangle} $R$ in the flow space $\mc Q$ or $\mc P$ is a topological closed disk with no singularities in its interior with boundary consisting of four segments of leaves of $\FF^s$ and $\FF^u$. The boundary of $R$ necessarily consists of two stable leaf segments, which we call the \define{vertical boundary} of $R$ and denote $\partial_v R$, and two unstable leaf segments, which we call the \define{horizontal boundary} and denote $\partial_h R$.  (Note that by convention, we draw $\mc F^s$ vertically and $\mc F^u$ horizontally.) A \define{maximal rectangle} is a rectangle that contains a singularity in the interior of each of its sides, and so it is maximal with respect to inclusion.

As an informal definition, we say that a leaf
$\lambda^u$ of $\mc F^u$ and a leaf $\lambda^s$ of $\mc F^u$
form a \define{perfect fit} if they are disjoint but ``meet at infinity." We say that $\phi$ has \define{no perfect fits} if its flow space $\mc Q$ has {no perfect fits}. We omit the precise definition of a perfect fit (see \cite[Def. 2.2]{fenley2012ideal}) because, given the fact that singular leaves are dense in $\mc Q$ (see \Cref{lem:flow_space_prop}), no perfect fits is equivalent to the condition that every sequence of nested rectangles is contained in a maximal rectangle. The reader can take this as the definition of no perfect fits.
It is also proven by Fenley \cite[Theorem 4.8]{fenley1999foliations}, that when $\phi$ has no perfect fits, each $g\in \pi_1(\ol M)$ fixes at most one point of $\mc Q$ (again see \Cref{lem:flow_space_prop}).
Existence of maximal rectangles and uniqueness of fixed points are the essential properties of $\phi$ that we will use throughout this paper.

\begin{convention}[No perfect fits]
Henceforth, we will assume that the pseudo-Anosov flow $\phi$ has no perfect fits.
\end{convention}

Continuing with terminology, we define an \define{edge rectangle} $Q$ to be a rectangle in either $\mc Q$ or $\mc P$ with singularities at two of its (necessarily opposite) corners. 
(These were called \emph{spanning rectangles} in \cite{minsky2017fibered}).
Each edge rectangle $Q$ has a \define{veer} defined as follows: if the singularities of $Q$ are at its SW and NE corners, then $Q$ is \define{right veering}.
 Otherwise, $Q$ is \define{left veering}. 
 Here, the position of the singular vertices is determined by an orientation preserving embedding of $Q$ into $\mathbb{R}^2$ for which the restricted foliation $\mc F^s \cap Q$ maps to vertical lines and $\mc F^u \cap Q$ maps to horizontal lines.
 The veer of $Q$ is well-defined and an invariant of the $\pi_1$-actions on $\mc Q$ and $\mc P$ since these actions are orientation preserving. A \define{face rectangle} is a rectangle with a singularity at one of its corners and singularities in the interiors in each of its sides not containing the singular corner. Note that each face rectangle contains exactly three edge rectangles and is contained in exactly two maximal rectangles. Moreover, each maximal rectangle contains the face and edge rectangles determined by the pairs and triples of its singularities. See \Cref{fig:rectangles}.

\begin{figure}[h]
\begin{center}
\includegraphics{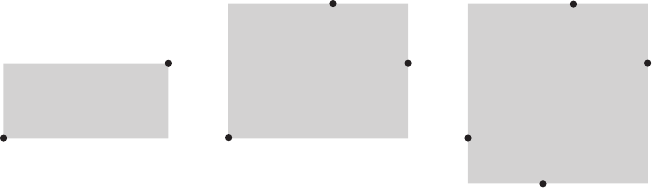}
\caption{From left to right we see an edge rectangle, a face rectangle, and a maximal rectangle.}
\label{fig:rectangles}
\end{center}
\end{figure}

We next define a partial order on rectangles in $\mc Q$ or $\mc P$.
We call rectangles $R_1$ and $R_2$ \define{ordered} if their interiors intersect but do not contain
any of each other's corners. 
Assuming $R_1$ and $R_2$ are ordered, if the interior of $R_2$ meets $\partial_h R_1$, then we say that $R_2$ is \define{taller} than $R_1$. If the interior of $R_1$ meets $\partial_v R_2$, then $R_1$ is \define{wider} than $R_2$. 
Finally, for ordered rectangles we say that $R_2$ \define{lies above} $R_1$ if $R_1$ is not taller than $R_2$, and $R_1$ \define{lies below} $R_2$ if $R_2$ is not wider than $R_1$. Put differently, $R_2$ lies above $R_1$ if they are ordered and $R_2 \cap \partial R_1$ contains a segment in each component of the horizontal boundary of $R_1$.
We say that $R_2$ lies \textbf{strictly above} $R_1$ if $R_2$ is taller than $R_1$ and $R_1$ is wider than $R_2$.

We note that if $R_1$ and $R_2$ are distinct maximal rectangles, then $R_2$ lies above $R_1$ if and only if $R_2$ is taller than $R_1$ if and only if $R_1$ is wider than $R_2$.
See \Cref{fig:lies_above}.

\begin{figure}[htbp]
\begin{center}
\includegraphics{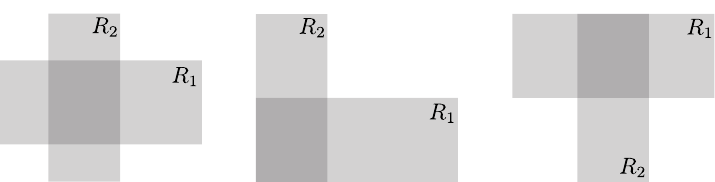}
\caption{Some ways $R_2$ can lie above $R_1$.}
\label{fig:lies_above}
\end{center}
\end{figure}

Properties of the flow $\phi$ translate to properties of the actions $\pi_1(\ol M) \curvearrowright \mc Q$ and $\pi_1(M) \curvearrowright \mc P$.  We record these in the following lemma that summarizes results from several papers of Fenley--Mosher, Fenley, and Mosher.
\begin{lemma}[Properties of the flow space] 
\label{lem:flow_space_prop}
With $\ol M, \phi, M, \mc Q, \mc P$ as above:
\begin{enumerate}
\item The foliations $\mc F^{s/u}$ are transverse singular foliations of $\mc Q$ with discrete singularities, no saddle connections, and dense singular leaves. 
\item The stabilizer of any leaf of $\mc F^{s/u}$ is either trivial or infinite cyclic, and each $g \neq 1$ in a leaf stabilizer fixes exactly one point in that leaf.
\item The orbit of any periodic point (i.e. point with nontrivial stabilizer) in $\mc Q$ is discrete and periodic points are dense. Moreover, since $\phi$ has no perfect fits, each $g \ne 1$ fixes at most one point in $\mc Q$. 
\item Suppose that $g$ fixes a point $p$ of $\mc Q$, chosen so that $g$ translates the $g$-periodic flow line projecting to $p$ in its positive direction.
\begin{itemize}
\item If $p$ is nonsingular, then for any edge rectangle or maximal rectangle $R$ containing $p$, $g(R)$ lies strictly above $R$.
\item If $p$ is singular and $R$ is a maximal rectangle containing $p$ in its boundary, then either $g(R)$ and $R$ have disjoint interiors or $g(R)$ lies strictly above $R$.
\end{itemize}
\end{enumerate}

Moreover, the corresponding statements for the completed flow space $\mc P$ also hold.
\end{lemma}

\begin{proof}
The properties listed in (1) have already been discussed except for the claim that singular leaves are dense in $\mc Q$. For this, we first recall that since $\ol M$ admits a pseudo-Anosov flow without perfect fits that is not conjugate to the suspension of an Anosov diffeomorphism, it is atoroidal \cite[Main theorem]{fenley2003pseudo} (see also the remarks following \cite[Theorem D]{fenley2012ideal}).
Then, since $\ol M$ is atoroidal, the flow $\phi$ is transitive by \cite[Proposition 2.7]{Mos92}. Finally, \cite[Proposition 1.1]{Mos92} and the sentence following it imply that \emph{every} leaf of the stable and unstable foliations on $\ol M$ is dense. Hence, the singular leaves of $\mc F^{s/u}$ are dense in $\mc Q$.

Next, the contracting/expanding dynamics within each leaf of the stable/unstable foliations implies that each nonsingular leaf with nontrivial $\pi_1$ is either an annulus or M\"obius band containing a unique closed orbit (see \cite[Section 1]{Mos92}). From this (2) easily follows.

The first statement of (3) follows from the fact that the orbit of a point in $\mc Q$ with nontrivial stabilizer corresponds to a closed orbit of $\phi$ in $\ol M$ and that the lifts of such an orbit to the universal cover $\wt{\ol M}$ form a discrete collection of flow lines.  The second statement follows from 
\cite[Theorem 4.8]{fenley1999foliations}. There, Fenley shows that if $g \neq 1$ fixes distinct points $p_1$ and $p_2$, then these points are connected by a so-called \emph{chain of lozenges}. The existence of a lozenge in $\mc Q$, which is essentially a rectangle with 2 ideal corners, implies that $\mc Q$ has a perfect fit. 

Finally, (4) follows from considering first return maps to transverse sections of the flow and using the expanding/contracting dynamics.
\end{proof}

A more uniform version of (4) will be useful later: Note first that it is easy to
obtain a collection of sections of the bundle $\wt {\ol M} \to {\mc Q}$
over the maximal
rectangles, which is equivariant by $\pi_1({\ol M})$, since the group action is free on maximal rectangles (any nontrivial element fixing a maximal rectangle fixes each of its singularities, contradicting \Cref{lem:flow_space_prop}.3). 
\begin{lemma}
\label{lem:flow and rectangles}
With $M, \phi,  {\mc Q}$ as above, fix a $\pi_1({\ol M})$-equivariant family of
sections $s_R \colon R \to \wt{\ol M}$ over the maximal rectangles in the flow space. Given
$\ep>0$ there is
a constant $L$ such that, if $J$ is an oriented segment in a nonsingular flow line in $\wt {\ol M}$ of length at least
$L$, so that its forward endpoint lies in the section over a rectangle $R_+$ and
its
backward endpoint lies in the section over a rectangle $R_-$ at distance at least $\ep$
from the boundary of $R_-$, 
then $R_+$ lies above $R_-$.
In fact both horizontal boundary components of $R_-$ pass through the interior of $R_+$ and both vertical boundary components of $R_+$ pass through the interior of $R_-$.
\end{lemma}
\begin{proof}
Let $p$ be the backward endpoint of $J$ in $s_{R_-}(R_-)$. Let $U$ be the maximal
connected set within $R_-$ containing $p$ such that the flow from $s_{R_-}(R_-)$ to
$s_{R_+}(R_+)$ is defined on $s_{R_-}(U)$. Then $U$ must be a subrectangle, and the pseudo-Anosov
properties of the flow, particularly its expansion on the unstable (horizontal) foliation,
implies that the width of $s_{R_-}(U)$ is bounded exponentially in $-L$. Thus for $L$ large
enough (depending on $\ep$) the width is small enough that both vertical (stable) sides of
$\boundary U$ are in the interior of $R_-$. Each horizontal side of $\boundary U$ must
therefore lie in the boundary of $R_-$, because otherwise it would be an
entire horizontal side of $R_+$, which would imply the interior of $R_-$ contains
a singular point.

We conclude that $R_+$ must cross $R_-$ from top to bottom, which proves the claim.
\end{proof}

We next observe a few basic consequences. The first essentially says that an infinite sequence of maximal rectangles which is increasing with respect to our partial order ``lies above" limits to a leaf of the vertical foliation $\mc F^s$.

\begin{fact}[Limits of rectangles]
\label{fact:non-accumulation}
Suppose that $(R_i)_{i \in \mathbb{Z}}$ is a sequence of distinct maximal rectangles such that $R_{i+1}$ lies above $R_i$ for each $i$. 

Then $\bigcap_{i \ge 0} R_i$
is a segment of the vertical foliation $\mc F^s$ in $R_0$ joining the components of $\partial_h R_0$, and
$\bigcap_{i \le 0} R_i$ is a segment of the horizontal foliation $\mc F^u$ in $R_0$ joining the components of $\partial_v R_0$.
\end{fact}

\begin{proof}
Note that if $Q = \bigcap_{i \ge 0} R_i$ is a rectangle with nonempty interior, then we could extend it vertically along leaves of $\mc F^s$ to a rectangle $Q'$ with singularities in its horizontal boundary. This follows from the density of singular leaves in \Cref{lem:flow_space_prop}. But then each $R_i$ necessarily lies below $Q'$ and so the singularities in the boundary of the $R_i$ would have to accumulate in $\mc Q$. This contradicts the discreteness of singularities, again as in \Cref{lem:flow_space_prop}.
\end{proof}

The next lemma will be used to show that the veering triangulation discussed in the next section has finitely many simplices.

\begin{lemma} \label{lem:finite_quotient}
There are finitely many maximal, face, and edge rectangles in $\mc Q$ (or $\mc P$) up to the $\pi_1$-action.
\end{lemma}

Before giving the proof, we make a few more observations.
Let $s\in \PP$ be a singularity. 
There are countably many singular leaves terminating at $s$; let $\ell_1$ and $\ell_2$ be two such such singular leaves. There is a unique component $C$ of $\PP-(\ell_1\cup \ell_2)$ whose frontier completely contains $\ell_1\cup\ell_2$. If $C$ contains no singular leaves terminating at $s$ (i.e. $\ell_1$ and $\ell_2$ are ``neighbors" at $s$) then the union $\Omega$ of $C$ with $\ell_1$ and $\ell_2$ is called an \define{orthant}. The point $s$ is called the \define{corner singularity} of $\Omega$. Note that if $\ell_1$ and $\ell_2$ bound an orthant, then one is stable and the other is unstable.

\begin{figure}[h]
\begin{center}
\includegraphics{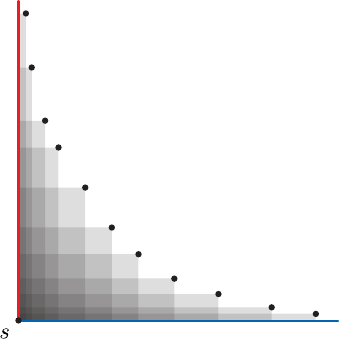}
\caption{Part of a staircase with corner singularity $s$.}
\label{fig:staircase}
\end{center}
\end{figure}

Following Gu\'eritaud, we say that a set $S$ of edge rectangles is a \textbf{staircase} if there is an orthant $\Omega\subset\PP$ with corner singularity $s\in \PP$ such that $S$ consists of exactly the edge rectangles contained in $\Omega$ with one corner at $s$. 
See \Cref{fig:staircase}.
Each orthant determines a unique staircase. Note that all the edge rectangles in a single staircase have common veer. 
The staircase $S$ has cyclic stabilizer $\langle g\rangle\le \pi_1(M)$, coinciding with the stabilizer of $\Omega$, and our convention will be to choose the generator $g$ so that $g Q$ lies above $Q$ for all edge rectangles $Q$ in $S$. Note, again by discreteness of singularities, that there are only finitely many edge rectangles in $S$ that lie above $Q$ and below $gQ$.
This allows us to choose an indexing $\dots,Q_{-1},Q_0,Q_1,\dots$ of the elements of $S$ so that $Q_i$ lies above $Q_j$ if and only if $i\ge j$. 

\begin{proof}[Proof of \Cref{lem:finite_quotient}]
We show that the $\pi_1$-action is cofinite on the edge rectangles in $\mc Q$. This immediately implies the same result for $\mc P$ and the case of face rectangles and maximal rectangles easily follows.

Since each edge rectangle lies in the staircase associated to each of its singular corners and there are only finitely many orthants up to the $\pi_1$ action, it suffices to show that for each staircase $S$ its cyclic stabilizer $\langle g \rangle$ acts cofinitely on the edge rectangles of $S$. This however is clear using the above ordering  
$\dots,Q_{-1},Q_0,Q_1,\dots$
and the fact that $g$ acts on this sequence by increasing the index. This completes the
proof.
\end{proof}

Since we are interested in the punctured manifold $M$ obtained by removing singular closed orbits from the closed manifold $\overline M$, we introduce some terminology to help remove the need to make special arguments when dealing with the singular orbits. Each singular orbit of $\overline M$ has some number of stable/unstable \define{prong curves} which are obtained by intersecting the stable/unstable leaves through the singular orbit with the boundary of a small neighborhood of the orbit. We consider the resulting prong curves as peripheral curves in $M$.

We will use the unstable prong curves to replace the missing singular orbits in our discussion below. For the flow $\phi$ on $M$, we denote by $\mc O_\phi$ the periodic orbits of the flow and by $\mc O_\phi^+$ the periodic orbits plus all positive multiples of the finitely many unstable prong curves.

\begin{remark}[The blown up flow on the compact model for $M$]
\label{rmk:cmpt}
One can also think of prong curves in the following way. In \cite[Section 5]{fried1982geometry}, Fried explains in detail how one can replace any orbit of a flow by its sphere of normal directions, and obtain a natural flow on the resulting manifold with boundary. If we apply this blowup operation to the singular orbits of $\phi$ on $\ol M$, we obtain a new flow $\phi^*$ on a manifold $M^*$ with toral boundary. The flow $\phi^*$ is tangent to $\partial M^*$ and, when restricted to the interior of $M^*$, conjugate to $\phi$ on $M$.

On each torus boundary component of $M^*$, $\phi^*$ has a finite even number of closed orbits, half of which are attracting and half of which are repelling. The attracting orbits correspond to unstable prong curves and the repelling orbits correspond to stable prong curves. While this is an attractive picture, we will continue to work with the flow $\phi$ on the noncompact manifold $M$.
\end{remark}
%%%%%%%%%%%%%%

\subsection{The Agol--Gu\'eritaud  construction} \label{sec:AG}
Let $\phi$ be a pseudo-Anosov flow on $\ol M$ with no perfect fits. Here we briefly describe the Agol--Gu\'eritaud  construction of a veering triangulation on a manifold homeomorphic to $M$. We will not dwell on the details here since in the next section we establish the stronger fact that the veering triangulation can be realized on $M$ so that it is positively transverse to flow lines.

Associate to each maximal rectangle $R$ in the completed flow space $\mc P$ 
a taut ideal tetrahedron $t_R$. 
We identify the ideal vertices of $t_R$ with the singularities of $R$ so the the edge rectangles contained in $R$ correspond to edges of $t_R$ and faces rectangles correspond to faces of $t_R$. The two angle $\pi$ edges of $t_R$ are the ones that correspond to edge rectangles spanning the singularities in $\partial_h R$ and $\partial_v R$, respectively. Moreover, the coorientations on the faces of $t_R$ are determined by declaring the two bottom faces of $t_R$ are the ones which contain the $\pi$ edge spanning the singularities in $\partial_v R$.
This convention is indicated in \Cref{fig:tet_rect} by drawing the edge joining the singularities in $\partial_h R$ above the edge joining the singularities in $\partial_v R$.

If faces $f_1$ of $t_{R_1}$ and $f_2$ of $t_{R_2}$ determine the same face rectangle in $\mc P$ (i.e. the rectangles spanned by their vertices are equal), then we glue together the corresponding faces. Since each face rectangle is contained in exactly two maximal rectangles, the resulting space $\wt N$ is a manifold away from its $1$--skeleton. By examining the ways that an edge rectangle can be extended to a maximal rectangle, one similarly verifies that $\wt N$ is a manifold. It is also the case that $\wt N$ is contractible.

\begin{figure}[h]
\begin{center}
\includegraphics{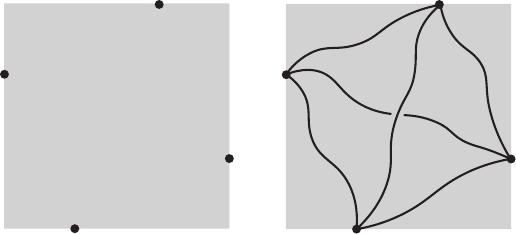}
\caption{From a maximal rectangle to an ideal tetrahedron. The coorientation convention is indicated by drawing the tetrahedron in the flow space as shown and coorienting its faces to point out of the page. The ideal edges are curvy so as to emphasize that there is no canonical way to draw them.}
\label{fig:tet_rect}
\end{center}
\end{figure}

Since the action of $\pi_1(M)$ on $\mc P$ preserves maximal, face, and edge rectangles, it induces a simplicial action on $\wt N$ which is cofinite on simplicies (\Cref{lem:finite_quotient}).
Because distinct singularity stabilizers have trivial intersection (and $\pi_1(M)$ is torsion free), each ideal simplex of $\wt N$ has trivial stabilizer and the action $\pi_1(M) \curvearrowright \wt N$ is discontinuous.
Moreover, because the peripheral subgroups of $\pi_1(M)$ precisely correspond to the 
stabilizers of singularities in $\mc P$, it follows that each of these subgroups acts peripherally on $\wt N$. Hence, by a theorem of Waldhausen 
\cite[Corollary 6.5]{Wal68}, the manifolds $M$ and $N = \wt N / \pi_1(M)$ are homeomorphic by a homeomorphism that is the identity on $\pi_1(M)$.

Let $\tau$ be the induced ideal triangulation of $N$. It is now straightforward to see that $\tau$ is naturally a veering triangulation. The coorientations on the faces of $\tau$ come from the convention discussed above and
taut structure on each tetrahedron comes from lifting to $\wt N$ and `projecting' the tetrahedron to its corresponding maximal rectangle.
Note that we are not claiming that there is a single coherent projection from $\wt N$ to $\mc {\mr P}$, although we will establish this in the next section. An edge is declared to be right veering if its lift to $\wt N$ determine an edge rectangle in $\mc P$ whose singularities are at its SW and NE corners. Otherwise, it is left veering. 

We summarize this as follows:

\begin{theorem}[Agol--Gu\'eritaud ]\label{th:AG}
Suppose that $\phi$ is a pseudo-Anosov flow on $\ol M$ without perfect fits. Then the above construction produces a veering triangulation $\tau$ on a manifold $N$ that is homeomorphic to $M = \ol M \ssm \{\text{singular orbits}\}$.
\end{theorem}

If the veering triangulation $\tau$ comes from the above construction, we say that $\tau$ is \define{associated} or \define{dual} to the flow $\varphi$.

%%%%%%%%%%%%%%

%% !TEX root =veering_poly2.tex

\section{Transversality to the flow}
\label{sec:trans}

\Cref{th:AG} constructs the veering triangulation from the structure of the flow space
of a pseudo-Anosov flow, but it does not make any claims about how the triangulation 
and flow coexist in the same manifold.
In this section we show that one can 
make the two positively transverse in the following sense:

\begin{theorem}\label{thm:flow transverse}
Let $\varphi$ be a pseudo-Anosov flow on $\ol M$ without perfect fits. 
Then the veering triangulation $\tau$ can be realized in $M$ so that
$\tau^{(2)}$ is a smooth cooriented branched surface which is positively transverse to the flow lines of $\varphi$.
\end{theorem}

Starting with $(N,\tau)$ as constructed in the previous section, we will build a
homeomorphism $N\to M$ which takes $\tau$ to the smooth transverse position of \Cref{thm:flow transverse}.
The proof has four main steps, which we summarize:

\subsection*{Fibration on $\wt N$:}  We first produce an equivariant fibration $p\colon \wt N \to \mr
{\mc P}$, which is an orientation preserving embedding on each face of
$\wt\tau^{(2)}$. The goal is to complete this diagram with an equivariant homeomorphism:
\begin{equation*}
  \begin{tikzcd}
    \wt N\arrow[rr,dashrightarrow]\arrow{dr}[swap]{p}  &    & \wt M \arrow[dl,"q"]\\
    & \mr {\mc P} & 
  \end{tikzcd}
\end{equation*}
where $q\colon\wt M\to\mr{\mc P}$ is the map to the flow space of the flow.

The key step is \Cref{prop_nobigons},  which gives an embedding of each edge of $\wt\tau$ in its
associated rectangle in $\mr{\mc P}$, so that the three edges of every face have disjoint
interiors. In the suspension-flow case this is simple because the flow space admits an
invariant affine structure in which every rectangle is Euclidean, and we may simply use
straight lines in \Cref{fig:tet_rect} (indeed this is how the original veering picture is obtained). In the
general setting there is no obvious way to do this -- equivariance produces some tricky
constraints which are reflected in the argument we give in \Cref{sec:draw diags}. 
Most of the effort of the proof goes into this step. 
The map $p$ is then produced in \Cref{prop:positive fibration}.

\subsection*{Compactification and a fiberwise map:}

We next build a preliminary map that takes $p$-fibers in $\wt N$ to $q$-fibers in $\wt
M$. But in order to have uniform control of it, we compactify $N$ and $M$ and extend the
map. We compactify $N$ to a manifold $\ol N$ with toral boundary and construct a
map $\ol h: \ol N \to \ol M$ that takes $p$-fibers to flow orbits, and boundary tori to
singular orbits. Thus we obtain the diagram
\begin{equation}\label{N hat diagram}
  \begin{tikzcd}
      & {\mc P} & \\
    \wh N\arrow[rr,"\wh h"]\arrow[d]\arrow{ur}{p}  &    & \wh M \arrow[d]\arrow[swap]{ul}{q}\\
	\ol N\arrow[rr,"\ol h"]  &    & \ol M 
  \end{tikzcd}
\end{equation}
where $\wh N$ and $\wh M$ are completions of $\wt N$ and $\wt M$ obtained by lifting the
compactifications. 

The restriction of $\wh h$ to each fiber may not be an embedding,
but we show in \Cref{lemma:   regularity of h} that it
is proper and degree $1$ to its image fiber.

\subsection*{Straightening the fibers:}
An averaging step, convolving with a fiberwise bump function, produces a map
which is an embedding on the fibers and hence a global homeomorphism. 
This is carried out in \Cref{prop:straighten}, and gives us a topological version
of our main result, \Cref{prop:top_trans}.

\subsection*{Smoothing}
Finally we address the issue of making the branched surface smooth, and furthermore
making sure that the images of edges in the flow space are smooth and transverse to both
foliations. The first of these is explained in \Cref{prop:smooth f}, and the second
in \Cref{prop:transversal diagonals}. 

\medskip

We next turn to carrying out the details.

\subsection{Step 1: Drawing diagonals and building a fibration}
\label{sec:draw diags}
For any points $p,q \in \mc P$ lying in a maximal rectangle $R$, but not in single leaf of $\mc F^{s}$ or $\mc F^u$, we denote by $R(p,q) \subset R$ the unique rectangle with opposite vertices at $p$ and $q$. So if $p,q$ are singularities of $R$, then $Q = R(p,q)$ is their edge rectangle. Recall that each edge rectangle corresponds to an edge of $\wt \tau$ by construction, and 
the veer of an edge rectangle is the veer of its associated edge. 

A \textbf{veering diagonal} is a topological arc in an edge rectangle $Q = R(p,q)$ which connects $p$ to $q$ and is topologically transverse to the stable and unstable foliations, meaning that the path intersects each leaf in $R(p,q)$ at most, and so exactly, once.

Our first step is to prove the following:
\begin{proposition}\label{prop_nobigons}
There exists an equivariant family of veering diagonals 
so that the three veering diagonals of every face rectangle have disjoint interiors.
\end{proposition}

\subsubsection{Drawing diagonals given anchors}

We say that the pair $(A,\alpha)$ is an \define{anchor system} if $\alpha$ is a bijection from the set of edge rectangles in $\mc P$ onto a subset $A \subset \mc P$ with the following properties:
\begin{itemize}
\item \emph{containment}: for each edge rectangle $Q$, $\alpha(Q)$ lies in the interior of $Q$,
\item\emph{equivariance}: $g \cdot \alpha(Q)=\alpha(g \cdot Q)$ for each edge rectangle $Q$ and each $g\in \pi_1(M)$, and
\item \emph{staircase monotonicity}: for edge rectangles $Q_1$ and $Q_2$ that share a singular corner $s$, 
if $Q_1$ is wider than $Q_2$, then $R(s,\alpha(Q_1))$ is wider and no taller than $R(s,\alpha(Q_2))$.
\end{itemize}
When working with a given anchor system, we will refer to $\alpha(Q)$ as \emph{the anchor for $Q$} and call $A$ the \emph{set of anchors}.
Note that in the description of staircase monotonicity, $Q_2$ must be taller than $Q_1$. However, we do not require the same of $R(s,\alpha(Q_2))$ and $R(s,\alpha(Q_1))$, meaning $\alpha(Q_1)$ and $\alpha(Q_2)$ are allowed to live in the same horizontal leaf.

\begin{figure}[h]
\includegraphics{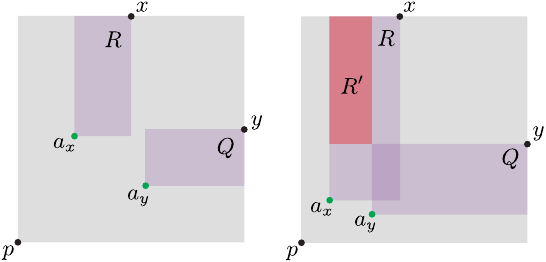}
\caption{The right veering case in the definition of busy, with anchors shown in green. The face rectangle on the left is not busy, and the face rectangle on the right is busy. Note that it is possible that the two anchors lie in the same horizontal line.
}
\label{fig_busy}
\end{figure}

Let $(A,\alpha)$ be an anchor system.
Let $F\subset \mc P$ be a face rectangle and let $p$ denote the unique singularity lying at a corner of $F$. Let $x$ be the singularity lying on a horizontal edge of $\partial F$ and let $y$ be the last singularity, which necessarily lies on a vertical edge of $F$. Let $a_x=\alpha(R(p,x))$ and $a_y=\alpha(R(p,y))$.
If $R=R(a_x,x)$ and $Q=R(a_y,y)$ intersect nontrivially, we say $F$ is \textbf{busy}. If $F$ is busy, let $R'$ be the maximal subrectangle of $R$ with the property that the stable and unstable leaves through each point in $R'$ do not intersect the interior of $Q$ (see the right side of \Cref{fig_busy}). This subrectangle exists by staircase monotonicity. A point in $R'$ which corresponds to a periodic orbit is called an \textbf{$F$-buoy}. Because the points corresponding to periodic orbits are dense in $\mc P$ (\Cref{lem:flow_space_prop}), any busy face rectangle $F$ has an $F$-buoy.

\begin{lemma}\label{lem_nobigons}
If there exists an anchor system for $\PP$, then there exists an equivariant family of veering diagonals 
so that the three veering diagonals of every face rectangle have disjoint interiors.
\end{lemma}

That is, if an anchor system exists, then \Cref{prop_nobigons} holds.

\begin{figure}[h]
\includegraphics{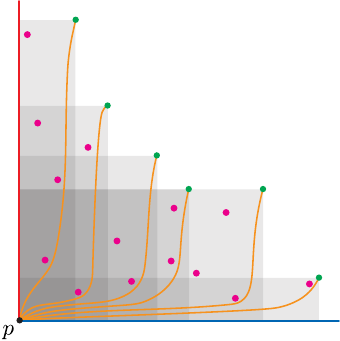}
\caption{Drawing half-diagonals (orange) satisfying properties (1)--(3) in the proof of \Cref{prop_nobigons}. The green points are anchors and the pink points are buoys.}
\label{fig_halfdiag}
\end{figure}
\begin{proof}

Let $(A,\alpha)$ be the given anchor system, which determines busy face rectangles. For each $\pi_1(M)$-orbit of busy face rectangle $F$, choose an $F$-buoy
$b_F$ and let $B_F$ be the $\pi_1(M)$-orbit of $b_F$. There are finitely many orbits of
face rectangles, so there are finitely many sets $B_F$. Let $B=\bigcup B_F$ be their
union, which we call the set of \emph{buoys}. Note that $B$ is $\pi_1(M)$-invariant and
discrete, since orbits of periodic points are discrete by \Cref{lem:flow_space_prop}.

Let $S=\{\dots ,Q_{-1}, Q_0, Q_1,\dots\}$ be a staircase with corner singularity $p$. If $g$ generates the stabilizer of $S$, choose a $\langle g\rangle$--equivariant family of continuous paths from $p$ to the anchors of elements of $S$ with the following three properties:
\begin{enumerate}
\item For each $Q_i$, the path from $p$ to $\alpha(Q_i)$ is homotopic rel endpoints in $R(p,\alpha(Q_i)) \ssm B$ to the first-horizontal-then-vertical path from $p$ to $\alpha(Q_i)$,
\item the paths are disjoint except at $p$, and
\item the paths are topologically transverse to the stable and unstable foliations, meaning that no path intersects a leaf more than once.
\end{enumerate}
See \Cref{fig_halfdiag}. We can choose such a family by staircase monotonicity and the discreteness of $B$. We call each of these paths a \emph{half-diagonal}.

Having chosen such a family of half-diagonals for each $\pi_1(M)$--orbit of staircase in $\mc P$, we can specify a veering diagonal for each edge rectangle $R(p,q)$ as the union of the two half diagonals from $p$ and $q$ to the anchor for $R(p,q)$. Let $\mc D$ denote the union of all these veering diagonals.

Let $F$ be a face rectangle with corner singularity $p$, and let $e,f\subset F$ be the two diagonals in $\mc D$ of the same veer. The two half-diagonals incident to $p$ are disjoint by property (2) above. The two half-diagonals not incident to $p$ are also disjoint since $F$ is either not busy, in which case disjointness is clear; or busy, in which case property (1) above guarantees disjointness. Therefore $e\cap f=\{p\}$, and it is clear that $e$ and $f$ are the only pair of diagonals of $F$ whose interiors could intersect. 
This completes the proof.
 \end{proof}

\subsubsection{Choosing anchors}
\Cref{lem_nobigons} reduces the problem of drawing diagonals to finding an anchor system, which we shall do now.

Let $Q=Q_0\subset \PP$ be an edge rectangle. Let $\kappa(Q)$ be the unique bi-infinite sequence of edge rectangles
\[
\kappa(Q)=(\dots,Q_{-2},Q_{-1},Q_0,Q_1,Q_2\dots)
\]
such that for all $i$ there exists a maximal rectangle $R_i$ such that $Q_i$ and $Q_{i+1}$ are the widest and tallest edge rectangles of $R_i$, respectively.

We call $\kappa(Q)$ the \textbf{core sequence} of $Q$. If each edge rectangle in $\kappa(Q)$ has the same veer, we say that $Q$ is \textbf{homogeneous}.

By the density of singular stable and unstable leaves in $\PP$, the intersection of all rectangles of $\kappa(Q)$ contains only one point (see \Cref{fact:non-accumulation}). We denote this point $c(Q)$ and call it the \textbf{core point} of $Q$. It is clear that all rectangles in $\kappa(Q)$ have the same core point.

\begin{remark}
The core sequence $\kappa(Q)$ can also be regarded as a line in the lift $\wt \Phi$ of the flow graph $\Phi$ to $\wt M$. In \Cref{sec:flowandgraph} we define a function $\wt{\mf F}$ which maps $\wt\Phi$-lines to points in $\mc P$. In the language of that section, the core point $c(Q)$ is the image of $\kappa(Q)$ under $\wt{\mf F}$.
\end{remark}

The core function $c$ mapping each edge rectangle to its core point satisfies the containment and equivariance properties. In addition, it nearly satisfies staircase monotonicity. Ultimately our construction of a set of anchors will be a slight modification of this core point mapping, where the modification will be necessary wherever two rectangles in the same staircase share a core point. 
The following lemma precisely describes the failure of core points to be monotonic in staircases.

\begin{lemma}[Weak monotonicity for core points] \label{lem:weak_mono}
Let $Q_1$ and $Q_2$ be edge rectangles that share a singular corner $s$ where $Q_2$ lies strictly above $Q_1$, and suppose that $Q_1$ and $Q_2$ are adjacent in this staircase. If $c_i$ denote the core point of $Q_i$, then $c_1=c_2$ if and only if $Q_1$ and $Q_2$ are both homogeneous, and otherwise $R(s,c_2)$ lies strictly above $R(s,c_1)$.
\end{lemma}

Recall from \Cref{sec:flow_setup} that $Q_2$ lies \textbf{strictly above} $Q_1$ if $Q_2$ is taller and $Q_1$ is wider. The same applies to the $R(s, c_i)$.  

\begin{proof}
Assume without loss of generality that $Q_1$ and $Q_2$ are right veering. In this proof we will assume for readability that the orthant determined by the staircase at $s$ containing $Q_1$ and $Q_2$ is identified in an orientation-preserving way with the first quadrant in $\R^2$; in particular there are well defined local notions of north, south, east, and west. 

Let $q_i$ be the singular points such that $Q_i=R(s,q_i)$, and let $R_i$ be the maximal rectangle for which $Q_i$ is the bottom edge rectangle. Note that since $Q_1$ and $Q_2$ are adjacent in the staircase at $s$, $R_1$ contains $s, q_1, q_2$ in its boundary. Let $s_2$ be the fourth singular point in $\partial R_1$.
Finally, let $\ell_i^{v/h}$ be the vertical/horizontal leaf through $c_i$. That is, $\ell_i^{v/h}$ is the leaf of $\mc F^{s/u}$ through $c_i$.

We first show that if either of $R_1$ or $R_2$ is hinge, then $\ell_2^v$ lies west of $\ell_1^v$.
If $R_1$ is hinge, then $\ell_1^v$ must lie strictly to the right of $\ell_2^v$ because $\ell_2^v$ must pass through the interior of $Q_2$ and $\ell_1^v$ must pass through the interior of $R(q_2,s_2)$, and when $R_1$ is hinge these rectangles have disjoint interiors. If $R_1$ is non-hinge, 
then $R_2$ contains $s, s_2$, and $q_2$ in its boundary. Let $q_3$ be the fourth singular point in $\partial R_2$. The leaves $\ell_2^v$ and $\ell_1^v$ must pass through the interiors of $R(q_3,s_2)$ and $R(s_2, q_2)$. If $R_2$ is hinge,
 then these rectangles have disjoint interiors so $\ell_2^v$ lies west of $\ell_1^v$ in this case. 

Moving backward in the core sequences, let $S_1$ and $S_2$ be the maximal rectangles for which $Q_1$ and $Q_2$ are the top edge rectangles. A symmetric argument to the one in the previous paragraph shows that if either $S_1$ or $S_2$ is hinge, then $\ell_2^h$ must lie north of $\ell_1^h$.

Now suppose that $R_1$ and $R_2$ are both non-hinge. In this case $R(s_2,q_3)$ and $R(s_2, q_2)$ are the next edge rectangles in the core sequences of $Q_2$ and $Q_1$ respectively, they are adjacent in a staircase at $s_2$, and they have core points $c_1$ and $c_2$ respectively. Iterating the reasoning from above shows that if the subsequence $\kappa_+(Q_i)$ of $\kappa(Q_i)$ starting at $Q_i$ contains a left veering edge for $i=1$ or $i=2$, then $\ell_2^v$ will lie west of $\ell_1^v$.
Symmetrically, if the subsequence $\kappa_-(Q_i)$ of $\kappa(Q_i)$ ending at $Q_i$ contains a left veering edge for either $i=1$ or $i=2$ then $\ell_2^h$ will lie north of $\ell_1^h$.

Since the core sequence is periodic modulo an element of $\pi_1(M)$, $\kappa_+(Q_i)$ contains a left veering rectangle if and only if $\kappa_-(Q_i)$ does. This shows that $c_1$ lies strictly northwest of $c_2$ and hence $R(s,c_1)$ lies strictly above $R(s,c_2)$ unless both $Q_1$ and $Q_2$ are homogeneous.

\begin{figure}[h]
\centering
\includegraphics{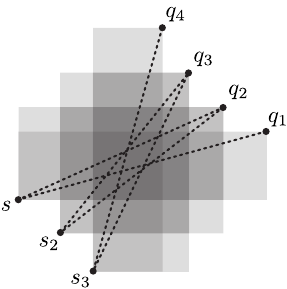}
\caption{A diagram of the labeling in the end of the proof of \Cref{lem:weak_mono}. We have drawn dotted diagonals in the figure as a visual aid, but we emphasize that \Cref{lem:weak_mono} logically precedes the drawing of any diagonals.
}
\label{fig:2pinching}
\end{figure}

It remains to show that if both $Q_1$ and $Q_2$ are homogeneous, then $c_1=c_2$. If $Q_1$ and $Q_2$ are both homogeneous, then let $s_2$ and $q_3$ be as above. Further, let $s_3,s_4\dots$ and $q_4, q_5,\dots$ be the singular points so that the forward core subsequences starting at
$Q_1$ and $Q_2$ are
\[
\kappa_+(Q_1)=(Q_1=R(s,q_1), R(s_2,q_2), R(s_3,q_3), \dots)
\]
and
\[
\kappa_+(Q_2)=(Q_2=R(s,q_2), R(s_2, q_3), R(s_3, q_4),\dots).
\]
See \Cref{fig:2pinching}. The sequence starting with $Q_1$ converges to $\ell_1^v$ and the sequence starting with $Q_2$ converges to $\ell_2^v$. Since $R(s_i, q_{i+1})$ lies strictly above (and is in particular contained east-west in) $R(s_i,q_i)$ for all $i$, we see that $\ell_1^v=\ell_2^v$. A symmetric argument moving backward in the core sequences shows that $\ell_1^h=\ell_2^h$, so we see that $c_1=c_2$ as claimed.
\end{proof}

\Cref{lem:weak_mono} says that core points fail to be monotonic in staircases precisely when a staircase has consecutive homogeneous edge rectangles. If $n\ge 2$ and $Q_1,\dots, Q_n$ are consecutive homogeneous edge rectangles in a staircase (i.e. $c(Q_1) = \ldots =c(Q_n)$), we say that $Q_1,\dots, Q_n$ are \textbf{pinched}.

Each core point $c$ has a unique nontrivial element $g$ generating its stabilizer and translating upward (i.e. mapping an edge rectangle containing $c$ to one that lies strictly above it). Let $P = P_c$ be the set
 of all the edge rectangles in $\mc P$ that have core point $c$. We call $P$ the \textbf{preimage} 
 of $c$. Note that each edge rectangle belongs to a unique preimage and that each preimage is $g$-invariant. 
We have the following basic fact about preimages, which says that if a single core sequence associated to a core point $c$ is homogenous, then every core sequence associated to $c$ is homogeneous.

\begin{lemma}\label{lem:sameveer}
Let $c\in \PP$ be a core point. Then $c$ is associated to a homogenous edge rectangle if and only if the preimage of $c$ contains edge rectangles of only one veer.
\end{lemma}

\begin{proof}
For the if statement, if the preimage of $c$ contains only edges of one veer then it is immediate that every edge rectangle with core point $c$ is homogenous.

Now suppose that $Q_R$ and $Q_L$ are respectively right and left veering edge rectangles that share the core point $c$. Because they have opposite veer, one must lie strictly above the other. Suppose without loss of generality that $Q_R$ lies above $Q_L$. One can see from a picture that if $Q'$ is the next edge rectangle in the core sequence for $Q_L$, then $Q'$ is either right veering or lies strictly beneath $Q_R$. It follows that the core sequence $\kappa(Q_L)$ for $Q_L$ must contain a right veering term. A symmetric argument moving backward in the core sequence shows that $\kappa(Q_R)$ contains a left veering term. This proves the contrapositive of the only if statement.
\end{proof}

If the preimage $P$ of a core point $c$ contains pinched edge rectangles, we say both that $c$ is pinched and that $P$ is pinched. To review the terminology: a homogeneous edge rectangle is pinched if it has a neighbor in a staircase which is also homogeneous. A core point is pinched if it is the core point of a pinched edge rectangle. A preimage is pinched if it is the preimage of a pinched core point, or equivalently if it contains a pinched edge rectangle.
By \Cref{lem:weak_mono}, if the core points of a staircase are not strictly monotonic, then the corner singularity of that staircase meets a rectangle in a pinched preimage.

\begin{claim} \label{claim:core_box}
There exists a family of rectangles $B=\{b(Q)\subset \PP\mid Q \text{ is an edge rectangle}\}$ satisfying the following properties.
\begin{enumerate}[label=(\alph*)]
\item For every edge rectangle $Q$, $b(Q)\subset\intr (Q)$ and contains the core point of $Q$.
\item If $Q_1$ and $Q_2$ are edge rectangles with distinct core points in the same staircase with corner singularity $p$, and $Q_2$ lies strictly above $Q_1$, then $R(p,x_2)$ lies strictly above $R(p,x_1)$ for any $x_i \in b(Q_i)$.
\item The family $B$ is $\pi_1(M)$-equivariant, meaning $b(g\cdot Q)=g\cdot b(Q)$ for all $b(Q)\in B$ and $g\in \pi_1(M)$.
\end{enumerate}
\end{claim}

\begin{figure}
    \centering
    \includegraphics{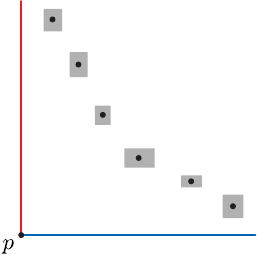}
    \caption{Choosing preliminary rectangles around the core points of a staircase with corner singularity $p$, such that the monotonicity condition (b) in \Cref{claim:core_box} is satisfied.}
    \label{fig:coreboxes}
\end{figure}

\begin{proof}
Let $S=\{Q_i\}$ be the staircase incident to a singular point $p$. Since the core points of $Q_i$ are monotonic, we can choose a rectangle for each core point which satisfies the monotonicity condition (b) for $S$ from the claim. See \Cref{fig:coreboxes}.
We can do this in a $\langle g\rangle$-equivariant way, where $g$ is the primitive element of $\pi_1(M)$ stabilizing $S$. We call these rectangles \emph{preliminary rectangles}.
We can use these preliminary rectangles to define preliminary rectangles for every core point in every $\pi_1$-translate of $S$ by requiring equivariance. We repeat this construction on every $\pi_1$-orbit of staircases. The result is that for every edge rectangle we have two preliminary rectangles, and the collection of all preliminary rectangles is $\pi_1(M)$-invariant. For each edge rectangle $Q$, we can take $b(Q)$ to be the intersection of the two preliminary rectangles for $Q$.
\end{proof}

We choose, and fix for the remainder of this section, a family $B=\{b(Q)\}$ satisfying the conditions of \Cref{claim:core_box}. We will call the elements of $B$ \textbf{core boxes}.

We will now construct a pair $(A, \alpha)$ and show that it is an anchor system.
First, if $Q$ is an edge rectangle that is not pinched (the preimage of its core point is not pinched), then set $\alpha(Q) = c(Q)$. If a preimage is pinched, we will coherently choose $\alpha$-values for each edge rectangle in the preimage, guided by our collection $B$ of core boxes. Suppose that $Q$ is a pinched edge rectangle, let $P$ be the preimage of $c=c(Q)$, and let $\ol P$ be the union of all edge rectangles of $P$. By \Cref{lem:sameveer} every rectangle of $P$ has the same veer. Without loss of generality we will treat the case when each rectangle is right veering.

\begin{claim} \label{claim:coordinates}
Let $a>1$ and $Q_0\in P$.
There exists an embedding $\Psi_a \colon \ol P \to \RR^2$ such that $\Psi_a(c) = 0$ and $\Psi_a$ conjugates the action of $g$ on $\ol P$ to 
\begin{align}\label{eq:coordinates}
(x,y) \mapsto \pm \left (\frac{x}{a}, ay \right )
\end{align}
where the minus sign occurs if and only if $c$ corresponds to a twisted orbit, and such that:
\begin{itemize}
\item if $c$ is untwisted, the singularities of $Q_0$ map to $\pm (1,1)$, and 
\item if $c$ is twisted, then one singular corner of $Q_0$ maps to $(1,1)$ and the other maps to a point $(x,y)$, for some $x,y$ such that $-a \le x,y \le -1/a$.
\end{itemize}
\end{claim}

\begin{proof}
Let $\ell^v, \ell^h$ be the vertical and horizontal leaves (i.e. leaves of $\mc F^{s/u}$) through $c$. By definition, the vertical/horizontal leaves through each point of $\ol P$ meet $\ell^{h/v}$, giving a coordinate system on $\ol P$ once we have chosen identifications of both $\ell^v, \ell^h$ with $\RR$. 

If $c$ is untwisted then the dynamics of the action of $\pi_1(M)$ on $\PP$ allow us to choose homeomorphisms of each half leaf $\ell^{v}_\pm$ (resp. $\ell^h_\pm$) at $c$ with $\R_{\ge0}$, which conjugate the action of $g$ with multiplication by $a$ (resp. $1/a$), so that the associated function $\ol P\to \R^2$ maps the corners of $Q_0$ to $(\pm1, \pm1)$.

In the twisted case, we first send the half leaf $\ell^v_+$ to $\RR$ via a map $f^v$ so that the point of intersection between $\ell_v^+$ and the vertical boundary of $Q_0$ goes to $1$ and so that the action of $g^2$ is conjugated to multiplication by $a^2$. Then we define the map on $\ell^v_-$ by $p \mapsto -a \; f^v(g^{-1}p)$. The symmetric procedure for the horizontal leaf produces the required coordinates.
\end{proof}

In both the twisted and untwisted cases, we call $Q_0$ the \textbf{normalizing rectangle} for the coordinates given by $\Psi_a$. We fix a particular normalizing rectangle $Q_0$.

For any $a>1$, and every edge rectangle $Q$ in $P$, we can draw a straight line $\gamma_Q$ in the $\Psi_a$-coordinates on $\ol P$ connecting the singularities of $Q$.
Identifying $\ol P$ with its image under $\Psi_a$, we define $\alpha_a(Q)$ to be the point of intersection between $\gamma_Q$ and the $x$-axis (i.e. the horizontal leaf through $c$). For all $Q$ in $P$, we have $\alpha_a(g\cdot Q)=g\cdot \alpha_a(Q)$ since $g$ preserves straight lines in $\Psi_a$-coordinates.

\begin{claim}
There exists $a>1$ such that for each edge rectangle $Q$ in $P$, the point $\alpha_a(Q)$ lies in the core box $b(Q)$. 
\end{claim}

\begin{proof}
By \Cref{lem:finite_quotient},
the action of $\pi_1(M)$ on the set of all edge rectangles is cofinite. If two elements of $P$ are related by an element of $\pi_1(M)$, then this element must lie in the stabilizer of $c$, which is equal to $\langle g\rangle$. It follows that the action of $\langle g\rangle$ on $P$ is cofinite.

Let $Q$ be an edge rectangle of $P$. We claim that as $a\to 1^+$, $\alpha_a(Q)\to c$. To see this, first note that since $c$ is $g$--invariant, it suffices to prove the claim for any $g$--translate of $Q$. Because the action of $\langle g\rangle$ on $P$ is cofinite, each $Q$ has a translate lying between $Q_0$ and $g^kQ_0$ for some $k$, not depending on $Q$, and so we replace $Q$ with this translate. Then using the description of the action of $g$ in $\Psi_a$--coordinates from \cref{eq:coordinates}, we observe that $\Psi_a(g^kQ_0)$ (and $\Psi_a(Q_0)$) converge to the square with corners at $(\pm 1, \pm 1)$ as $a \to \infty$.
In the case where $c$ is twisted, this uses the fact (\Cref{claim:coordinates}) that when $a$ is close to $1$, the singular corner of the normalizing rectangle $Q_0$ in the negative quadrant is approaching $(-1,-1)$. Since $Q$ is above $Q_0$ and below $g^kQ_0$, we also must have the same for $\Psi_a(Q)$ and we conclude that $\alpha_a(Q)\to c$ as $a\to 1^+$.

Next, let $Q_1,\dots Q_n$ be elements of $P$ that together represent each $g$-orbit. We see that there exists an $a>1$ such that $\alpha_a(Q_i)\in b(Q_i)$ for $i=1,\dots, n$ since $c$ lies in the interior of each core box. Since the collection of core boxes is equivariant, this implies that $\alpha_a(Q)\in b(Q)$ for all $Q$ in the preimage $P$.
\end{proof}

Now fix such an $a>1$, and define $\alpha(Q)=\alpha_a(Q)$ for all $Q$ in the preimage $P$. 

We can perform this procedure for an orbit representative of each pinched preimage, and extend to all pinched preimages by $\pi_1(M)$-equivariance. Since each edge rectangle $Q$ is either unpinched or is contained in a unique pinched preimage, this equivariantly assigns a point $\alpha(Q)$ to each edge rectangle $Q$. Set $A = \{\alpha(Q)\}$, where $Q$ varies over all edge rectangles. 

\begin{lemma} \label{lem:anchors}
The pair $(\alpha, A)$ is an anchor system.
\end{lemma}

\begin{proof} 
It only remains to prove monotonicity in staircases, i.e.
for edge rectangles $Q_1$ and $Q_2$ that share a singular corner $s$, 
if $Q_1$ is wider than $Q_2$ and $a_i=\alpha(Q_i)$, then $R(s,a_1)$ is wider and no taller than $R(s,a_2)$. We assume without loss of generality that $Q_1$ and $Q_2$ are right veering, and for convenience we identify the orthant of $s$ determined by the $Q_i$ with the first quadrant of $\R^2$.

By \Cref{lem:weak_mono}, monotonicity can only fail if $Q_1$ and $Q_2$ share a core point $c$. In this case, $Q_1$ and $Q_2$ lie in the same pinched preimage $P$, and we can consider these edge rectangles in the coordinates $\Psi_a$ where their $\alpha$-values were chosen. In this case, since $Q_1$ is wider than $Q_2$, and  $\gamma_{Q_1}$ and $\gamma_{Q_2}$ are line segments with disjoint interiors, 
we see $a_1$ lies east of $a_2$. 
This immediately implies that $R(s,a_1)$ is wider than $R(s,a_2)$ and the proof is complete.
\end{proof}

\Cref{lem_nobigons} and \Cref{lem:anchors} together complete the proof of \Cref{prop_nobigons}.

With \Cref{prop_nobigons} in hand, we can produce the fibration $p$: 

\begin{proposition}\label{prop:positive fibration}
  There exists a $\pi_1$-equivariant fibration $p\colon \wt N \to \mathring\PP$ whose
  restriction to each face of $\wt \tau$ is an orientation-preserving embedding into its
  associated rectangle.
\end{proposition}

The fibers of $p$ are then oriented lines and the quotient by $\pi_1$ yields an oriented
$1$-dimensional foliation positively transverse to $\tau^{(2)}$.

\begin{proof}
Fix an equivariant family of veering diagonals, as determined by \Cref{prop_nobigons}.
For each edge of $e$ of $\wt \tau$, map $e$ homeomorphically to the veering diagonal associated to its edge rectangle. We choose these maps to be equivariant with respect to the $\pi_1(M)$ action. 

If $f$ is a face of $\wt \tau$, then the edges in $\partial f$ are mapped to veering diagonals with disjoint interiors. Hence, we can equivariantly extend our map so that its restriction to each face is an orientation preserving embedding. Finally, we extend our map equivariantly over each tetrahedron of $\wt \tau$ as in \Cref{fig:tet_rect}. 
In particular, the fibers of the projection in each tetrahedron are compact intervals that degenerate to points at the angle--$0$ edges.
Using the local picture around faces and edges of $\wt \tau$, we see that the resulting
map $p \colon \wt N \to \mr {\mc P}$ is a fibration.
\end{proof}

\subsection{Step 2: The fiberwise map}\label{sec:step2}

We begin by compactifying $N$ to $\ol N$ and extending the foliation to the boundary
components. This will allow us to realize this extension of diagram (\ref{N hat diagram}):
\begin{equation}\label{three level diagram}
  \begin{tikzcd}
      & {\mc P} & \\
    \wh N\arrow[rr,"\wh h"]\arrow[d]\arrow{ur}{p}  &    & \wh M \arrow[d]\arrow[swap]{ul}{q}\\    \wc N \arrow[rr,"\wc h"]\arrow[d] & & \wc M \arrow[d]\\
\ol N\arrow[rr,"\ol h"] &    & \ol M 
  \end{tikzcd}
\end{equation}
Here, $\wc M$ is a renamed $\wt{\ol M}$, the universal cover of $\ol M$, while $\wh M$ is 
the completion of the universal cover $\wt M$, with the metric induced from the inclusion
$M\hookrightarrow \ol M$ and any fixed metric on $\ol M$.  
Note that $\wh M \to \wc M$ is an infinite branched covering,
to which the flow $\varphi$ lifts and that the components of the completion locus are the
preimages of singular orbits (we refer to these as the singular orbits of $\wh M$). Hence,
the map to the flow space $\wt M \to \mr {\mc P}$ extends equivariantly to a map $\wh M
\to \mc P$. 

On the left side of the diagram we compactify $N$ to $\ol N$ by adding torus boundary
components (done carefully below so as to extend the foliation by $p$-fibers). We then
let $\wh N \to \ol N$ be the universal cover, with the intermediate cover $\wc N\to \ol N$
obtained as the one associated to $\ker(\pi_1(\ol  N) \to \pi_1(\ol M))$.

The compactification of $N$ is carried out equivariantly on each tetrahedron
 $\kappa$ of $\wt N$. Each ideal vertex of $\kappa$ has a neighborhood of the form
$\Delta\times(0,1)$ where $\Delta$ is a cross-sectional triangle, and moreover the foliation
by $p$-fibers can be taken to be the same on each slice $\Delta\times\{t\}$. To see this
note that $p$ maps such a neighborhood to a region in $\mr{\mc P}$ bounded between two
veering diagonals, which can be written as an arc cross $(0,1)$. The $p$-preimage of each
arc is a cross-sectional triangle foliated by arcs. 

We can therefore
compactify $\kappa$ by adding a triangle $\Delta\times\{0\}$ for each
ideal vertex, so that a neighborhood of $\Delta$ is of the form $\Delta \times[0,1)$, and the
map $p$ extends to the added faces. 
We call the resulting polytope $\ol \kappa$ and refer to the added faces as cusp
triangles. Doing this equivariantly, these compactifed tetrahedra glue together along hexagonal faces so that the quotient is a
compactification $\overline N$ of $N$ which is homeomorphic to $N$ minus an open
neighborhood of each cusp. We denote by $\wh N$ the universal cover of $\ol N$ cellulated by the polytopes $\ol \kappa$ as above.

By construction, the foliation by $p$-fibers in $\wt N$ equivariantly extends to a foliation of $\wh N$ by lines and we refer to the leaves of this foliation as \define{$p$-leaves}.
We assign to each compactified tetrahedron $\ol \kappa$ a fixed
continuously varying metric along its $p$-leaf segments
which induces a continuously varying leafwise metric  in $\ol N$.

\subsubsection*{Defining an initial map}

We next construct a preliminary $\pi_1$-equivariant map $\wh h \colon \wh N \to \wh M$ whose restriction $\wt h \colon \wt N \to \wt M$ commutes with the projections to the flow space $\mr{\mc P}$. In fact, $\wh h \colon \wh N \to \wh M$ will commute with the natural projections to $\mc P$, as we will soon see.

The $p$-leaves in $\wt N$, together with their
orientations and metric as given above, can be identified with $\R$ up to translation, and
the same holds for the leaves of the flow in $\wt M$.

\begin{lemma}\label{lemma: regularity of h}
There exists a $\pi_1$-equivariant map $\wh h \colon \wh N \to \wh M$ whose restriction $\wt h \colon \wt N \to \wt M$ commutes with the projections to the flow space $\mr{\mc P}$ so that
for each $p$-leaf $\ell$ in $\wh N$, the restriction $\wh h|_\ell$, viewed as a map of
oriented lines, is a degree $1$ $(a,b)$-quasi-isometry, where $a$ and $b$ are independent of the leaf. 
\end{lemma}

\begin{proof}
We construct $\wh h$ successively on the skeleta of the completed 2-skeleton of $\wt\tau$.
For each  $\tau$-edge $ e$ of $\wt N$, the projection $p(e)$ in $\mr{\mc P}$ is a diagonal
whose closure in $\mc P$ is an arc $\ol {p(e)}$ with endpoints at singularities. The
restriction of the line bundle  $\wh M \to \mc P$ to $\ol {p(e)}$ is a trivial
bundle, so we can choose a lift  of $\ol {p(e)}$ (i.e. a section of the bundle) whose endpoints are on singular orbits of $\wh M$. Thus we have defined $\wh h$ on a closed edge of $\wh N$ whose interior is an edge of $\wt N$. We do this for an edge in each $\pi_1$-orbit and extend equivariantly. For any edge $c$ of a cusp triangle $\Delta$, we note that $\wh h(\partial c)$ are two points in a single singular orbit of $\wh M$ and so we extend $\wh h$ over $c$ by mapping it 
to the segment of the singular orbit joining its endpoints. This defines $\wh h$ on $\wh N^{(1)}$ in a $\pi_1$--equivariant way.  

The extension over $\tau$-faces of $\wh N$ is similar. 
For a face $F$ of $\wt \tau^{(2)}$, its compactification $\ol F$ in $\wh N$ is a hexagonal face of
$\wh N$. The embedding $F \to p(F)$ extends to a map $\ol F \to \ol {p(F)}$ that collapses
the edges of $\ol F$ contained in cusp triangles to the corresponding singular points of
$\mc P$. Pulling back the line bundle $\wh M \to \mc P$ to $\ol F$ under this map allows us to extend the section already defined on $\ol F^{(1)}$ to $\ol F$. We use this section of the pullback bundle to extend $\wh h$ over $\ol F$. Note that the restriction to $F$ commutes with the projections to $\mc P$ by construction. Again, we extend over a face in each $\pi_1$--orbit and extend equivariantly,
This defines $\wh h$ on the closures of the $\tau$-faces in $\wh N$. 

Finally, we extend the map $\wh h$ continuously to the $p$-leaf segments. 
That is, for any leaf segment $\alpha$ in a compactified tetrahedron $\ol \kappa$, its
endpoints $\partial \alpha$ are in the compactified part of the $2$-skeleton where $\hat
h$ has already been defined and commutes with the projections to $\mc P$. In particular
both endpoints of $\alpha$ map to the same leaf of the flow in $\wh M$, so that we may
extend $\wh h$ to a constant-speed map from $\alpha$ to that leaf. 
This completes the construction of $\wh h \colon \wh N \to \wh M$ with the required properties.

Equivariance means that $\wh h$ descends to a continuous map $\ol h \colon \ol N \to \ol M$,
which maps each boundary torus of $\ol N$ to a singular orbit of $\ol M$.
We then lift $\ol h$ to $\wc h: \wc N \to \wc M$, completing diagram (\ref{three level
  diagram}). 
 Properness of
 the deck group $\pi_1(\ol M)$ acting on $\wc M$ and compactness of $\ol N$ together
 imply that $\wc h$ is proper.  (Note by comparison that $\wh h$ is not proper -- the
 preimage of a singular leaf in $\wh M$ is a plane in $\wh N$, whereas the preimage of a
 singular leaf in $\wc M$ is an annulus in $\partial\wc N$; this is
 why we need $\wc h$).

We can now complete the proof of \Cref{lemma: regularity of h} by proving that $\wh h$ has
the required properties. 

\medskip\noindent{\em Coarse Lipschitz:} This follows immediately from compactness of $\ol N$ and
continuity of $\wh h$.

\medskip\noindent{\em  Uniform Properness:} Consider the lifted map $\wc h\colon \wc N \to \wc
M$.  Every leaf in $\wc M$ is
 properly embedded -- indeed as we know the universal cover is a product whose vertical
 factors are the leaves. The same is true in $\wc N$: here, by construction each leaf meets an infinite
non-repeating sequence of cells of the triangulation, and since these are discrete in 
$\wc N$ the leaf must be properly embedded. 
Now, since as above $\wc h$ is a proper map, its restriction to any leaf must be proper.

 Moreover, the map is {\em uniformly} proper on leaves: 
 For any point $z$ in $\wc M$
 let $K$ be a compact neighborhood. The preimage of $K$ is a compact set $K'$, so for all leaves
$m$ passing through $K$, the preimage of $m\cap K$ is contained in $K'$. This implies that
 the diameter in a leaf of the preimage of a segment in $K$ is uniformly bounded. 
After covering a  fundamental domain by finitely many such neighborhoods, we deduce that
for any leaf segment in $\wc M$ with diameter less than (say) $1$ there is a uniform bound
on the diameter of its preimage by $\wc h$. This implies uniform properness over all leaves.

\medskip\noindent{\em Degree $1$:}
 We now check that $\wc h|_\ell$ has degree $1$ for for every
 leaf. That is, we must check that the ``upward'' direction along leaves of $\wc N$
 maps to the ``upward'' direction along orbits of $\wc M$, at large scale.

The coorientations on the faces of the tetrahedra in $\wt N$ (from \Cref{sec:AG}) and the $p$-fibers were chosen so that when a
leaf passes from a tetrahedron $t$ to $t'$ in the forward direction, $t'$ lies above $t$,
and this means that the rectangle of $t'$ is above that of $t$ in the sense of \Cref{sec:flow_setup}: 
\begin{fact}\label{fact:order}
Suppose that $t$ and $t'$ are adjacent tetrahedra of $\wt N$ such that $t$ lies below $t'$ in the sense that there is an oriented $p$-leaf passing from $t$ to $t'$. Then the maximal rectangle associated to $t$ lies below the maximal rectangle associated to $t'$.
\end{fact}
This fact follows from considering the (finitely many) diagrams of a pair of adjacent
tetrahedra.  Now consider the sequence of tetrahedra that a forward ray of $\ell$
visits. Discreteness of the rectangles of $\mr{\mc P} $ implies that these rectangles must have
widths going to 0 and heights going to $\infty$, in the sense of \Cref{fact:non-accumulation}.

On the other hand, consider a leaf of the flow in $\wt M$. Fix an
equivariant collection of sections of $\wt M \to \mr{\mc P}$ over
maximal rectangles as in \Cref{lem:flow and rectangles}. 
The sequence of rectangles met by a forward flow ray must, by 
\Cref{lem:flow and rectangles}, be eventually ordered with later ones lying above earlier ones.

Thus, upward motion in the leaves of both $\wt N$ and $\wt M$ corresponds to the same behavior of
rectangles. Now to connect the two via $\wh h$, note that for each 
face $F$ of $\wt\tau$, the restriction of $\wh h$ to $F$ can be pushed along the flow in
$\wt M$ until it meets the selected section of the rectangle associated to $F$. The
distance along the flow required for this is uniformly bounded, since there are only
finitely many orbits of faces. It follows, using properness of the map on leaves, that, if an upward ray in $\wt N$ meets a
sequence of faces $F_1,\ldots,$ then its $\wt h$ image meets all but finitely many of the
associated rectangle sections. Thus, upward rays map to upward rays.  The same idea
applies to downward rays. 

\medskip\noindent{\em Quasi-isometry:} To finish, we need a lower bound on distances in the
image. Identifying both $\ell$ and $\wh h(\ell)$ with $\R$ in a length and orientation
preserving way, it suffices to prove the following: There exists $A>0$ independent of
$\ell$ such that for any $x,y\in\ell$
\begin{equation}\label{eqn: h lower bound}
y > x+A  \implies \wh h|_\ell(y) > \wh h|_\ell(x) + 1.
\end{equation}
Uniform properness implies that there exists $A>0$, independent of $\ell $ and $x$, such
that $y>x+A$ implies the distance between $\wh h|_\ell(y)$ and $\wh h|_\ell(x)$ is greater than 1. Degree
1 implies that, in fact $\wh h|_\ell(y)$ lies above $\wh h|_\ell(x)$ in the orientation
of the image leaf. This implies (\ref{eqn: h lower bound}). 
This (together with coarse-Lipschitz above) suffices to prove that $\wh h|_\ell$ is a
quasi-isometry. This completes the proof of \Cref{lemma: regularity of h}. 
\end{proof}

\subsection{Step 3: Straightening by convolution}
We are now ready to obtain the homeomorphism: 

\begin{proposition}\label{prop:straighten}
There is a $\pi_1$-equivariant orientation-preserving homeomorphism $\wt f : \wt N\to \wt M$ which commutes with
the fibrations $p \colon \wt N \to \mr{\mc P}$ and $q\colon \wt M \to \mr {\mc P}$. 
\end{proposition}

Once we have this, we'll denote by $f \colon N \to M$ the homeomorphism obtained by passing to the
quotients. It follows easily that $f(\tau^{(2)})$ is `topologically transverse' to
the flow $\phi$ -- see \Cref{prop:top_trans} below. Although this is all that we will need in practice, we will show in Step $4$ that this can be promoted to a smooth branched surface transverse to the flow.

\begin{proof}
The idea now is to convolve $\wt h|_\ell$ (for each leaf $\ell$) with a bump function to
get the desired map.

Let $A$ be the constant in property (\ref{eqn: h lower bound}) in the proof of
\Cref{lemma: regularity of h}. Let $\rho$ be a smooth bump function on $\RR$ satisfying: $\rho \ge 0$, $\int \rho =
1$, $\rho$ is supported on $\{|x|\le A+1\}$ and constant  on $[-A,A]$,  $\rho$ is even ($\rho(-x)
= \rho(x)$),  and increasing on $[-A-1,-A]$.

Now if $k:\R\to\R$ is a continuous map we form $\rho\star k(t) = \int k(y) \rho(t-y)
\,dy$, which has the following properties: 
\begin{enumerate}
  \item $\rho\star k$ is differentiable. 
\item $\rho\star $ commutes with translations. That is, if $T(x) = x+a$ then $\rho\star
  (f\circ T) = (\rho\star f)\circ T$, and $\rho\star(T\circ f) = T\circ(\rho\star
  f)$.
  \item $\rho\star$ is continuous with respect to the compact-open topology on
    $C(\R,\R)$.
    \item If $k$ satisfies property (\ref{eqn: h lower bound}) then $(\rho\star k)' > \ep
      >0$, where $\ep$ depends only on $\rho$. 
  \item If $k $ is $(K,\delta)$-coarse Lipschitz then $(\rho\star k)' < c$
    where $c$     depends on $\rho, K$ and $\delta$, 
      \item If $k $ is $(K,\delta)$-coarse Lipschitz then  $| \rho\star k - k| < c$,
    where $c$     depends on $\rho, K$ and $\delta$.
\end{enumerate}
Properties (1), (2) and (3) are standard. 
Properties (4) and (5)  follow from the fact that, given the properties of
$\rho$,
$$
(\rho\star k)'(t) = \int_A^{A+1} (k(t+u)-k(t-u)) |\rho'(u)| du
$$
which is easily verified. Property (6) is also a consequence of the averaging properties
of convolution.

Translation-invariance implies that this convolution operation is well defined on the maps
$\wt h|_\ell$, because our identification of the leaves with $\R$ is well-defined up to
translation. We let $\rho\star \wt h$ denote this operation carried out simultaneously on
all the leaves in $\wt N$. Continuity property (3), and the continuity of $\wt h$ and the
leafwise metrics that we chose in $\wt N$, imply that the result is a continuous map.

By \Cref{lemma: regularity of h}, each leafwise $\wt h|_\ell$ is coarse Lipschitz (with
uniform constants) and satisfies (\ref{eqn: h lower bound}). Thus, $\rho\star \wt h$ has
positive derivative on each leaf, so it is a homeomorphism on leaves, and it is a bounded
distance from $\wt h$ along the leaves (and in particular the two maps are
homotopic). On the leaf space the map is the identity, so it is globally a homeomorphism
from $\wt N$ to $\wt M$ which commutes with the projections to the leaf space. 

Finally, $\rho\star\wt h$ is equivariant: since the group acts by orientation-preserving isometries on the leaves
both in $\wt N$ and $\wt M$, this follows from equivariance of $\wt h$ and translation-invariance of $\rho\star$. 
This completes the proof.
\end{proof}

\bigskip

We can now state an immediate application of \Cref{prop:straighten}, which is the
topological version of our main result, \Cref{thm:flow transverse}: 

\begin{proposition}\label{prop:top_trans}
There is a homeomorphism $f\colon N \to M$, inducing the identity on $\pi_1(M)$,  such that the 
image of $\tau^{(2)}$ is a cooriented branched surface that is topologically positively transverse to flow lines of $\varphi$.
\end{proposition}

Here, by `topologically positively transverse,' we mean that the image of $\tau^{(2)}$ has a branched surface fibered neighborhood whose oriented fibers are segments of flow lines. 

\subsection{Step 4: Smoothing}
The next proposition completes the proof of \Cref{thm:flow transverse}.

\begin{proposition}\label{prop:smooth f}
The homeomorphism $f \colon N \to M$ from \Cref{prop:top_trans} can be chosen so that
$f(\tau^{(2)})$ is a smooth branched surface which is positively transverse to flow lines of $\varphi$. 
\end{proposition}

\begin{proof}
We will indicate how  the previous construction can be adjusted so as
to yield a smooth result. The first step is to give the line bundle $\wt N \to \mr{\mc P}$
a smooth structure with respect to which sections carried by the branched surface are smooth.

Since the flow is smooth off its singular orbits, the 
flow space $\mr{\mc P}$ inherits a smooth structure from $M$. 
After a small equivariant perturbation we may assume that the diagonals are smooth and
that triangles are still embedded. (The diagonals may no longer be transverse to the
stable/unstable foliations, but they are still contained in their respective edge
rectangles -- we will improve this in \Cref{prop:transversal diagonals}).

Next we need to specify the fiberwise metrics on the $p$-leaves so that they vary smoothly
with respect to the base. 
Each $p$-leaf is composed of segments from the foliation of the
tetrahedra. We can metrize these segments in each tetrahedron (equivariantly) so that
their lengths vary smoothly and converge to $0$ at the $0$-angle edges of the tetrahedron, and have
derivatives $0$ there (we can make higher derivatives match across the edge for greater smoothness).  

This allows us to give local trivializations of the bundle $p\colon \wt N \to \mr{\mc P}$:
Over a small disk consider a section that lies in the branched surface and use the metric
on leaves to define a trivialization where that section is at constant height. 
The way we chose the 
fiber segment lengths implies that different choices of sections give trivializations for
which transition maps are smooth. Thus we have a smooth structure on the bundle for which sections lying in
the branched surface are smooth.

When we define the map $\wh h$, we first choose sections over the veering
edges. These can be chosen smoothly on the interiors of the edges. We then extend to the faces of $\wt{\tau}^{(2)}$ smoothly,
and in such a way that the sections are tangent to each other at the edges. 
The extension of $\wh h$ to the $p$-leaf segments in each tetrahedron can be done at
constant speed, so that since the maps on the endpoints are smooth by the previous
paragraph we find that, in a local trivialization of $p$, 
the map varies smoothly with respect to the coordinates in $\mr{\mc  P}$. Note that we do
not obtain smoothness of $\wh h$ on the completion points of $\wh N$, but we only need
continuity there. 

The map, which may still not be injective or smooth in the fiber direction, is now averaged in the convolution step. The
final map is smooth in the fiberwise direction because the bump function $\rho$ is smooth,
and it is smooth in the $\mr{\mc P}$ direction because the fiberwise metrics and the map
$\wt h$ are smooth with respect to the $\mr{\mc P}$ direction. Thus our final map is a
diffeomorphism and the image of the branched surface is transverse to the flow. 
\end{proof}

\subsection{Transversality in the flow space}
\label{sec:trans_flow}
In the interest of recovering as much as possible of the picture in the suspension flow
case, we would also like the smooth veering edges in each edge rectangle in $\mr{\mc P}$ to be
transverse to both stable and unstable foliations. We note that this is not needed for
the flow-transversality of \Cref{prop:smooth f}.

\begin{proposition}\label{prop:transversal diagonals}
The fibration $p\colon \wt N \to \mr{\mc P}$ can be chosen so that the images of the veering edges are smooth and transversal to both the stable and unstable foliations.
\end{proposition}

We note that this is easy if the stable and unstable foliations are at least $C^2$,
because then the rectangles can be smoothly identified with Euclidean rectangles foliated
by axis-parallel lines. Our foliations may not have this regularity, although for
dynamical reasons they do have smooth leaves and line fields which are uniquely integrable. These facts are well-known for Anosov flows
\cite{Anosov} and the proofs also apply more generally to pseudo-Anosov flows (see \cite{fenley2001quasigeodesic}).
This turns out to be enough.

The Proposition will follow directly from \Cref{cor:approximate with smooth diagonal}
below.

\medskip

In this section, by {\em smooth} we mean at least $C^2$. A {\em smooth quadrilateral} is a smooth disk-with-corners that has four corners and two
transverse foliations, so that each foliation is tangent to two opposite boundary
edges. We do {\em not} assume that the foliations themselves are smooth.

Such a quadrilateral has a diffeomorphism to the unit square, taking the two foliations to
foliations that include the horizontal and vertical boundary edges, respectively. From now
on we identify $Q$ with $[0,1]^2$, we call the foliations $F_h$ and $F_v$, and we say that a {\em diagonal} of $Q$ is a path from $(0,0)$ to $(1,1)$. 

\begin{lemma}\label{lem:smooth diagonal}
Let $Q$ be a smooth quadrilateral. If the line fields of $F_h$ and $F_v$ are uniquely
integrable, then there exists a smooth diagonal which is transverse to both $F_h$ and
$F_v$. Moreover one can prescribe the tangent direction of the diagonal at each of its
endpoints. 
\end{lemma}
We remark that the lemma is false without the unique integrability assumption, so that
there really is something to do here. 

As a corollary we have:

\begin{corollary}\label{cor:approximate with smooth diagonal}
  Let $Q$ be a smooth quadrilateral for which the line fields of $F_h$ and $F_v$ are
  uniquely integrable, and $\alpha$ a continuous diagonal which is topologically
  transverse to both of the foliations. Then $\alpha$ can be  $C^0$ approximated by a
  smooth diagonal transverse to both foliations. 
\end{corollary}
Here ``topologically transverse'' means that the diagonal meets every foliation leaf
exactly once. 

\begin{proof}[Proof of \Cref{cor:approximate with smooth diagonal}]
Subdivide $\alpha$ into small segments. Because it is topologically transverse, the
endpoints of every segment are opposite corners for a foliated sub-quadrilateral, so that
adjacent quadrilaterals meet exactly at their common corner. We may choose tangent lines
at every corner which point into the two adjacent quadrilaterals, and then use
\Cref{lem:smooth diagonal} to find a smooth diagonal for each quadrilateral matching the
given tangent direction at the corners. These piece together to a $C^1$ diagonal which
is transverse to both foliations and closely approximates $\alpha$ in $C^0$. A further
(standard) smoothing step upgrades this to a smooth diagonal. 
\end{proof}

\begin{proof}[Proof of \Cref{lem:smooth diagonal}]
We will need this calculus lemma, whose proof we omit:
\begin{lemma}\label{lem:calc}
Let $\alpha_n:[0,1] \to \R^d$ be $C^1$ curves. Suppose that $\alpha_n \to \alpha$
pointwise, and $\alpha'_n$ converges to a continuous vector field $u$ along $\alpha$.
Suppose moreover that all the functions $t\mapsto\alpha'_n(t)$ have a common modulus of
continuity.  Then
$\alpha$ is differentiable and $\alpha' = u$.
\end{lemma}

Now we apply this to our setting. 
Let $u,v$ be $C^0$ vector fields on $Q$ that are tangent to $F_h$ and $F_v$, respectively.
Extend $u,v$ continuously to a small neighborhood of $Q$,
and form the open tangent cone field $C$ where $C(x) = \{au(x)+bv(x) : a,b>0\}$. Given any fixed
$a,b>0$, the vector field $au+bv$ lies in $C$, and varying over convex combinations
$a+b=1$ we obtain a family of vector fields with a common modulus of continuity.

Let $\rho_a(x)$ be a family of bump functions varying smoothly with $a\in(0,1)$, with mass
1 and support of size $\ep(a)$, such that
$\ep\to 0$ smoothly as $a\to 0$ or
$a\to 1$.
If the function $\ep$ is chosen small enough
then, convolving $au+bv$ with $\rho_a$, we
get a family of smooth vector fields $\xi_a$ on $Q$ in the cone field $C$, all with a common modulus of
continuity, such that $\xi_a \to u$ as $a\to 1$ and $\xi_a \to v$ as $a\to 0$.

Now for each $a\in (0,1)$, smoothness implies $\xi_a$ is uniquely integrable 
so let $\alpha_a$ be an integral
curve starting at the lower-left corner of $Q$. Thus $\alpha_a$
satisfies
$\alpha_a'(t) = \xi_a(\alpha_a(t)) \in C(\alpha_a(t))$, for any $t$ for which the curve is
defined. In fact (since $\xi_a$ is smooth) this is defined until it leaves $Q$, and this
must be on the right or top edge since $C$ points into $Q$ at points of the left and
bottom edges. 

Now we can take the limit as $a\to 0$. Because $\alpha_a$ have bounded derivatives,
Arzel\`a-Ascoli gives us some sequence $a_n\to 0$ for which the curves converge to some
limit curve $\alpha_0$. We know that the vector fields $\alpha'_a(t)$ along the curves
satisfy a common modulus of continuity because $\xi_a$ do, and the $\alpha_a$ have bounded
speed. Thus \Cref{lem:calc} applies to tell us that $\alpha_0$ is differentiable and its
derivative is just the limit of $\xi_a$ (restricted to $\alpha_0$) which is the vector
field $v$. That is, $\alpha_0$ is an integral curve of $v$ starting at the lower left
corner, and hence the left-boundary leaf of the foliation $F_v$ by unique integrability. 
This means $\alpha_0$
terminates on the upper left corner.

Similarly a limit as $a \to 1$ gives $\alpha_1$ which terminates on the lower right
corner.

Continuity gives us a value of $a$ for which $\alpha_a$ terminates on the upper-right
corner.

\medskip

Once we have the desired path $\alpha_a$, we need to perturb it so that it has the desired
tangent directions at the corners. This is more simple: the vector fields $u$ and $v$ are
continuous at the corners. So in a small enough neighborhood of (say) the corner $(0,0)$
they are much closer to the coordinate vector fields than they are to the direction of
$\alpha_a$. Now, thinking of $\alpha_a$ as the graph of a function, add a smooth function
with small support near 0 and the appropriate derivative at 0.
\end{proof}

%%%%%%%%%%%%%%

%% !TEX root =veering_poly2.tex

\section{The flow graph and orbits of the flow}\label{sec:flowandgraph}

Let $\phi$ be a pseudo-Anosov flow without perfect fits on $\ol M$ and let $\tau$ be the veering triangulation of $M = \ol M \ssm \{\text{singular orbits}\}$ dual to $\phi$. 
In this section we detail how the flow graph $\Phi$ uniformly
codes orbits of the flow $\phi$. 

Recall that $\mc O_\phi^+ \subset M$ is the union of closed orbits $\mc O_\phi$ of the flow along with all positive multiples of unstable prong curves in $M$. 
We also denote by $\mathcal{Z}_{\Phi}$ the set of directed cycles of $\Phi$. 
In \Cref{sec:cycles_orbits} we produce a map
\[
 {\mf F} \colon \mathcal{Z}_{\Phi} \to \mc O_\phi^+,
 \]
with the property that the directed cycle $c$ is homotopic to ${\mf F}(c)$ in $M$. We remark that when ${\mf F}(c)$ is nonsingular (i.e. a closed orbit in $M$; not a prong curve), it is the unique closed orbit of $\phi$ homotopic to the flow cycle $c$. 

The main theorem of this section (proven in \Cref{sec:cycles_orbits}) is a summary of the essential features of the map ${\mf F}$. For its statement, we need one additional definition. Let $\gamma_1$ and $\gamma_2$ be two directed closed curves in $M$ which are positively transverse to $\hbs$. We say that $\gamma_1$ and $\gamma_2$ are \define{transversely homotopic} if they are homotopic through closed curves that are positively transverse to $\hbs$.

We also remind the reader that $\delta_\tau$ denotes the length of the longest fan in $\tau$.

\begin{theorem}[Closed orbits and the flow graph]
\label{th:closed_orbits}
The map  ${\mf F} \colon \mc Z_\Phi \to \mc O_\phi^+$ has the following properties:
\begin{enumerate}
\item The image $\gamma = {\mf F}(c)$, which is either a nonsingular closed orbit of $\phi$ or an unstable prong curve, is \emph{transversely homotopic} to $c$ in $M$.
\item For each unstable prong curve $\gamma$, $1 \le \# \mf F^{-1}(\gamma) \le 2$.
\item For each nonsingular closed orbit $\gamma$ of $\phi$, either 
$\# \mf F^{-1}(\gamma) \le 1$ or
$\gamma$ is homotopic to an $AB$-cycle in which case
$\# \mf F^{-1}(\gamma) \le \delta_\tau$.
\item For each nonsingular closed orbit $\gamma$ of $\phi$, either $\gamma$ is in the image of $\mf F$ or $\gamma$ is homotopic to an odd AB-cycle. 
\end{enumerate}
\end{theorem}

In short, the flow graph $\Phi$ encodes all but finitely many primitive orbits of the flow in a one-to-one fashion. We remark that simple cycles of $\Phi$ can map to nonprimitive orbits of $\phi$; this happens for example in the presence of twisted orbits.

\subsection{The flow space and the flow graph}
We begin by explaining how the structure of $\Phi$ is recorded by the maximal rectangles of the completed flow space $\mc P$. 

By the construction of $\tau$, each $\tau$-edge $e$ of the lifted triangulation $\wt \tau$ on $\wt M$ corresponds to a unique edge rectangle in $\mc P$. Similarly, faces of $\wt \tau$ correspond to face rectangles and tetrahedra correspond to maximal rectangles.
In fact, we can use \Cref{prop_nobigons} to fix a $\pi_1(M)$-equivariant map $\wt M \to \mc P$ that embeds each $\tau$-edge in its edge rectangle so that the restriction to each face of $\wt \tau$ is also an embedding.
We fix such a map once and for all, and, abusing terminology, we will also refer to the image of $e$ in $\mc P$ as a \define{$\tau$-edge}.
For example, a $\tau$-edge $e$ is contained in a maximal rectangle $R$ if and only if its edge rectangle $Q$ is contained in $R$. The singularities at the corners of $Q$ are necessarily contained in the interiors of the sides of $R$. 

\begin{figure}[htbp]
\begin{center}
\includegraphics{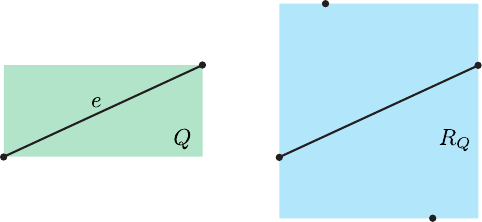}
\caption{Vertically extending an edge rectangle (left, green) to a maximal rectangle (right, blue).}
\label{fig:ext_rectangles}
\end{center}
\end{figure}

If $Q$ is the edge rectangle for a $\tau$-edge $e$, then we 
denote by either $R_e$ or $R_Q$ the maximal rectangle obtained by extending $Q$ 
vertically along leaves of $\mc F^s$ as far as possible, so that 
$e$ joins its vertical sides. See \Cref{fig:ext_rectangles}.
In terms of the veering triangulation $\wt \tau$ of $\wt M$, $R_e$ is the maximal rectangle corresponding to the tetrahedron having $e$ as its bottom edge.

\smallskip

We also fix the inclusion $\iota \colon \Phi \to M$ in dual position. In this section, it will be convenient to identify $\Phi$ with its image under $\iota$. Recall that in this position, the vertices of $\Phi$ agree with the vertices of the dual graph $\Gamma$ and hence with the triple points of the stable branched surface $B^s$.

By \cite[Lemma 4.4]{LMT20}, $\iota \colon \Phi \to M$ is $\pi_1$-surjective and so
the flow graph $\Phi$ has connected preimage in the universal cover of $M$, which we denote by $\wt \Phi$. 
Identifying each $\tau$-edge 
$e$ with the maximal rectangle $R_e$ leads to the following alternative description of $\wt \Phi$: the vertices of $\wt \Phi$ are maximal rectangles and for each maximal rectangle $R$ there are directed edges from $R=R_b$ to the three rectangles $R_t, R_{s_1}, R_{s_2}$, where $t$ is the $\tau$-edge joining the horizontal sides of $R$ and $s_1$ and $s_2$ are $\tau$-edges of $R$ such that the rectangles $R_t, R_{s_1}, R_{s_2}$ have nonoverlapping interiors  (see \Cref{fig:flowgraph_rects}). 
This is to say that in the rectangle $R$, the set $t\cup s_1\cup s_2$ passes the ``vertical line test."
We will freely use this correspondence between $\wt \Phi$-vertices and maximal rectangles.

\begin{figure}[htbp]
\begin{center}
\includegraphics{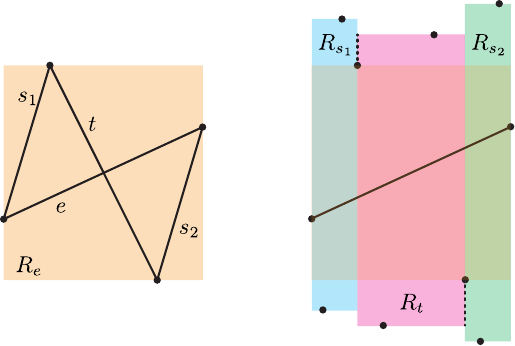}
\caption{The outward $\wt \Phi$-edges of the $\tau$-edge $e$ in terms of maximal rectangles. The dotted lines indicate that those portions of the boundaries of rectangles do not meet.
}
\label{fig:flowgraph_rects}
\end{center}
\end{figure}

We similarly consider the preimage $\wt \Gamma$ of the dual graph $\Gamma$ in the universal cover of $M$. Each vertex of $\wt \Gamma$ is contained in the interior of a unique tetrahedron of $\wt M$ and hence also corresponds to a unique maximal rectangle. Understanding paths in $\wt \Gamma$ with this perspective is fairly straightforward:

\begin{lemma}\label{lem:relheights}
Let $R_1$ and $R_2$ be maximal rectangles in $\PP$ 
and suppose that $R_2$ lies above $R_1$. Let $v_1$ and $v_2$ be the $\wt \Gamma$-vertices
corresponding to $R_1$ and $R_2$. Then there exists a directed $\wt\Gamma$-path from $v_1$ to $v_2$.
\end{lemma}

\begin{proof}
Let $T_1$ and $T_2$ be the $\wt \tau$-tetrahedra corresponding to $R_1$ and $R_2$, respectively. Since $R_2$ lies above $R_1$, the projections of $T_1$ and $T_2$ to $\PP$ must overlap.
Hence there is a $\wt\phi$-orbit passing through both $T_1$ and $T_2$. Further observe 
that whenever an orbit passes from a tetrahedron $T_a$ to an adjacent tetrahedron $T_b$, the maximal rectangle $R_a$ associated to $T_a$ lies below the maximal rectangle $R_b$ associated to $T_b$ (see \Cref{fact:order}).
Hence, the given orbit must pass first through the tetrahedron $T_1$ and then through the tetrahedron $T_2$. 

By truncating this orbit and adding small segments in $T_1$ and $T_2$, we obtain a path from $v_1$ to $v_2$ which is positively transverse to $\wt\tau^{(2)}$. After perturbing rel endpoints to make it disjoint from $\wt \tau^{(1)}$, the sequence of $\wt \tau$-faces traversed by this path corresponds to a directed $\wt\Gamma$-path from $v_1$ to $v_2$.
\end{proof}

To understand directed paths in $\wt \Phi$ it is convenient to work with the dynamic planes of \Cref{sec:dynamicplanes}, as we now explain.

A \define{singular leaf} of either the stable or unstable foliation of $\mc P$ is a leaf homeomorphic to $[0,\infty)$ with its endpoint on a singularity of $\mc P$. A point $p$ of $\mc P$ is a \define{regular point} if it does not lie in a singular leaf of either foliation. We remark that all fixed points of $\mc P$ under the $\pi_1(M)$-action are either regular or singular, since singularities are the only fixed points in their stable/unstable leaves.

Now let $p\in \mc P$ be a regular point.  A \define{$p$-rectangle} or \define{maximal $p$-rectangle} is a rectangle or maximal rectangle, 
respectively, which contains $p$ in its interior. A \define{$p$-ray} is a directed infinite ray in $\wt \Phi$ traversing only maximal $p$-rectangles.

\begin{lemma} \label{lem:unique_rays}
Let $R$ be a maximal $p$-rectangle for a regular point $p\in \mc P$. There is a unique $p$-ray starting at $R$.
\end{lemma}

\begin{proof}
By definition, there are directed edges from $R=R_b$ to $R_{s_1},R_{s_2},R_t$, where $R_{s_1},R_{s_2},R_t$ cover $R$ and have disjoint interiors. 
Since $p$ is a regular point, it is interior to exactly one of $R_{s_1},R_{s_2}, R_t$. In other words, every maximal $p$-rectangle has a unique outgoing $\wt \Phi$-edge connecting it to another maximal $p$-rectangle.
\end{proof}

In the next proposition, we associate to each (singular) leaf $\ell$ of $\mc F^s$ a unique dynamic (half-) plane $D_\ell$. As in \Cref{sec:width}, we denote by $\sigma(v)$ the unique sector of $\wt B^s$ whose top vertex is $v$.
If $R$ is the maximal rectangle of $\mc P$ corresponding to the vertex $v$ of $\wt \Gamma$, we extend this notation to $\sigma(R) = \sigma(v)$. The reader can check that if $e$ is an edge of $\wt \tau$, then $\sigma(R_e)$ is the unique sector of $\wt B^s$ dual to $e$.

\begin{proposition}[Dynamic planes for stable leaves]
\label{prop:planes_for_leaves}
Let $\ell$ be a (singular) leaf of the vertical foliation $\mc F^s$. The union 
\[
D_\ell  = \bigcup_{\ell \cap \mathrm{int}(R) \neq \emptyset} \sigma(R)
\]
of all sectors of $\wt B^s$ associated to maximal rectangles $R$ that meet $\ell$ in their interior is a dynamic (half-) plane. 

Moreover, $D_\ell$ has the property that for any dual ray (or flow ray) $\wt \gamma$ whose vertices correspond to maximal rectangles that meet $\ell$ in their interior, we have $D_\ell = D(\wt \gamma)$.
\end{proposition}

For the proof, we first define the dynamic (half-) plane associated to any increasing sequence of rectangles. Let $A=(A_1, A_2, A_3,\dots)$ be any sequence of distinct maximal rectangles with the property that $A_{i+1}$ lies above $A_i$ for all $i$. 
We remark that $A$ could be a path in $\wt \Gamma$, in $\wt \Phi$, or in neither, though only the $\wt \Gamma$ and $\wt \Phi$ cases are relevant for us.

By \Cref{lem:relheights}, there is a $\wt\Gamma$-path from $A_i$ to $A_{i+1}$ for all $i$. The union of these $\wt\Gamma$-paths gives a (possibly nonunique) $\wt \Gamma$-ray $\gamma_A$. 
It follows that $A$ determines a dynamic (half-) plane $D_A$, which can be defined by
\[
D_A= D(\gamma_A) = \bigcup_{i} \Delta(\sigma(A_i)),
\]
where $\Delta(\sigma)$ is the descending set of $\sigma$, is as in \Cref{sec:width}.
Note that the $\Delta(\sigma(A_i))$ form a nested union of descending sets by \Cref{lem:pushdown} and so $D_A$ is independent of the choice of $\gamma_A$. Also, $D_A$ is a dynamic half-plane if and only if $\gamma_A$ is eventually a branch ray. In terms of rectangles, this is equivalent to the condition that there is a single singularity $s$ such that either the top or bottom components of all $\partial_h A_i$ eventually contain $s$.
To see this, first note that by \Cref{lem:relheights} it suffices to assume that $A$ is a sequence of consecutive maximal rectangles in the sense that $A_i$ and $A_{i+1}$ intersect along a face rectangle $F_i$. Then from the picture in the flow space, one sees that there exists a singularity $s$ eventually contained in all $\partial_h A_i$ if and only if it is eventually the case that the edge rectangle of intersection between $F_i$ and $F_{i+1}$ always has the opposite veer as the top edge rectangle of $A_{i+1}$. This exactly characterizes branch rays; see for example \cite[Lemma 4.5]{LMT20}.

\begin{proof}[Proof of \Cref{prop:planes_for_leaves}]
First fix an arbitrary maximal rectangle $A_0$ such that $\sigma(A_0)$ is a sector of $D_\ell$, and a sequence $A=(A_0, A_1, A_2, A_3,\dots)$ of distinct maximal rectangles meeting $\ell$ in their interiors with the property that $A_{i+1}$ lies above $A_i$ for all $i$. Note that $\ell$ is a singular leaf with singularity $p$ if and only if $A_i$ contains $p$ in its horizontal boundary for $i$ sufficiently large. We will show that $D_\ell = D_A$.

Next let $B_0$ be any maximal rectangle with $\sigma(B_0)$ a sector of $D_\ell$ and as above let $B=(B_0, B_1,B_2,B_3,\dots)$ be another such sequence of maximal rectangles, determining a dynamic plane $D_B$ in the same way. We claim that $D_A=D_B$. Indeed, let $A_i$ be a term of $A$. By the discreteness of singularities in $\PP$, there exists some $j$ such that $B_j$ lies above $A_i$ (see \Cref{fact:non-accumulation}). By \Cref{lem:relheights} there is a dual path from $A_i$ to $B_j$, so by \Cref{lem:pushdown} we have $\Delta(\sigma(A_i))\subset \Delta(\sigma(B_j))$. Since this holds for any rectangle in $A$, we see that $D_A\subset D_B$. The proof of the reverse inclusion is the same, so we have the equality $D_A=D_B$. 

Since $B_0$ was an arbitrary maximal rectangle with $\sigma(B_0) \subset D_\ell$, we conclude that $D_\ell \subset D_A$. Moreover, since any rectangle $R$ with $\sigma(R) \subset D_A$ lies below some $A_i$, we also have the interior of $R$ meets $\ell$. Hence, $D_A \subset D_\ell$. This proves that $D_\ell$ is a dynamic (half-) plane. 
The moreover claim is now clear from the construction of $D_A$ and the equality $D_A = D_\ell$.
\end{proof}

For any point $p \in \mc P$ that is not a singularity, we further define $D_p = D_\ell$ where $\ell$ is the unique leaf of $\mc F^s$ through $p$. The planes $D_\ell$ and $D_p$ are called the \define{dynamic planes} for $\ell$ and $p$, respectively.

\begin{remark}[Singular leaves and dynamic planes]\label{rmk:singularplanes}
If $\ell$ is leaf of $\mc F^s$, then $\ell$ is singular if and only if $D_\ell$ is a dynamic half-plane. Moreover, in this case, any increasing sequence of rectangles $(A_i)$ whose terms all intersect $\ell$ in their interior has the property that the singular point $p$ along $\ell$ is eventually contained in the horizontal boundary of $A_i$ for sufficiently large $i$.

\begin{figure}[h]
\centering
\includegraphics{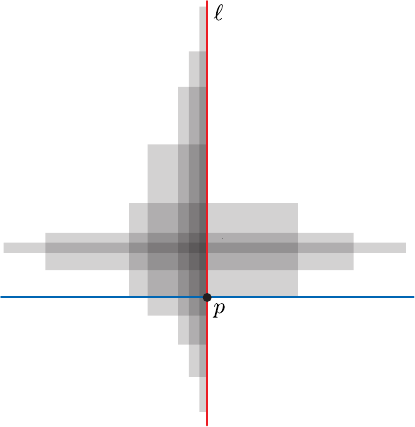}
\caption{
A singular leaf $\ell$ with singularity $p$ and an increasing sequence of maximal rectangles that intersect $\ell$ with the property that $p$ is eventually contained in the vertical boundary of all terms. The resulting dynamic plane contains the dynamic half-plane $D_\ell$.}
\label{fig:pickaside}
\end{figure}

In this situation, we can pick one of two sides of $\ell$ and extend along maximal rectangles which are contained in that side and which contain $p$ in thier vertical boundary 
(see \Cref{fig:pickaside}). The argument from the proof of \Cref{prop:planes_for_leaves} shows that doing so produces a unique dynamic plane containing $D_\ell$. Since $\ell$ has two sides, 
this procedure produces two dynamic planes that contain $D_\ell$. (In fact, these are the only two dynamic planes containing $D_\ell$, but we will not need this.)
Moreover, each of these dynamic planes is stabilized by the (necessarily nontrivial) stabilizers of $D_\ell$.
\end{remark}

\Cref{prop:planes_for_leaves} implies that for each regular point $p\in \PP$, any two $p$-rays lie in the dynamic plane for $p$. Combining with \Cref{cor:width} gives us the following lemma.

\begin{lemma} \label{lem:numb_classes}
Let $p\in \mc P$ be a regular point with dynamic plane $D$. The number of asymptotic classes of $p$-rays is equal to the width of $D$,
and hence is at most $\delta_\tau$.
\end{lemma}

\smallskip

With these results relating dynamic planes to the flow space in hand, we can now characterize when closed paths in $M$ are transversely homotopic. This will be essential for the results in \Cref{sec:growth}.

\begin{proposition}[Transverse homotopies]\label{lem:transversely_homotopic}
Let $\gamma_1$ and $\gamma_2$ be two homotopic closed curves which are positively transverse to $\hbs$. Then either they are transverse homotopic or they are homotopic to branch curves.
\end{proposition}

\begin{figure}[h]
    \centering
    \includegraphics{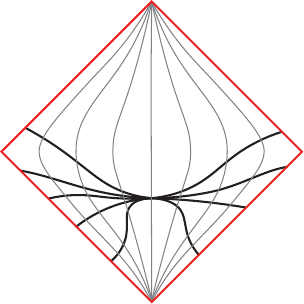}
    \caption{Homotoping one side of a $B^s$-sector to another through curves positively transverse to $\hbs$.}
    \label{fig:sidemove}
\end{figure}

\begin{proof}
We first perturb $\gamma_1$ and $\gamma_2$ to avoid $\tau^{(1)}$. Then each $\gamma_i$ determines a unique $\Gamma$-cycle to which it is transversely homotopic. 
Hence, it suffices to prove the claim when $\gamma_1$ and $\gamma_2$ are $\Gamma$-cycles. Assume that neither is homotopic to a branch curve.

Lifting a homotopy from $\gamma_1$ to $\gamma_2$, we obtain $\Gamma$-lines $\wt \gamma_1$ and $\wt \gamma_2$ that are stabilized by $\langle g \rangle \le \pi_1(M)$.  If we intersect the sequence of maximal rectangles associated to $\wt \gamma_1$ (or $\wt\gamma_2$), we obtain a single $p\in \mc P$ by \Cref{fact:non-accumulation}. By construction, $p$ is stabilized by $g$ and so $p$ is regular. Otherwise $p$ would necessarily be a singularity of $\mc P$ and $\gamma_1$ would be homotopic to a branch curve in the corresponding cusp of $M$.
 
Therefore, $p$ determines a dynamic plane $D_p$ that contains $\wt \gamma_1$ and $\wt \gamma_2$ by \Cref{prop:planes_for_leaves}. Applying \Cref{lem:trans_hom}, we see that $\gamma_1$ and $\gamma_2$ are homotopic by sweeping across sectors of $B^s$. Since such homotopies are visibly through curves that are transverse to $\tau^{(2)}$ (see \Cref{fig:sidemove}), the proof is complete.
\end{proof}

\subsection{Lines of $\wt \Phi$ and the flow}
\label{sec:lines}

We now focus on associating to each directed line of the graph $\Phi$ an orbit of the flow $\phi$. 
More precisely, we define a map 
\[
\wt {\mf F} \colon \{\text{directed lines in } \wt \Phi \} \to \mc P
\]
from directed lines in $\wt \Phi \subset \uM$ to the
completed flow space $\mc P$.
Each directed line $\wt\gamma$ in $\wt \Phi$ corresponds to a sequence of maximal rectangles which become taller in the positive direction and wider in the negative direction. Then, as in \Cref{lem:transversely_homotopic}, \Cref{fact:non-accumulation} implies that
the intersection of the rectangles in this sequence is a single point $\wt {\mf F}(\wt\gamma) \in \mc P$. 
See \Cref{fig:flowmap}.
It is not hard to see that this map is $\pi_1(M)$-equivariant and continuous with respect to the usual topology on the space of lines.

\begin{figure}[h]
\centering
\includegraphics{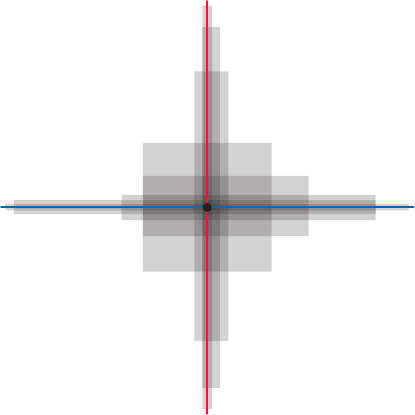}
\caption{The map $\wt{\mf F}$ sends a $\wt\Phi$-line, which corresponds to a certain bi-infinite sequence of maximal rectangles, to the unique point in $\PP$ lying in all of the rectangles.}
\label{fig:flowmap}
\end{figure}

To understand the image of $\wt {\mf F}$ in $\mc P$, we again study the structure of dynamic planes. 

\begin{figure}[h]
\begin{center}
\includegraphics{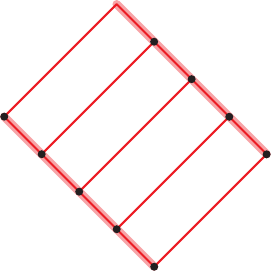}
\caption{A chain of length four. The highlighted branch segments are long, and the other two branch segments in the boundary of the chain are short. The vertices indicated by black dots all have the same veer by \Cref{lem:sectorveer}. Either of the two bottom branch segments could possibly have more vertices than shown.}
\label{fig:chain}
\end{center}
\end{figure}

Define a \define{chain of sectors} in a dynamic plane $D$ to be a union of sectors attached as in \Cref{fig:chain}.
More precisely, a chain is a collection of sectors $\sigma_1,\dots, \sigma_n$ such that an entire bottom branch segment of $\sigma_i$ is identified with a top branch segment of $\sigma_{i+1}$ for $i = 1,...,n-1$, and there is a single branch segment that contains a top branch segment of each $\sigma_i$ for $i = 1,...,n$, i.e.
the union $\bigcup_{i=1}^n \sigma_i$ is bounded by four branch segments. 
Two of these are the \define{top} branch segments of the chain and two are the \define{bottom}. 
Note that every sector has two chains, possibly of length $1$. When a chain has length at least $2$, we say a branch segment in its boundary is \define{long} if it contains an edge of each sector of the chain and is \define{short} otherwise. See \Cref{fig:chain} for an example.

We now show that lengths of chains are uniformly bounded by $\delta = \delta_\tau$, which as a reminder is the length of the longest fan of $\tau$.

\begin{lemma}\label{lem:chaintofan}
Any chain of sectors in a dynamic plane has length less than $\delta$.
\end{lemma}

\begin{proof}
Suppose that $C$ is a chain of length $k\ge 2$. An application of \Cref{lem:sectorveer} shows that the bottom $k-1$ sectors of $C$ have top and bottom vertices of the same veer. This means that the long top branch segment of $C$ passes through $k-1$ consecutive non-hinge tetrahedra. Applying \Cref{claim:branching} gives that these tetrahedra lie in the fan of a single edge. As in the proof of \Cref{lem:ABregion}, we note that this implies the existence of a fan of length $k-1+2=k+1$ (see \cite[Observation 2.6]{futer2013explicit}). 
\end{proof}

Recall that \Cref{lem_rayconvergence} says that flow rays converge in dynamic planes unless they are separated by AB strips. 
The next lemma essentially says that this convergence is rapid, and is key to proving \Cref{prop:flowline_bound} which says that $\wt{\mf F}$ is surjective and uniformly finite-to-one away from singularities.

\begin{lemma} \label{lem:speed_converge}
Let $\sigma$ be a sector of a dynamic plane $D$, and let $\sigma'$ be the sector directly below $\sigma$ so that the top vertex of $\sigma'$ is the bottom vertex of $\sigma$. Then any flow ray of $D$ starting in the descending set $\Delta(\sigma') \subset D$ passes through 
a vertex in a chain of $\sigma$.
\end{lemma}

\begin{proof}
Since each flow ray of $D$ starting in $\Delta(\sigma')$ eventually meets $\partial \Delta(\sigma')$ by \Cref{lem_dstructure}, we may suppose that the flow rays in question start at vertices along $\partial \Delta(\sigma')$. Moreover, also by \Cref{lem_dstructure}, 
$\partial \Delta(\sigma')$ consists of the negative subrays $b_1$ and $b_2$ of the branch lines through the top vertex $v'$ of $\sigma'$. We will prove the claim for flow rays of $D$ starting at $b= b_1$ since the proof for $b_2$ is the same.

Let $S$ be the union of all sectors in $\Delta(\sigma) \ssm \mathrm{int}(\Delta(\sigma')$ with bottom vertex lying in $b$. Note that one of the chains of sectors of $\sigma$ is contained in $S$ and that every sector in $S$ other than $\sigma$ has a segment of $b$ as a complete branch segment in its boundary. See \Cref{fig:sawblade}. 

\begin{figure}[h]
\begin{center}
\includegraphics{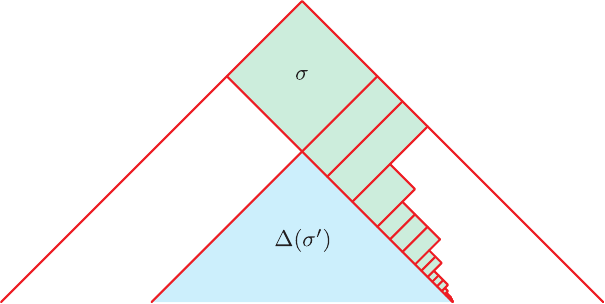}
\caption{The set $S$ (green) as a union of chains of sectors.}
\label{fig:sawblade}
\end{center}
\end{figure}

Let $C_0, C_1, \ldots $ be the decomposition of $S$ into a union of chains of sectors so that $C_0$ is the chain of sectors of $\sigma$ in $S$ and the top (short) branch segment of $C_{i+1}$ whose initial vertex is along $b$ is identified with a proper branch segment along the bottom of $C_i$. The remainder of the proof will establish that any flow ray starting at a vertex along $b$ passes through vertices of $S$ until it exits $\Delta(\sigma)$ at some vertex along $C_0$, the chain for $\sigma$ in $S$. The key technical step is the following claim which implies that the long top branch segment in each chain $C_i$ (for $i\ge 1$) is contained in a side of a single sector of $\Delta(\sigma)$.

\begin{claim}
\label{claim:chains_and_sectors}
Suppose the short top branch segment of a chain $C$ of sectors in the dynamic plane $D$ is identified with the lowermost edge in the side of some sector $\sigma_a\subset D$ and that $\sigma_a \cup C$ is not a chain. Then the long top branch segment of $C$ is contained in the boundary of a single sector $\sigma_b$ of $D$.
\end{claim}

\begin{figure}[h]
\begin{center}
\includegraphics{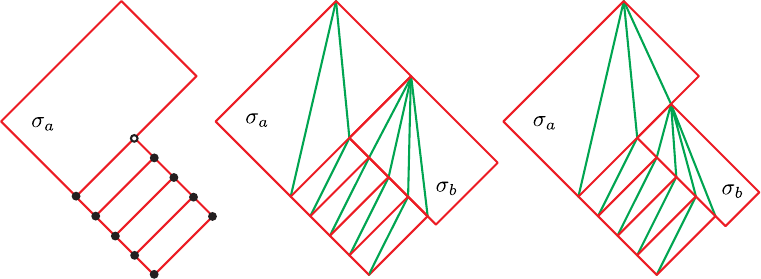}
\caption{Pictures from the proof of \Cref{claim:chains_and_sectors} and \Cref{lem:speed_converge}.
}
\label{fig:chainsectorclaim}
\end{center}
\end{figure}

We note that the use of `long' indicates that $C$ has length at least $2$, but the corresponding claim when $C$ has length $1$ is immediate since every edge in $D$ is in the boundary of exactly two sectors.

\begin{proof}[Proof of \Cref{claim:chains_and_sectors}]
This follows almost immediately from \Cref{lem:sectorveer} after labeling the veers of each vertex. 

In more detail, 
first note that the veer of the bottommost vertex of $C$ determines the veer of every other vertex of $\sigma_a \cup C$ except the top vertex of $\sigma_a$. If $\ell$ is the long top branch segment of $C$, then every vertex of $\ell$ except the final vertex has the same veer as that of the bottom-most vertex of $C$. The final vertex of $\ell$, which lies on the bottom branch segment of $\sigma_a$, must have the opposite veer (see the leftmost image in \Cref{fig:chainsectorclaim}, where the vertex colors indicate opposite veers).
Hence, if we let $\sigma_b$ be the sector of $D$ not in $C$ that contains the last edge in $\ell$, another application of \Cref{lem:sectorveer} implies that $\ell$ is completely contained in a side of $\sigma_b$ as in either the center or right image in \Cref{fig:chainsectorclaim}. 
\end{proof}

Now returning to the proof of the lemma, we observe that 
a flow ray $\rho$ in $D$ starting at a vertex along $b$ has as its next vertex the top vertex of a sector in $S$, and that this vertex lies in the top branch segment of some chain $C_i$ opposite $b$. We claim that $\rho$ (after one or two additional flow edges) meets 
the top branch segment of $C_{i-1}$ opposite $b$. Applying this claim inductively, we obtained that $\rho$ eventually meets the top branch segment of $C_0$ opposite $b$. Since $C_0$ is a chain of $\sigma$, this will complete the proof.

For this final claim, we use \Cref{claim:chains_and_sectors} to see that the top branch segment of $C_i$ opposite $b$ lies in the boundary of a single sector $\sigma_b$. Hence, the next vertex along $\rho$ is the top vertex of $\sigma_b$, 
which is either along the branch segment of $C_{i-1}$ opposite $b$, as in the center of \Cref{fig:chainsectorclaim}, or along the interior of a branch segment at the bottom of $C_{i-1}$, as in the right side of \Cref{fig:chainsectorclaim}. In the first case, we are immediately finished. In the second, the next flow edge from the top vertex of $\sigma_b$ is through the interior of the sector at the bottom of $C_{i-1}$. Hence, the next vertex along $\rho$ is in the branch segment of $C_{i-1}$ opposite $b$ as claimed.
\qedhere
\end{proof}

\begin{remark}\label{rmk:fastconverge}
The proof of the above lemma may be easily modified to show the following: 
for a sector $\sigma$ in a dynamic plane $D$,
any flow ray in $D$ starting in $\Delta(\sigma)$ passes through a vertex in a chain of one of the sectors immediately above $\sigma$ in $D$. In the case when the top vertex $v$ of $\sigma$ is the bottom vertex of another sector $\sigma'$ in $D$, the proof is exactly the same but with the roles of $\sigma$ and $\sigma'$ reversed. If $v$ is not the bottom vertex of any sector in $D$ then the flow ray will pass through a vertex in a chain of either of the two sectors immediately above $\sigma$. See \Cref{fig:fastconverge}. This fact will be used in \Cref{sec:apps}.
\end{remark}

\begin{figure}[h]
\centering
\includegraphics{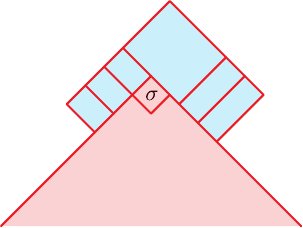}
\caption{Picture to accompany \Cref{rmk:fastconverge}. If the top vertex of $\sigma$ is not the bottom vertex of a sector in $D$, then any $\wt\Phi$-ray in $D$ starting in $\Delta(\sigma)$ must pass through a vertex in a chain of one of the two sectors immediately above $\sigma$. In the picture the chains of the two sectors immediately above $\sigma$ are colored blue.}
\label{fig:fastconverge}
\end{figure}

We now establish the following, which roughly states that $\Phi$ records orbits of the flow in a manner which is uniformly finite-to-one.

\begin{proposition} \label{prop:flowline_bound}
The map $\wt {\mf F}$ is surjective. If $p\in \PP$ does not lie on a singular stable leaf, then $|\wt {\mf F}^{-1}(p)| < 2\delta$, and if $p$ is nonsingular but lies on a singular stable leaf then $|\wt {\mf F}^{-1}(p)| < 4\delta$. 
\end{proposition}

\begin{proof}

For the moment, suppose that $p\in \mc P$ does not lie in a singular stable leaf. 

Let $(R_i)_{i\le 0}$ be a sequence of maximal $p$-rectangles that limit 
to a horizontal leaf through $p$. After refining the sequence, we may assume that $R_{i+1}$ lies above $R_{i}$ for each $i \le -1$.
For each $R_i$, let $\rho_i = \ray_p(R_i)$ be the $p$-ray starting at $R_i$. By 
\Cref{prop:planes_for_leaves}, each $\rho_i$ is contained in $D = D_p$, the dynamic plane associated to $p$. Hence, each maximal rectangle $R_i$ corresponds to a vertex $v_i$ in $D$, which is the initial vertex of $\rho_i$.

Next, let $Q$ be any edge rectangle in $\mc P$ that contains $p$. As before, let $R = R_Q$ be the maximal rectangle obtained by extending $Q$ vertically and also let $R'$ be the maximal rectangle obtained by extending $Q$ horizontally. 
These correspond to sectors $\sigma = \sigma(R)$ and $\sigma' = \sigma(R')$ in $D$, where $\sigma$ lies directly above $\sigma'$. Indeed, $\sigma$ is the sector of $D$ whose top vertex corresponds to $R$ and similarly for $\sigma'$.

By the choice of the sequence $(R_i)_{i\le 0}$, the rectangle $Q$ (and hence $R'$) lies above $R_j$ for $j \ll 0$. By \Cref{lem:relheights} there is a dual path from $R_j$ to $R'$, so by \Cref{lem:pushdown} we have $\Delta(\sigma(R_j))\subset \Delta(\sigma') \subset D$. Hence, we see that $\rho_j$ meets the descending set $\Delta(\sigma')$ for all $j$ sufficiently small. So by \Cref{lem:speed_converge}, $\rho_j$ must pass through a sector in one of the two chains for $\sigma$. Since the number of such sectors is uniformly bounded by \Cref{lem:chaintofan}, we can pass to a subsequence so that for all $j \le 0$, all $\rho_j$ pass through a fixed vertex $v_Q$ in the chain for $\sigma = \sigma(R_Q)$ and thereafter agree.
Let $\rho_Q$ be the $\wt \Phi$-ray starting at $v_Q$.

Iterating this construction for a sequence of edge rectangles $Q_{-1}, Q_{-2}, Q_{-3},\dots$ limiting to the horizontal leaf through $p$ yields a nested sequence of rays $\rho_{Q_{-1}}\subset \rho_{Q_{-2}}\subset \rho_{Q_{-3}}\subset \cdots$, the union of which is a $\wt \Phi$-line $\ell$ in $D$ such that each maximal rectangle along $\ell$ contains $p$. Hence, $\wt {\mf F}(l) = p$, as required.

For the bound on the preimage, note that any line in $\mf F^{-1}(p)$ is contained in the
dynamic plane $D_p$ for $p$ by \Cref{prop:planes_for_leaves}. Since the length of a chain
of sectors is less than $\delta$ by \Cref{lem:chaintofan},
the argument above implies that there are less than $2 \delta$ $p$-lines in $D_p$.

Next, suppose that $p\in \PP$ lies on a singular stable leaf $\ell$ but is nonsingular. Since any line in $\wt \Phi$ that determines a sequence of maximal rectangles containing $p$
must eventually have $p$ appearing in the vertical boundary of each of its rectangle, we observe that any line in the preimage $\mf F^{-1}(p)$ is contained in one of the two dynamic planes containing $D_\ell$ introduced in \Cref{rmk:singularplanes}.
Applying the same argument as in the previous case but in each of these planes containing $D_\ell$ produces at least two and less than $4\delta$ $\wt\Phi$-lines mapping to $p$ under $\wt{\mf F}$.

Finally, suppose that $p\in \mc P$ is a singular point. Pick a singular stable leaf $\ell$ through $p$ that is stabilized by some $g \ne 1$. Then $g$ stabilizes $D_\ell$ and we let $D$ be one of the two dynamic planes containing $D_\ell$ as in \Cref{rmk:singularplanes}. Since $D$ is also stabilized by $g$, $D$ contains a $g$-periodic line $l$ in $\wt\Phi$ by the proof of \Cref{prop:enough_flow}. Since the sequence of maximal rectangles associated to $l$ is $g$-invariant, the intersection of these rectangles $\wt{\mf F}(l)$ contains the fixed point of $g$, which is the singularity $p$. This completes the proof.
\end{proof}

\subsection{Cycles of $\Phi$ and closed orbits of $\varphi$}
\label{sec:cycles_orbits}
We next define the map ${\mf F}$ from \Cref{th:closed_orbits}.
Let $c$ be a directed cycle in $\Phi$, let $\wt c$ be a lift of $c$ to $\wt \Phi$, and let $g \in \pi_1 (M)$ be the deck transformation that generates the stabilizer of $\wt c$ and translates in the positive direction.  Then $ g \cdot \wt {\mf F}(\wt c)=\wt {\mf F}(g \cdot \wt c)=\wt{\mf F}(\wt c)$, so $p =\wt {\mf F}(\wt c)\in \PP$ is fixed by the action of $g$. Hence, $p$ corresponds to either a $g$-invariant flow line $\wt \gamma$ of $\uM$ or a singularity of $\mc P$ fixed by $g$. In the case where $p$ is nonsingular, the directed cycle $c$ is homotopic in $M$ to the closed orbit $\gamma = \wt \gamma / \langle g \rangle$ of $\phi$ and we set $\mf F(c) = \gamma$.

When $p = \wt{\mf F}(\wt c)$ is a singularity, each rectangle in the $g$-periodic sequence of maximal rectangles along $\wt c$ contains the singularity $p$ in its \emph{vertical} boundary. Indeed, $p$ must eventually be in a vertical side of the associated rectangles (by the description of $\wt \Phi$ in terms of maximal rectangles)
and so by $g$-periodicity, $p$ must be in a vertical side of every maximal rectangle associated to $\wt c$. Let $\ell^u$ be the unique leaf of the unstable foliation $\mc F^u$ containing $p$ that meets all the maximal rectangles along $\wt c$. Then $\ell^u$ is invariant under $g$. Let $\gamma$ be the unique multiple of the unstable prong curve in $M$ determined by $\ell^u$ to which $c$ is homotopic, and set $\mf F(c) = \gamma$.

Recall that $\mc O_\phi^+ \subset M$ is the union of closed orbits $\mc O_\phi$ of the flow along with all positive multiples of unstable prong curves in $M$. 
If we denote by $\mathcal{Z}_{\Phi}$ the set of directed cycles of $\Phi$, 
then the above discussion produces a map
\[
 {\mf F} \colon \mathcal{Z}_{\Phi} \to \mc O_\phi^+,
 \]
with the property that the directed cycle $c$ is homotopic to ${\mf F}(c)$ in $M$. We remark that when ${\mf F}(c)$ is nonsingular (i.e. a closed orbit in $M$; not a prong curve), it is the unique closed orbit of $\phi$ homotopic to the flow cycle $c$
since the flow does not have distinct homotopic orbits.

\smallskip

We need to further discuss the case when $\wt{\mf F}(\wt c)=p$ is a singularity, or equivalently when $\mf F(c)=\gamma$ is an unstable prong curve. 
Let $\wt \gamma$ be the corresponding lift of $\gamma$ to $\wt M$, and let $U$ be the component of $\wt M-\wt B^s$ containing $\wt\gamma$.
Let $T$ be the boundary of a small regular neighborhood of the singular orbit corresponding to $\gamma$, and let $\wt T$ be the lift to $\wt M$ of $T$ that is contained in $U$.
We describe some structure of $U$ and $\partial U$ that follows from the discussion in \cite[Section 5]{LMT20}.

The boundary tesselation by $\wt\tau$ of $\wt T$ can be naturally identified with the tesselation $\partial U\cap \wt\tau^{(2)}$, because $\wt\tau^{(2)}\cap U$ is homeomorphic to the product of the tesselation of the cusp with $(0,1)$ (\cite[Lemma 5.2]{LMT20}). 
Thus it makes sense to speak of upward/downward triangles and ladders on $\partial U$ (see \cite[Section 2.1.2]{LMT20} for terminology) and the following facts are contained in Lemmas 5.3--5.5 of  \cite{LMT20}.
Each upward ladder of $\partial U$ contains a unique branch line in its interior. The complementary components  of these branch lines are called \emph{bands}. Each band $B$ contains a unique downward ladder $L$ in its interior, and any $\wt \Phi$-line contained in a given band must lie in $L$. The structure of $\wt\Phi\cap B$ is such that there is at least one and no more than two asymptotic classes of $\wt \Phi$-lines contained in $L$ .
Further, if $\lambda^u$ is the unstable leaf of $\wt\phi$ corresponding to $\wt\gamma$, with projection $\ell^u$ in $\PP$, then the idea of the proof of \cite[Lemma 2.8]{landry2019stable} shows that $\lambda^u$ intersects $\wt T$ in the core of the ladder on $\wt T$ corresponding to $L$, and that $L$ can be characterized as the intersection with $\wt B^s$ of all tetrahedra $t$ such that the maximal rectangle corresponding to $t$ intersects $\ell^u$ in its interior and contains $p$ in its boundary. It follows that $\wt c$ lies in the core of $L$.

This discussion gives us the following lemma.

\begin{lemma}\label{lem:prongcurves}
Let $c$ be a $\Phi$-cycle and $\gamma$ a prong curve such that $\mf F(c)=\gamma$. Then $c$ lies in the unique band of $B^s$ corresponding to the downward ladder determined by $\gamma$. It follows that $c$ is transversely homotopic to $\gamma$ and that there are at most two $\Phi$-cycles mapping to $c$ under $\mf F$.
\end{lemma}

\smallskip
The next proposition establishes that all but finitely many primitive closed orbits of $\phi$ are homotopic to directed cycles of $\Phi$. Note that if the closed orbit $\gamma$ is homotopic to the flow cycle $c$, then we necessarily have $\mf F(c) =\gamma$ since no two closed orbits of the flow are homotopic.

\begin{proposition} \label{prop:flow_homotopic}
Let $\gamma$ be a nonsingular closed orbit of $\phi$. Then $\gamma$ is homotopic in $M$  to either a directed cycle of the flow graph $\Phi$ or an odd $AB$-cycle in $\Gamma$.
\end{proposition}

\begin{proof}
The closed orbit $\gamma$ is homotopic to a (nonunique) $\Gamma$-cycle $c$. 
This follows from the fact that all orbits of $\varphi$ are positive transverse to $\tau^{(2)}$ and perturbing $\gamma$ slightly, if necessary, we can assume that it misses $\tau^{(1)}$. The sequence of faces of $\tau$ intersected by $\gamma$ defines a dual cycle $c$ homotopic to $\gamma$. 
The proposition now follows from \Cref{prop:enough_flow}.
\end{proof}

We are now ready to prove the main theorem of this section.

\begin{proof}[Proof of \Cref{th:closed_orbits}]
For item $(1)$, the definition of $\mf F$ ensures that $\gamma = \mf F(c)$ is homotopic to $c$ in $M$. It remains to prove that $c$ and $\gamma$ are transversely homotopic. When $\gamma$ is a regular orbit, this follows immediately from \Cref{lem:transversely_homotopic}. For this, we use that the regular orbit $\gamma$ is not homotopic to a branch curve because branch curves in $\ol M$ are homotopic to singular orbits and no distinct closed orbits in $\ol M$ are homotopic (\Cref{lem:flow_space_prop}).

If $\gamma$ is instead an unstable prong curve, $\gamma$ is homotopic to branch curves corresponding to the same singular orbit and so \Cref{lem:transversely_homotopic} does not apply. Instead we simply apply \Cref{lem:prongcurves}. 

Item (2) is also a direct application of \Cref{lem:prongcurves}.

Item $(3)$ essentially follows from \Cref{lem:numb_classes}. For this, let $c_1, \ldots, c_n \in \mc Z_\phi$ with $\mf F(c_i) =\gamma$. Let $\wt \gamma$ be a fixed lift of $\gamma$, let $g \in \pi_1(M)$ generate the stabilizer of $\wt \gamma$, and let $p$ be the image of $\wt \gamma$ in $\mc P$. Since each $c_i$ is homotopic to $\gamma$ we can choose lifts $\wt c_i$ that are also invariant under $g$. Hence, $p = \wt{\mf F}(\wt c_i)$.
That is, each $\wt c_i$ is a $p$-line in $\wt \Phi$.
By \Cref{prop:planes_for_leaves}, each $\wt c_i$ is contained in the dynamic plane $D_p$, and by \Cref{lem:numb_classes} the number of asymptotic classes of the $\wt c_i$ is equal to the width of $D$. However, if two $g$-invariant $p$-lines are asymptotic, then they are equal. We conclude that $n$ is equal to the width of $D_p$. 
By \Cref{lem:ABregion}, the width of $D_p$ is equal to one, unless $\gamma$ is homotopic to an $AB$-cycle. Regardless, the width is no more than $\delta$ by \Cref{cor:width}.

Finally, item $(4)$ follows from \Cref{prop:flow_homotopic} and the fact that no distinct closed orbits of $\phi$ are homotopic.
\end{proof}

Next we mention a corollary that further connects the flow and triangulation. 
Recall from \Cref{sec:cones} that $\cone_1(\Gamma) \subset H_1(M; \RR)$ denotes the cone of homology directions of $\tau$, which is the cone positively spanned by the classes of closed curves positively transverse to $\hbs$. We proved in \cite[Theorem 5.1]{LMT20} that this agrees with the cone positively spanned by $\Phi$-cycles (c.f. \Cref{prop:enough_flow}).
In \cite{fried1982geometry}, Fried associates to any flow a \define{cone of homology directions} in first homology which can be thought of as the positive span of classes of nearly closed orbits. In the current context, the cone of homology directions of our pseudo-Anosov flow $\phi$ is polyhedral and positively spanned by closed orbits of $\phi$. Since the flow is positively transverse to $\tau^{(2)}$ away from the singular orbits, and each singular orbit has a multiple which is homotopic to a transversal by \Cref{th:closed_orbits}, it is clear that $\cone_1(\Gamma)$ contains Fried's cone. 
\Cref{th:closed_orbits} also easily implies the reverse containment, giving us the following.

\begin{corollary}[Homology directions]
Suppose that the veering triangulation $\tau$ is associated to the flow $\phi$. Then the image of $\cone_1(\Gamma)$ in $H_1(\overline M;\RR)$ is equal to Fried's cone of homology directions for $\phi$.
\end{corollary}

\smallskip

We conclude this subsection by showing that the veering triangulation also detects which orbits of $\phi$ are twisted.

\begin{lemma} \label{lem:double_cover}
Let $\gamma$ be a nonsingular closed orbit of $\phi$ and let $c$ be any
directed cycle of $\Gamma$ homotopic to $\gamma$. Then $\gamma$ is untwisted if and only if $c$ has an even number of $AB$--turns.
\end{lemma}

\begin{proof}
As in the proof of \Cref{lem:dynamic_plane_orient}, $c$ has an even number of AB turns if and only if the pullback of the tangent bundle over $B^s$ is orientable \cite[Lemma 5.6]{LMT20}. Lifting to the universal cover $\uM$, this is equivalent to a fixed coorientation on $\wt B^s$ being preserved by the deck transformation $g \in \pi_1(M)$ with $\langle g \rangle = \mathrm{stab}(\wt c)$. (We recall that since $M$ deformation retracts to $B^s$, the branched surface $\wt B^s$ is contractable. Hence, its tangent plane bundle is trivial.) Such an coorientation on $\wt B^s$ orients all edges of the lifted triangulation $\wt \tau$ and these orientations are preserved by $g$. Note that by looking at the intersection of $\wt B^s$ with any face of $\wt \tau$, we see that the widest edge of the face is oriented consistently with respect to the other edges, i.e. the widest edge is the homological sum of the other two.

Now each vertex crossed by $\wt c$ corresponds to a tetrahedron of $\wt \tau$ and hence to a maximal rectangle in $\mc P$. As in the construction of the map $\wt {\mf F}$, the intersection of all these maximal rectangles is the fixed point $p$ of $g$, which by construction is the projected image of the $g$-periodic flow line $\wt \gamma$. 
Moreover, for any positive ray $\wt c^+$ of $\wt c$ the intersection of the associated maximal rectangles is a segment of the stable leaf $\ell$ through $p$ in $\mc P$ (see \Cref{fact:non-accumulation}). The fact that the edges of $\tau$ are coherently oriented along the faces crossed by $\wt c$ translates to the fact that the $\tau$-edges
of the maximal rectangles in our collection coherently cross $\ell$ from one (say the left) side to the other. Since this ordering is preserved by $g$, the stable leaf $\ell$ has a $g$-invariant coorientation. Hence, the stable leaf through $\wt{\mf F}(\wt \gamma)$ also has a coorientation preserved by $g$ and so the orbit $\gamma$ is untwisted.

Reversing the logic, if $\ell$ has a $g$-invariant coorientation, then we can used this to coherently orient the $\tau$-edges crossing $\ell$ and this translates to a coherent orientation on the edges of $\wt \tau$ that is compatible on faces in the above sense and which is $g$-invariant. Hence, the orientation on any one of these edges coorients $\wt B^s$ in a $g$-invariant fashion. This implies, again as in \Cref{lem:dynamic_plane_orient}, that $c$ has an even number of $AB$--turns. The proof is complete.
\end{proof}

%% !TEX root =veering_poly2.tex
\newcommand\ZF{{\mc Z}}

\section{Growth rates of orbits and the veering polynomial}
\label{sec:growth}
In this section, we show how a modified version of the veering polynomial can detect
growth rates of closed orbits of subsets of the flow, even in the nonlayered setting. Our
main theorems are \Cref{th:entropy_nonlayered}, which relates growth rates of the flow to
those of the flow graph, and \Cref{th:entropy-veering-nonlayered}, which relates the growth
rates to the veering polynomial. These are new even in the case of surfaces
contained in the boundary of a fibered face and more on this special case is discussed in
\Cref{sec:boundary_cone}. 
In \Cref{prop:exp_growth} we will use these to give a topological criterion
for these  growth rates to be strictly greater than 1. 

In \Cref{sec:closed_case} we extend these results to study growth rates for the closed manifold $\ol M$ after cutting along a transverse surface.

\medskip

\noindent {\em Cutting along a surface:} Let $S\subset M$ be a properly embedded surface positively transverse to the
flow $\varphi$, and let $M|S$ denote $M$ cut along $S$, with its components indicated as 
$M|S = \bigcup_i M|_iS$. 
We let $\varphi|S$ denote the restricted semiflow on $M|S$
and let $\varphi|_iS$ denote the further restriction to the 
component $M_i|S$.
Let $\mc O|S$ and $\mc O|_i S$ denote the directed closed orbits of
$\varphi|S$ and $\varphi|_iS$, respectively. In particular, $\mc O|_i S$ are the closed orbits of $\phi$ that are contained in $M|_iS$.

Let $\Phi$ be the flow graph of the veering triangulation $\tau$ and $\iota \colon \Phi\to M$ be
its embedding in dual position. If $S$ is carried by the veering triangulation $\tau$ then it is positively transverse
to $\iota(\Phi)$ as well as the flow (\Cref{thm:flow transverse}), and we denote by $\Phi\ssm S$ the flow
graph cut along $\iota^{-1}(S)$.
Then let $\Phi|S$ denote the recurrent subgraph of $\Phi\ssm S$, i.e. the union of edges traversed by directed cycles of $\Phi\ssm S$. As for
$\varphi$, let $\Phi|_i S$ denote the subgraph of $\Phi|S$ contained in $M|_i S$, and let
$\ZF_{\Phi|S}$ and $\ZF_{\Phi|_iS}$ denote the directed cycles of $\Phi|S$ and $\Phi|_i S$,
respectively.

Now let $\xi\in H^1(M|_iS)$ be a cohomology class which is positive 
on the closed orbits $\mc O|S_i \subset M|_iS$ as well as on unstable prong curves that are contained in $M|S_i$. We call any such class \define{positive} with respect to $\phi|_iS$ and note that such positive classes determine a (possibly empty) open cone in  $H^1(M|_iS)$.

We then consider for a positive class $\xi$ the exponential growth rates
\begin{align} \label{eq:entropy_semiflow_non}
\mathrm{gr}_{\varphi|_iS}(\xi) = \lim_{L\to \infty}  \# \{ \gamma \in \mc O|_iS : \xi(\gamma) \le L   \} ^{\frac{1}{L}},
\end{align}
and
\begin{align} \label{eq:entropy_flowgraph_non}
\mathrm{gr}_{\Phi|_iS}(\xi) = \lim_{L\to \infty}  \# \{ c \in \ZF_{\Phi|_iS} : \xi(\iota(c)) \le L   \} ^{\frac{1}{L}}.
\end{align}

The first main theorem of this section will be: 

\begin{theorem}[Growth rates in $M|_iS$] \label{th:entropy_nonlayered}
Let $\tau$ be a veering triangulation of $M$ with dual flow $\varphi$. Consider a surface $S$ carried by $\tau^{(2)}$ and fix a component $M|_iS$ of $M|S$. 

For any positive class $\xi \in H^1(M|_iS)$ the growth rates of $\varphi|_iS$ and $\Phi|_iS$ exist and
\[
\mathrm{gr}_{\varphi|_iS}(\xi) = \mathrm{gr}_{\Phi|_iS}(\xi).
\]
\end{theorem}

In fact, $\mathrm{gr}_{\varphi|_iS}(\xi) >1$ so long as $\mc O|_iS$ contains infinitely many primitive orbits. See \Cref{prop:exp_growth}.

To compute these growth rates, we will define 
a veering polynomial $V_{\phi|_iS} \in \mathbb{Z}[H_1(M|_iS) / \text{torsion}]$ (see \Cref{subsec:polynomial counting}) directly from the Perron polynomial 
$P_{\Phi}$
of the flow graph $\Phi$
and obtain this corollary:

\begin{theorem}[Growth rates and the polynomial] \label{th:entropy-veering-nonlayered}
Let $\tau$ be a veering triangulation of $M$ with dual flow $\varphi$. Consider a surface $S$ carried by $\tau^{(2)}$ and fix a component $M|_iS$ of $M|S$. 

For any positive $\xi \in H^1(M|_iS)$,
the growth rate $\mathrm{gr}_{\phi|_i S}(\xi)$
is equal to the reciprocal of the smallest positive root of $V_{\phi|_iS} ^{\xi}$, the veering polynomial of $M|_iS$ specialized at $\xi$. 
\end{theorem}

\subsection{Cutting with cohomology}
\label{sec:cutting_co}
We first observe that $\Phi|S$ 
depends only on the Poincar\'e dual of $[S]$ in $H^1(M)$: 

\begin{lemma} \label{lem:zero_cycles}
The directed cycles of $\Phi|S$ are exactly the directed cycles of $\Phi$ that are zero under $\iota^*\eta$, where $\eta \in H^1(M)$ is the Poincar\'e dual of $S$.
\end{lemma}

\begin{proof}
Let $c$ be a directed cycle of $\Phi$. If $c$ is in $\Phi|S$ it
 misses $\iota^{-1}(S)$, so $\eta(\iota(c)) = 0$. 
Conversely, if $\eta(\iota(c)) = 0$, then $\iota(c)$ must miss $S$
since all intersection of $\iota(\Phi)$ with $\tau^{(2)}$ are transverse and positive. 
\end{proof}

Motivated by this, for $\eta \in \cone_2(\tau)$, define $\Phi|\eta$ to be the subgraph of
$\Phi$ whose edges are traversed by directed cycles that are $\iota^*\eta$--null. 
Alternatively, $\Phi|\eta$ is the largest recurrent subgraph of
$\Phi$ on which the pullback of $\eta$ is $0$ (see e.g. \cite[Lemma 5.10]{LMT20}).
We call $\Phi|\eta$ the  
\define{restricted flow graph} for $\eta$.
When $\eta$ is dual to a carried surface $S$, \Cref{lem:zero_cycles} implies that
$\Phi|\eta = \Phi|S$.
Although this will not play a direct role here, we reconsider this perspective in \Cref{sec:boundary_cone}.

\subsection{Parameterizing orbits of $\phi|_iS$}
Recall that $\mc O^+$ denotes the union of $\phi$'s closed orbits $\mc O = \mc O_\phi$ together with all positive multiples of the finitely many unstable prong curves of $M$, and we define $\mc O^+|S$ accordingly.  We have $\mc O^+|S = \bigcup_i \mc O^+|_iS$, where $\mc O^+|_iS$ are the closed orbits and unstable prong curves that are contained in $M|_iS$. 

\begin{lemma}[Decomposing orbits]\label{lem:decomposing_flow}
The map $\mf F \colon \ZF_\Phi \to \mc O^+$ from \Cref{th:closed_orbits} restricts to a map
\[
\mf F|_iS \colon {\ZF}_{\Phi |_iS} \to \mc O^+|_iS,
\]
whose image is $\mathrm{im}({\mf F}) \bigcap \mc O^+|_iS$ for each component $M|_iS$ of $M|S$. 

Moreover, for each directed cycle $c$ of $\Phi|_iS$, $\iota(c)$ is homotopic to $\mf F(c)$ \emph{within} $M|_iS$.
\end{lemma}positively transverse

\begin{proof}
Fix a component $M|_iS$ and let $c \in {\ZF}_{\Phi |_iS}$ be a directed 
cycle of ${\Phi |_iS}$.
Recall from \Cref{th:closed_orbits} that $\gamma = {\mf F}(c)$ is the closed orbit or unstable prong curve of $\phi$ that is
transversely homotopic to $\iota(c)$. That is, there is a homotopy from $\iota(c)$ to $\gamma $ through curves that are positively transverse to $\tau^{(2)}$. Since $S$ is carried by $\tau$, the curves in this homotopy are also positively transverse to, and hence disjoint from, $S$. Since $\iota(c) \subset M|_iS$ by definition of ${\Phi |_iS}$, we conclude that $\iota(c)$ is homotopic to $\gamma$ within $M|_iS$ and so in particular $\gamma = {\mf F(c)} \in \mc O^+|_iS$. 

Similarly, if $\gamma \in \mc O^+|_i S$ is in the image of ${\mf F}$, then any preimage
$c$ must be $0$ under $\iota^*\eta$ (where $\eta$ is the Poincare\'e dual of $S$, as in
\Cref{lem:zero_cycles}), hence $c$ is in $\Phi|S$. Just as above, we may additionally
conclude that $c \in {\ZF}_{\Phi |_iS}$. 
\end{proof}

\subsection{Comparing growth rates}

We are now ready to prove \Cref{th:entropy_nonlayered}. 
Let $\xi \in H^1(M|_i S)$ be positive with respect to $\phi|_iS$. By 
definition, $\xi$ is positive on $\mc O^+|_iS$, the set of
closed orbits 
and unstable prong curves that are contained in $M|_iS$.

\begin{lemma}\label{lem:pos_Phi}
If $\xi \in H^1(M|_i S)$ is positive, then its pullback $\iota^*\xi$
to $H^1(\Phi|_iS)$ is positive on directed cycles.
\end{lemma}

\begin{proof}
By \Cref{lem:zero_cycles}, for each directed cycle $c$ of $\Phi|_iS$ the 
image $\iota(c)$ is homotopic in $M|S_i$ to a closed orbit or unstable prong
curve in $M|S_i$. The lemma follows. 
\end{proof}

We shall now prove that the growth rates, counting with respect to $\xi$, of
closed orbits of $\varphi|_i S$ and directed cycles of $\Phi|_i S$ exist and are equal: 
\[
\mathrm{gr}_{\varphi|_iS}(\xi) = \mathrm{gr}_{\Phi|_iS}(\xi). 
\]
We will use results from the theory of growth rates of cycles in directed graphs and refer to McMullen's paper \cite{mcmullen2015entropy}.

\begin{proof}[Proof of \Cref{th:entropy_nonlayered}]
Since $\iota^*\xi$ is positive on directed cycles of $\Phi|_iS$ (\Cref{lem:pos_Phi}), it follows that $\mathrm{gr}_{\Phi|_iS}(\xi)$ exists (see e.g. \cite[Lemma 3.1]{mcmullen2015entropy}). 

We first show that
$$
\mathrm{gr}_{\Phi|_iS}(\xi) \le \liminf_{L\to \infty}  \# \{ \gamma \in \mc O|_iS : \xi(\gamma) \le L   \} ^{\frac{1}{L}}.
$$
For this, it suffices to assume that $\mathrm{gr}_{\Phi|_iS}(\xi) > 1$,
otherwise there is nothing to show. 
By \Cref{th:closed_orbits}, there is a constant $m$ such that for any $\gamma \in \mc O^+$, $\# \mf F^{-1}(\gamma) \le m$. By \Cref{lem:decomposing_flow}, $\mf F$ maps $\ZF_{\Phi |_iS}$ into
$\mc O^+|_iS$ and for each directed cycle $c$ of $\Phi|_iS$, $\mf F(c)$ is homotopic to $\iota(c)$ within $M|_iS$.
From these facts, we have
\begin{align*}
\#\{ \mf F(c) \in \mc O|_i S: \xi(\mf F(c))\le L\} &\le 
\#\{ c \in \ZF_{\Phi|_i S}: \iota^*\xi(c)\le L\}\\
 &\le m \cdot \#\{ \mf F(c) \in \mc O|_i S: \xi(\mf F(c))\le L\}.
\end{align*}
Thus we have equality of growth rates: 
\begin{align*}
\mathrm{gr}_{\Phi|_iS}(\xi) & 
 = \lim_{L\to \infty}  \# \{ {\mf F} (c) : \xi({\mf F}(c)) \le L   \} ^{\frac{1}{L}}
\end{align*}
which shows, in particular, that $\# \{ {\mf F} (c) : \xi({\mf F}(c)) \le L \}$ has exponential growth. On the other hand, the multiples of unstable prong curves in $\mc O^+|_iS$
have
at most linear growth so removing them
from our count does not affect the growth rate. Hence,
\begin{align}
\label{eq:lower_b}
\mathrm{gr}_{\Phi|_iS}(\xi)  
&= \lim_{L\to \infty}  \# \{ \gamma \in \mathrm{Im}(\mf F) \cap \mc O|_iS: \xi(\gamma) \le L   \} ^{\frac{1}{L}}\\
&\le  \liminf_{L\to \infty}  \# \{ \gamma \in  \mc O|_iS: \xi(\gamma) \le L   \} ^{\frac{1}{L}} \nonumber. 
\end{align}

For the other direction, again note that we can assume that 
\[
1< \limsup_{L\to \infty}  \# \{ \gamma \in \mc O|_iS : \xi(\gamma) \le L   \} ^{\frac{1}{L}} 
\]
otherwise we are done. Hence, 
$ \# \{ \gamma \in  \mc O|_iS: \xi(\gamma) \le L   \}$ grows exponentially.
By \Cref{th:closed_orbits}
every primitive $\gamma \in \mc O|_iS$ is in the image of $\mf F$ with at most finitely many exceptions corresponding to closed orbits homotopic to odd $AB$-cycles.
Hence, the image of $\mf{F}|_iS$ misses at most finitely many primitive orbits in $\mc O_i|S$ and their multiples.
It then follows easily that 
\[
 \limsup_{L\to \infty}  \# \{ \gamma \in  \mc O|_iS: \xi(\gamma) \le L   \} ^{\frac{1}{L}} \le
 \mathrm{gr}_{\Phi|_iS}(\xi), 
\]
and the proof is complete.
\end{proof}

\subsection{Adapting the veering polynomial and counting orbits}
\label{subsec:polynomial counting}
The last object needed for our discussion is an adapted version of the veering polynomial.
For the directed graph $\Phi|S = \bigcup_i \Phi|_iS$, let $P_{\Phi|S}$ and $P_{\Phi|_iS}$ denote the respective Perron polynomial. For each component $M|_iS$ of $M|S$, define its \define{veering polynomial} to be 
\[
V_{\phi|_iS}  = \iota_*(P_{\Phi|_iS}) \in \Z[H_1(M|_iS)/\text{torsion}],
\]
where $\iota_* \colon \Z[H_1(\Phi|_iS)] \to \Z[H_1(M|_iS)/\text{torsion}]$ is the ring homomorphism induced by inclusion.

It not hard to see that 
\[
P_{\Phi|S} = \prod_i P_{\Phi|_iS}
\]
in $ \Z[H_1(\Phi|S)] =  \bigotimes_i \Z[H_1(\Phi|_iS)]$
since $\Phi|S$ is the disjoint union of the $\Phi|_iS$. 
Indeed, in this case, the adjacency matrix for $\Phi|S$ is 
a block diagonal matrix whose blocks are the adjacency 
matrices for the $\Phi|_iS$.

Recall from \Cref{sec:poly} that any directed graph $D$ has a cycle complex $\C(D)$ whose cliques are the disjoint simple directed cycles of $D$.
Moreover, the Perron polynomial $P_D$ of $D$ is equal to the clique polynomial of $\C(D)$.

\begin{proposition} \label{prop:polys}
Let $\eta \in H^1(M)$ be the Poincar\'e dual to $S$.
The inclusion $\Phi|S\to \Phi$ induces an inclusion $\C(\Phi|S) \to \C(\Phi)$ whose image is the full subcomplex spanned by simple cycles that are zero under $\iota^*\eta$. 

Hence, $P_{\Phi|S}$ can be obtained from $P_\Phi$ by removing terms which evaluate nontrivially under $\eta$. 
\end{proposition}

\begin{proof}
Since $\Phi|S \to \Phi$ is inclusion, we have the inclusion of vertices $\C^0(\Phi|S) \to \C^0(\Phi)$. This amounts to saying that simple cycles of $\Phi|S$ map to simple cycles of $\Phi$. The full inclusion statement is then equivalent to saying that cycles $c_1$ and $c_2$ of $\Phi|S$ are disjoint if and only if they are disjoint as cycles in $\Phi$. This is equally clear.

Finally, as in \Cref{sec:poly}, we know that the Perron polynomial $P_\Phi$
is equal to
\[
1 + \sum_\sigma -1^{|\sigma|} \sigma,
\]
where the sum is over cliques of $\C(\Phi)$. Hence the only terms of $P_\Phi$ that do not appear in $P_{\Phi|S}$ are those composed of multicurves that have positive evaluation under $\iota^*\eta$. This completes the proof. 
\end{proof}

We henceforth consider $P_{\Phi|S}$ as being obtained from $P_\Phi$ by removing the terms that correspond to cycles which are nontrivial under $\iota^*\eta$. 

\medskip

We can now prove \Cref{th:entropy-veering-nonlayered}
which relates growth rates of $\phi$ in $M|_iS$ to the veering polynomial:  

\begin{proof} [Proof of \Cref{th:entropy-veering-nonlayered}]
Let $\iota \colon \Phi|_iS \to M|_iS$ be as above. 
Since $\xi$ is positive, $\iota^*\xi$ is positive on all directed cycles of 
$\Phi|_iS$ (\Cref{lem:pos_Phi}).

By \cite[Theorem 3.2]{mcmullen2015entropy}, $\mathrm{gr}_{\Phi|_iS}(\xi) $ is equal to
the reciprocal of the smallest root of the Perron polynomial of $P_{\Phi|_iS}$ specialized at $\iota^*\xi$.
(Technically, this is applied to a metric on $\Phi|_i S$ representing $\iota^*\xi$; see
\cite[Lemma 5.1]{mcmullen2015entropy} or \cite[Lemma 5.10]{LMT20}.)
Since
\[
P_{\Phi|_iS}^{\iota^*\xi} = V_{\phi|_iS}^\xi,
\]
the result follows from \Cref{th:entropy_nonlayered}.
\end{proof}

We conclude this section with a characterization of when the entropy is positive. 

\begin{proposition} \label{prop:exp_growth}
With notation as in \Cref{th:entropy_nonlayered}, the growth rate $\mathrm{gr}_{\phi|_iS}(\xi)$ is strictly greater than $1$ for every positive $\xi \in H^1(M|_iS)$ if and only if there are infinitely many primitive closed orbits of $\phi$ contained in $M|_iS$.
\end{proposition}

\begin{proof}
If $\mathrm{gr}_{\phi|_iS}(\xi)>1$, then the claim that there are infinitely many primitive closed orbits in $M|_iS$ is clear, since otherwise the growth of all orbits is linear.

Now suppose that there are infinitely many primitive closed orbits in $M|_iS$. Then, as in the proof of \Cref{th:entropy_nonlayered}, there are infinitely many closed primitive cycles in $\Phi|_iS$. Since the directed graph $\Phi|_iS$ is finite, this mean that it has recurrent components that are neither trivial nor cyclic. Hence, the growth rate of directed cycles with respect to any positive cocycle is strictly greater than $1$. As this quantity is the same as $\mathrm{gr}_\phi(\xi)$, the proof is complete. 
\end{proof}

\section{Transverse surfaces and growth rates for closed manifolds}
\label{sec:closed_case}
In this section, we outline a way in which the results of the previous section extend to closed $3$-manifolds. Here the veering triangulation is still the central tool but does not appear in theorem statements.

Let $\ol M$ be a closed $3$-manifold and let $\phi$ be a pseudo-Anosov flow on $\ol M$ without perfect fits. Let $S$ be a closed surface in $\ol M$ that is transverse to $\phi$. For notational simplicity, we will assume that $S$ is connected. We orient $S$ so that each intersection with an orbit of $\phi$ is positive and note 
that $M|S$ is connected. Let $\mc O|S$ be the set of closed orbits of $\phi$ that miss $S$ and hence are contained in $M|S$. Below, we will define an invariant $V_{\phi|S} \in \mathbb{Z}[H_1(\ol M|S)]/\mathrm{torsion}$ 
which we call the \define{veering polynomial} of $\ol M|S$. We will call a class $\xi \in H^1(\ol M|S)$ \define{strongly positive} if it is positive on $\mc O|S$ as well as a certain finite collection of curves in $\partial M|S$ that we define below (\Cref{subsec:strongly positive}).

We will prove:

\begin{theorem}\label{th:growth_rates_closed}
Let $\phi$ be a pseudo-Anosov flow on $\ol M$ without perfect fits. Let $S$ be a closed connected surface in $\ol M$ that is transverse to $\phi$. 

For any strongly positive class $\xi \in H^1(\ol M|S)$, the growth rate
\[
\mathrm{gr}_{\phi|S}(\xi) = \lim_{L\to \infty}  \# \{ \gamma \in \mc O|S : \xi(\gamma) \le L   \} ^{\frac{1}{L}}
\] 
 of closed orbits in $M|S$ exists and equals
the reciprocal of the smallest root of the specialization $V_{\phi|S}^\xi$ of the veering polynomial.
\end{theorem}

Recall that $M = \ol M \ssm \{\text{singular orbits}\}$ admits a veering triangulation $\tau$. 
Let $\iota \colon \Phi \to M \subset \ol M$ be the embedding of the flow graph in dual position so that its edges are positively transverse to $\tau^{(2)}$. Fix $S$ as in the statement of \Cref{th:growth_rates_closed} and let $\eta \in H^1(\ol M)$ be its Poincar\'e dual.

We begin by noting that if we also puncture $S$ along the singular orbits 
of $\varphi$, we obtain a surface $\mr S$ in $M$ that is positively transverse to the remaining orbits. 
However, it is not clear whether $\mr S$ is necessarily carried by the branched surface $\tau^{(2)}$ and so the results of the previous section do not automatically apply. 
Instead we use the following claim, which is all we will need.

\begin{claim}[Homotoping the flow graph] 
\label{cl:hom_flow}
Let $S$ be a closed surface positively transverse to $\phi$. The flow graph $\iota \colon \Phi \to \ol M$ can be isotoped to a map $\iota_0 \colon \Phi \to \ol M$ so that its edges are positively transverse to $S$.
\end{claim}

\begin{proof}
Since the surface $S$ is positively transverse to $\varphi$, results in
\cite{mosher1992dynamical} imply that $S$ is taut 
and so its Thurston norm equals $\vert \chi(S) \vert$. Applying the Poincar\'e--Hopf index formula to the singular foliation $\mc F^s \cap S$ of $S$, we see that $\chi(S) = e_\tau(S)$, where $e_\tau$ is the combinatorial Euler class of \cite{Landry_norm}. 
Then the main theorem of \cite{Landry_norm} states that there exists an isotopy that pushes a certain family of annuli of $S$ into a neighborhood of the singular orbits so that outside this neighborhood $S$ is carried by $\tau^{(2)}$.
This implies, in particular, that we may isotope the flow graph $\Phi$ in $\ol M$ to be positively transverse to $S$, as required.
\end{proof}

For the proof of \Cref{th:growth_rates_closed}, we wish to follow along the lines of the proofs for \Cref{th:entropy_nonlayered} and \Cref{th:entropy-veering-nonlayered}, except that we no longer have the full strength of the veering triangulation available (see \Cref{subsec:strongly positive}). 
In what follows, we adapt the argument to only use the fact that the flow graph $\Phi$ is positively transverse to the surface $S$.

As before, we define $\Phi \ssm S$ by cutting $\Phi$ along $\iota_0^{-1}(S)$ and we take its
recurrent subgraph $\Phi|S$. By construction, the restriction $\iota_0 \colon \Phi|S \to \ol
M|S$ is defined and $\Phi|S$ is exactly the subgraph of $\Phi$ consisting of edges that are traversed by cycles which are $0$ under $\iota_0^*\eta \in H^1(\Phi)$ (c.f. \Cref{lem:zero_cycles}).  

\subsection{Stable and unstable curves}
\label{sec:stabcurve}
The main complication in studying flows in the cut manifold $\ol M|S$ is that orbits of the restricted flow may be homotopic into $S$ itself. We begin by analyzing this possibility.

For any embedded surface $S$ in $\ol M$ that is positively transverse to $\varphi$,
we define the singular foliations $\mc F_S^{s/u} = \mc F^{s/u} \cap S$ on $S$.
The following is an observation that follows easily from work of Cooper--Long--Reid \cite{cooper1994bundles} in the case of a circular flow and more generally from Fenley \cite{fenley1999surfaces}. 

\begin{lemma} \label{lem:fol}
Suppose that $\gamma$ is a closed orbit of $\varphi$ that is homotopic to a closed curve $c$ in $S$. Then $c$ is homotopic in $S$ to a closed leaf of either $\mc F_S^s$ or $\mc F_S^u$. 

Moreover, every closed leaf of $\mc F_S^s$ or $\mc F_S^u$ can be oriented so that it is homotopic to a closed orbit of $\varphi$.
\end{lemma}

Note that the conclusion of the lemma places $c$ into one of at most finitely many
homotopy classes of curves in $S$ and implies that there are at most finitely many closed
orbits of $\varphi$ that are homotopic into $S$. Here we are using the fact that no distinct closed orbits 
of $\varphi$ are homotopic (see \Cref{lem:flow_space_prop}(3)).

We call the closed leaves of $\mc F_S^{s/u}$, with their orientation determined by \Cref{lem:fol}, the \define{stable/unstable curves} of $S$. 

\begin{proof}
Consider lifts $\wt \gamma, \wt c\subset \wt S$ to the universal cover $\wt {\ol M}$
chosen so that there is a deck transformation $g \in \pi_1(S)$ preserving $\wt \gamma$,
$\wt c$ and $\wt S$. Further assume that $g$ translates $\wt \gamma$ in its positive
direction.
We note that $\wt S$ is a properly embedded plane in $\wt {\ol M}$ that is positively transverse to the lifted flow. Since $\wt S$ separates $\wt {\ol M}$, this implies that $\wt S$ intersects each flow lines at most once. Let $\wt{\mc F}^{s/u}(\wt \gamma)$ be the stable/unstable leaves through $\wt \gamma$.

Now consider the projections to the flow space $\mc Q$ of $\ol M$. To keep notation as
simple as possible, the projection of $\wt x$ in $\wt {\ol M}$ to $\mc Q$ will be denoted
by $\wh x$. Since $\gamma$ is homotopic into $S$ it has intersection pairing 0 with it,
which means by positive transversality of $S$ to the flow that $\gamma$ misses $S$ and hence
$\wh\gamma$ is not contained in $\wh S$. 
According to \cite[Proposition 4.3]{fenley1999surfaces}, the boundary of $\wh S$ in $\mc
Q$ is a disjoint union of \emph{leaf lines}, which are lines of the foliations $\wh
\FF^{s/u}$ that are regular on their $\wh S$--side, meaning that each compact subsegment of the line is contained in the boundary of a maximal rectangle whose interior is contained in $\wh S$. (This is discussed in more detail in \Cref{sec:entropy} where a generalization is also proven.)

Let $\ell$ be the unique leaf of either the stable or unstable foliation in the boundary of $\wh S$ that separates $\wh \gamma$ from $\wh S$. 
Since $g$ stabilizes $\wh \gamma$ and $\wh S$, it also stabilizes $\ell$. Hence, $g$ fixes
a point in $\ell$ and, because fixed points are unique, we conclude that $\wh\gamma \in
\ell$ (\Cref{lem:flow_space_prop}). 
 If $\ell$ is a leaf of the stable foliation,
then the unstable leaf through $\wh \gamma$ meets $\wh S$. Otherwise, $\ell$ is a leaf of
the unstable foliation and the stable leaf through $\wh \gamma$ meets $\wh S$.  
This means that
one of the stable or unstable leaves of $\wt{\mc F}^{s/u}$ through $\wt \gamma$ intersects $\wt S$ in a
$g$-invariant line. This line descends to a closed curve of $\mc F_S^{s/u}$ 
homotopic to $c$ in $S$, and
this finishes the proof in this direction.

Conversely, any closed leaf of $\FF_S^{s/u}$ is contained in a leaf of $\FF^{s/u}$ that is either an annulus, a Mobius band, or singular. In either case, the `core' of this leaf is a closed orbit of $\varphi$ and the proof is complete.
\end{proof}

\subsection{Strongly positive classes in $H^1(\ol M|S)$}
\label{subsec:strongly positive}
In our current setting, we would like to have an analogue of \Cref{lem:decomposing_flow}
stating that if $c$ is a directed cycle in $\Phi|S$ and $\gamma \in \mc O|S$ is the unique
orbit of $\phi$ homotopic to $\iota_0(c)$, then $\gamma$ and $\iota_0(c)$ are homotopic \emph{in}
$\ol M|S$. Unfortunately, this does not seem to necessarily hold without the additional 
assumption that $S \cap M$ is carried by $\tau$ (see the discussion preceding \Cref{cl:hom_flow}).
We have introduced the stable/unstable curves of $S$, and \Cref{lem:fol}, precisely to deal with this issue.

Now define $\mc O^{\partial}|S$ 
to be the set of closed orbits $\mc O|S$ together with 
positive multiplies of the stable/unstable curves of $S$ 
contained in $\partial (\ol M|S)$. 
We call a class $\xi \in H^1(\ol M|S)$ \define{strongly positive} 
if it is positive on $\mc O^{\partial}|S$.

\begin{lemma}[Strong positivity]\label{lem:strong_pos}
A class $\xi \in H^1(\ol M|S)$ is strongly positive if and only if 
$\xi$ is positive on any oriented curve of $\ol M|S$ that is homotopic in $\ol M$ to a closed orbit of $\phi$.

Moreover, for any strongly positive $\xi \in H^1(\ol M|S)$, the pullback ${\iota_0}^*\xi \in H^1(\Phi|S)$ is positive on directed cycles. 
\end{lemma}

\begin{proof}
Let us first show that the two properties are equivalent.

By \Cref{lem:fol}, every oriented curve in $\mc O^{\partial}|S$ is homotopic in $\ol M$ to
a closed orbit of $\phi$. Hence, any class $\xi$ positive on closed orbits is positive on $\mc O^{\partial}| S$. 

Conversely, suppose that $\xi$ is positive on $\mc O^{\partial}|S$ and let $c$ be an
oriented curve in $\ol M|S$ that is homotopic in $\ol M$ to a closed orbit $\gamma$. Then
either this homotopy can be altered to live in $\ol M|S$, and so $\xi$ is positive on $c$,
or $\gamma$ is homotopic (in $\ol M$) to a stable/unstable curve in the boundary of $\ol
M|S$ (\Cref{lem:fol}) which is homotopic in $\ol M|S$ to $c$. (To see this, note first
that $\gamma$ can't cut through $S$ by positive transversality of $S$, and consider a homotopy from $c$ to $\gamma$ that is transverse to $S$.)
Hence, $\xi$ is positive on $c$.

That these statements imply 
positivity on directed cycles of $\Phi|S$
follows from \Cref{th:closed_orbits} because for any directed cycle $c$ of $\Phi|S$, ${\iota_0}(c)$ is a oriented curve in $\ol M|S$ which is homotopic in $\ol M$ to a closed orbit of $\phi$.
\end{proof}

We now turn to the proof of \Cref{th:growth_rates_closed}.

\begin{proof} [Proof of \Cref{th:growth_rates_closed}]
Let ${\iota_0} \colon \Phi|S \to \ol M|S$ be as above.  
Since $\xi$ is strongly positive, ${\iota_0}^*\xi$ is positive on all directed cycles of $\Phi|S$ 
by \Cref{lem:strong_pos}. The proof is the same as for \Cref{th:entropy-veering-nonlayered}, once we establish that the growth rate $\mathrm{gr}_{\phi|S}(\xi)$ exists and equals 
\[
\mathrm{gr}_{\Phi|S}(\xi) = \lim_{L\to \infty}  \# \{ c \in \ZF_{\Phi|S} : \xi({\iota_0}(c)) \le L   \} ^{\frac{1}{L}}.
\]
For this, a slightly more delicate argument is needed since \Cref{lem:decomposing_flow} is not available in the closed setting.  

We begin by defining a map $\mf H$ from directed cycles of $\Phi|S$ to $\mc O^\partial|S$.
To do so, we make use of the map $\mf F \colon \ZF_{\Phi} \to \mc O^+$ and use the basic
fact that since $M \subset \ol M$, $\mf F(c)$ is homotopic to ${\iota_0}(c)$ in $\ol M$ and
each unstable prong curve in $M$ is homotopic in $\ol M$ to the corresponding singular
orbit. Define a slight modification $\mf F' \colon \ZF_{\Phi} \to \ol{\mc O}$, where $\ol {\mc O}$ is the set of all closed orbits of $\phi$ in $\ol M$, by setting $\mf
F'(c) = \mf F(c)$ if $\mf F(c)$ is a nonsingular orbit.
Otherwise, $\mf F(c)$ is an unstable prong curve and we set $\mf F'(c)$ to be the
corresponding singular orbit.

To define $\mf H$, first suppose that $c$ is a directed cycle in $\Phi|S$ and that ${\iota_0}(c)$ is homotopic to $\mf F'(c)$ \emph{in} $M|S$. Then $\mf H (c) = \mf F'(c) \in \mc O^{\partial}|S$.
Otherwise, as in the proof of \Cref{lem:strong_pos}, ${\iota_0}(c)$ is homotopic in $M|S$ to some stable/unstable curve in $\partial(M|S)$. We pick such a stable/unstable curve and call it $\mf H(c)$. Note that in either case, $\mf H(c)$ is homotopic in $M|S$ to ${\iota_0}(c)$. 

Now the proof is completed exactly as in \Cref{th:entropy_nonlayered} by using the map $\mf H$ and recalling that the stable/unstable curves in $\mc O|S$ have at most linear growth. To apply that argument, it only remains to show that there is some constant $m$ such that $\# \mf H^{-1}(\gamma) \le m$ for each $\gamma \in \mc O^\partial|S$. Indeed, if $\gamma$ is a nonsingular closed orbit $\mf H$ that is interior to $M|S$, then $\# \mf H^{-1}(\gamma) \le \# \mf F^{-1}(\gamma)$ which is bounded by \Cref{th:closed_orbits}. 
If $\gamma$ is a singular orbit, then there are $\deg(\gamma)$ unstable prong curves homotopic to $\gamma$. Since each of these has at most $2$ preimages under $\mf F$, again by \Cref{th:closed_orbits}, we are also done in this case.
Finally, suppose that $\gamma$ is a multiple of a stable/unstable curve of $S$. 
Note that if directed cycles $c$ and $d$ of $\Phi|S$ have $\mf H(c) = \mf H(d) = \gamma$, then ${\iota_0}(c)$ and $i_0(d)$ are also homotopic in $M$. If $\mf F(c)$ is a closed orbit, then $\mf F(d)$ is the same closed orbit. Otherwise, $\mf F(c)$ and $\mf F(d)$ are homotopic unstable prong curves. In either case, we again obtain a bound on $\# \mf H^{-1}(\gamma)$ and the proof is complete.
\end{proof}

%% !TEX root =veering_poly2.tex

\section{Entropy functions and stretch factors} \label{sec:apps}
Here we consider some applications of \Cref{th:entropy_nonlayered} and \Cref{th:entropy-veering-nonlayered}. In \Cref{sec:entropy}, we define and establish properties of the entropy function on the cone of positive cohomology classes, and in \Cref{sec:layered_flow} we collect applications to the classical setting of fibered manifolds and stretch factors. 

\subsection{Entropy function on positive cones}
\label{sec:entropy}

Let us return to the setup of \Cref{th:entropy_nonlayered}. To simplify notation, let $N = M|_iS$ be a fixed component of $M|S$ for a surface $S$ carried by $\tau$. Similarly, let $\Phi|N = \Phi|_iS$ be the flow graph restricted to $N$ 
and note that it may have several components, each of which is strongly connected. As before, we consider $N$ with the restricted semiflow $\phi|N$ and denote by $\mc O^+|N$ its set of closed orbits and positive multiples of unstable prong curves.

Let $\mc C^+ \subset H^1(N; \RR)$ be the cone consisting of positive classes. According to \Cref{th:entropy_nonlayered}, $\mathrm{gr}_{\phi|N} \colon \mc C^+ \to [1,\infty)$ defines a function that gives the exponential growth rates of closed orbits of the flow for each $\xi \in \mc C^+$. Since the value $\mathrm{gr}_{\phi|N}(\xi)$ is given by the reciprocal of the smallest root of $P_{\Phi|N}$ specialized at $\iota^*\xi$ by \Cref{th:entropy-veering-nonlayered}, we can use results of McMullen to study its properties.

For this, we define the associated \define{entropy function}
\begin{align*}
\ent_{\phi|N} (\xi) =  \log(\mathrm{gr}_{\phi|N} (\xi))
\end{align*}
and note that \Cref{prop:exp_growth} characterizes when entropy is nonzero. Our next theorem summarizes the entropy function's basic properties.

\begin{theorem}[Entropy] \label{th:entropy2}
The entropy function $\ent_{\phi|N} \colon  \mc C^+ \longrightarrow [0,\infty)$ is continuous, convex, and has degree $-1$, i.e. $\ent_{\phi|N}(r \cdot \xi) = 1/r \cdot \ent_{\phi|N}(\xi)$ for $r>0$.
\end{theorem}

\begin{proof}
As noted above, by \Cref{prop:exp_growth} there is nothing to prove if there are only finitely many closed primitive orbits in $N$ since then the entropy function is $0$. So we assume that this is not the case. That $\ent_{\phi|N}$ has degree $-1$ follows directly from the definition. 

The restricted flow graph $\Phi|N$ is itself the disjoint union of recurrent
subgraphs. For each such component $J$,
the inclusion $\iota \colon \Phi|N \to N$ induces a pullback $\iota ^*\colon H^1(N) \to
H^1(J)$ that maps the positive cone $\mc C^+$ to the cone $\mc C^+(J)$ of positive classes
on $J$, i.e. classes that are positive on directed cycles of $J$. 
Let $\ent_J \colon \mc C^+(J) \to [0,\infty)$ denote the corresponding entropy
  function. Clearly this function is 0 when $J$ is a cycle. When it is not, since $J$ is
  strongly connected, McMullen 
\cite[Theorem 5.2]{mcmullen2015entropy} shows that $\ent_J$ is real-analytic, strictly convex, and blows up at the boundary of $\mc C^+(J)$ (i.e. tends to infinity along a sequence that converges to a point in the boundary).

From \Cref{th:entropy_nonlayered}, we know that on $\mc C^+$ entropy is equal to the
pointwise max over the components of $\Phi|N$: 
\begin{align} \label{eq:entropy}
\ent_{\phi|N} = \ent_{\Phi|N} \circ \iota^*= \max\{\ent_J\circ \iota^*  \},
\end{align}
and so we immediately obtain that $\ent_{\phi|N}$ is continuous and convex.
\end{proof}

\begin{remark}[Strongly positive cones for $\ol M|S$]
A version of \Cref{th:entropy2} also applies to the setup of \Cref{sec:closed_case}, 
where $S$ is a closed connected transverse surface in the closed manifold $\ol M$. In this case,
$\mc C^+ \subset H^1(\ol M|S)$ is the cone of strongly positive classes as defined in \Cref{subsec:strongly positive}.
\end{remark}

In the special case of a fibered cone (i.e. when $S= \emptyset$ as in \Cref{sec:layered_flow}) it is well-known that the entropy function on the interior of the fibered cone is real analytic, strictly convex, and blows up at the boundary \cite{fried1982geometry, mcmullen2000polynomial}. 
However, this does not generally need to be the case 
for $\ent_{\phi|N} \colon  \mc C^+ \longrightarrow [0,\infty)$ defined here. For example, if the manifold $N$ has a non-separating properly embedded essential annulus disjoint from all of its closed orbits, then this annulus is dual to a nontrivial cohomology class $a$ on the boundary of $\mc C^+$ that pulls back to $0$ under $\iota^*\colon H^1(N) \to H^1(\Phi|N)$. If $u\in \C^+$, then $\{u+ta\mid t\in[0,1]\}$ is a line segment in $\C^+$ on which $\ent_{\phi|N}$ is constant, so $\ent_{\phi|N}$ is not strictly convex in this case. 
Similarly, if $N$ contains an essential separating annulus disjoint from the closed orbits, then $\ent_{\phi|N}$ may not be real analytic since more than one term of the maximum in \cref{eq:entropy} may be realized.

However, more can be said if the semiflow $\phi|N$ satisfies stronger dynamical conditions. To motivate the definition first recall that, as in the proof of \Cref{lem:flow_space_prop}, the flow $\varphi$ is always transitive on $\ol M$, meaning that  it has an orbit that is dense in both the forward and backward directions. It is also well known that the closed orbits of $\varphi$ generate $H_1(\ol M; \R)$ as a vector space. 
We say that the induced semiflow $\varphi|N$ is \define{essentially transitive} 
if $\mc O^+|N$ generates $H_1(N; \R)$ as a vector space, and
if the semiflow has an orbit that accumulates on each closed orbit of $\phi|N$ in the forward direction (i.e. the closure of any forward ray contains all closed orbits) and meets every neighborhood of each end of $N$ that contains an unstable prong curve.
We note that each end of $N$ is either an annulus or torus cross an interval.

The following theorem establishes the strongest properties of $\ent_{\phi|N}$ for essentially transitive flows.

\begin{theorem} \label{th:entropy3}
If the semiflow $\varphi|N$ is essentially transitive, and $\ent_{\phi|N}$ is not
identically 0, then $\ent_{\phi|N}$ is real-analytic, strictly convex, and blows up at the boundary of $\mc C^+$.
\end{theorem}

Before beginning the proof, we require an understanding of carried surfaces and their relation to the flow space.
We define a \emph{generalized leaf} $\ell$ of the stable/unstable foliation of $\mc P$ to be either a nonsingular leaf or the union of two singular leaves at the (unique) singularity they contain. We say that a generalized stable leaf is \emph{regular to one of its sides} if either it is nonsingular or the singularity that it contains has exactly one singular unstable leaf meeting the interior of that side. The definition of a generalized unstable leaf that is regular to one of its sides is analogous. We note that a generalized leaf $\ell$ is regular to one side if and only if every finite segment of $\ell$ is contained in the boundary of a rectangle $R$; this rectangle is necessarily contained in the regular side of $\ell$.
We also define the 
\emph{boundary of an orthant} at $p$ to be the union of a singular stable leaf at $p$ and a singular unstable leaf at $p$ that are adjacent in the ordering around $p$.

Now suppose that $S$ is a connected surface carried by $\tau^{(2)}$ 
and hence transverse to the flow $\phi$. 
Consider a lift $\wt S$ of $S$ to the universal cover $\uM$, and observe
that $\wt S$ is a properly embedded, ideally triangulated plane in $\uM$ (the triangulation being induced by $\wt \tau$) that is positively transverse to the lifted
flow. Since $\wt S$ separates $\uM$, this implies that $\wt S$ intersects each
flow line at most once. Hence, the projection of $\wt S$ to the flow space $\mc P$ is a
homeomorphism onto its image and we will consider its image with the projected ideal
triangulation. As in \Cref{sec:stabcurve}, the projection of $\wt x\subset \uM$ to $\mc P$ will be denoted by $\wh x$. 

The following lemma generalizes \cite[Proposition 3.9]{cooper1994bundles} and \cite[Proposition 4.3]{fenley1999surfaces}.

\begin{lemma}
\label{lem:boundaryflow}
The topological boundary of $\wh S$ in $\mc P$ is a disjoint union of stable and unstable generalized leaves that are regular to their $\wh S$--side, along with boundaries of orthants and isolated singularities.

Boundaries of orthants correspond to punctures of $S$ whose boundary slopes are those of prong curves, and isolated singularities correspond to punctures of $S$ whose boundary slopes are not those of prong curves.
\end{lemma}

\begin{proof}
The region $\wh S$ has an ideal triangulation $\T$ inherited from $\wt S$ whose vertices are singularities of $\mc P$ in the closure of $\wh S$ and whose edges are singularity-free diagonals, i.e. $\tau$-edges. 
We will see that components of the boundary of $\wh S$, other than isolated singularities, are limit sets of edges of this triangulation and that these limit sets have the required form.

If $x$ is an isolated singular point in the boundary of $\wh S$, then $\wt S$ intersects every singular leaf meeting $x$. It follows that the corresponding puncture of $S$ has a slope which is not that of a prong curve.

Let $x$ be a nonsingular point in the boundary of $\wh S$ and let $(x_i)_{i\ge0}$ be a sequence in $\wh S$ converging to $x$. We can assume that each $x_i$ lies in the interior of an edge $e_i$ of $\T$ and that $e_i,e_{i+1}$ are incident to a common face of $\T$ for each $i\ge 0$.  
Since $x$ is not a singularity, we may further assume that the $e_i$ are distinct. 

Let $Q_i$ be the edge rectangle of $e_i$. The sequence $(Q_i)$ cannot have both an upper bound and a lower bound with respect to the `above/below' partial order on rectangles. This is due to 
the discreteness of singularities as in the proof of \Cref{fact:non-accumulation}.
Without loss of generality, suppose that there is no rectangle $R$ that lies above each $Q_i$. 
In this case,
we will see that the $Q_i$ limit to a stable leaf or to a generalized stable leaf.
The other case, where $(Q_i)$ is has no lower bound the limit is an unstable leaf or generalized leaf and is handled similarly. 

First suppose that the stable leaf $\ell$ through $x$ is nonsingular. We will show that $\ell$ is in the boundary of $\wh S$.  For this, 
let $R$ be any maximal rectangle containing a (vertical) leaf segment of $\ell$ through $x$. For sufficiently large $i$, $x_i$ is contained in the interior of $R$ and $Q_i$ does not lie below $R$. Since $Q_i$ is the edge rectangle for $e_i$ containing $x_i$, it must be that $Q_i$ lies above $R$ for large enough $i$. 
By applying the same argument to rectangles $R$ that contain larger and larger leaf segments of $\ell$ about $x$, and using the fact that such rectangles converge to $\ell$, 
we see that $Q_i$ and hence $e_i$ limit to $\ell$. 

This shows that $\ell$ is in the closure of $\wh S$ in $\mc P$. To see that it is in the boundary, it suffices to show that no point of $\ell$ is contained in $\wh S$. This is easy since any point $y \in \wh S \cap \ell$ would be contained in a face $f$ of the triangulation $\T$ of $\wh S$ which crosses $\ell$. However, $f$ would then have to be crossed by the edges $e_i$ for large $i$, contradicting that these are all cells of a fixed triangulation $\T$.

It remains to consider the case where the (stable) leaf through $x$ contains a singularity $p$. The above argument still applies with a few minor modifications. Again, let $x_i, e_i, Q_i$ be defined as above and
let $\ell$ be the stable generalized leaf through $x$, containing $p$, that is regular to its side that contains infinitely many of the $x_i$. Let $R$ be any maximal rectangle that contains a leaf segment of $\ell$ through $x$ and $p$ in its vertical boundary. If the edges $e_i$ do not eventually all terminate at the singularity $p$, then the same 
argument as above implies that $Q_i$, and hence $e_i$, limit to $\ell$. So the entire generalized leaf $\ell$ is in the boundary of $\wh S$ as required.

Otherwise, the edges $e_i$ eventually all have $p$ as a singular endpoint. In this case, the rectangles $Q_i$ and edges $e_i$ limit to the singular stable leaf $\ell'$ through $x$ terminating at $p$ (i.e. the half of $\ell$ containing $x$). Since the set of edges $e_i$ is finite up to the $\pi_1(S)$ action, and these edges all eventually have $p$ as a singular endpoint, there is a $g \in \pi_1(S)$ fixing $p$ and an edge $e$ of $\T$ with endpoint $p$ such that $(g^j(e))_{j\ge 0}$ occurs as a subsequence $(e_i)$, and hence converges to $\ell'$. This implies that $g$ stabilizes $\ell'$ and hence stabilizes all stable/unstable leaves at $p$.
But then the sequence $(g^j(e))_{j \le 0}$ converges to the unstable leaf $\ell''$ through
$p$ such that $\ell' \cup \ell''$ forms the boundary of an orthant. In this case,
one easily sees that $g \in \pi_1(S)$ is peripheral and since it fixes each prong at $p$ the corresponding 
slope is that of a prong curve as claimed. 
\end{proof}

With \Cref{lem:boundaryflow} in hand, we can turn to the proof of \Cref{th:entropy3}.

\begin{proof}[Proof of \Cref{th:entropy3}]

The theorem will follow fairly directly from the following claim: 
\begin{claim}\label{claim:big component}
 If the semiflow $\varphi|N$ is essentially transitive and $\ent_{\phi|N}$ is not identically $0$, then the graph $\Phi|N$ contains a unique
  component $J$ which is not a cycle. All curves of $\mc O^+|N$ are, up to positive multiples, homotopic to images of directed cycles in $J$.
\end{claim}

Indeed, if $J$ is such a component then  \Cref{eq:entropy} becomes
$$\ent_{\phi|N} =  \ent_{J}\circ \iota^*.$$ Since $\mc O^+|N$ generate $H_1(N; \R)$, so do
the images of directed cycles in $J$. This implies that the homomorphism $\iota^* \colon
H^1(N) \to H^1(J)$ is injective and maps the boundary of  
$\mc C^+$ into the boundary of $\mc C^+(J)$.
Since $\ent_J \colon \mc C^+(J) \to [0,\infty)$ is real-analytic, strictly convex, and
  blows up at the boundary (again by \cite[Theorem 5.2]{mcmullen2015entropy}) this implies
  the same for  $\ent_{\phi|N} \colon  \mc C^+ \to [0,\infty)$.

    \medskip

We now proceed with the proof of \Cref{claim:big component}.

Let $\gamma$ be an orbit of $\phi|N$ which, in the forward direction, accumulates on every
closed orbit in $\mc O|N$ and meets every neighborhood of each end of $N$ that contains an
unstable prong curve.

Fix a lift $\wt N$ to $\wt M$ and let $\wt \gamma$ be a lift of $\gamma$ to $\wt N$, which is determined up to the action of $\pi_1(N)$. Let $p = \wh \gamma$ be its projection to the flow space $\mc P$ and note that $p$ is not contained in a singular stable leaf since otherwise $\gamma$ would be attracted to a singular orbit in the forward direction.

Let $D_p$ be the dynamic plane for $p$ given after \Cref{prop:planes_for_leaves} and let $\wt \gamma_\Phi$ be a $\wt \Phi$-line such that $\wt{\mf F} (\wt \gamma_\Phi) = p$, the existence of which is guaranteed by \Cref{prop:flowline_bound}. Note that either by the construction of $\wt \gamma_\Phi$ or \Cref{prop:planes_for_leaves}, we know that $\wt \gamma_\Phi$ is contained in $D_p$. Let $\gamma_\Phi$ be the projection of $\wt \gamma_\Phi$ to $\Phi$. We claim that 
\begin{enumerate}
\item $\gamma_\Phi$ is disjoint from $S$, and
\item for any directed cycle $c$ of $\Phi|N$, any directed subray $\gamma_\Phi^+$ of the the bi-infinite path $\gamma_\Phi$ contains a closed subpath $d$ such that as loops in $N$, $d$ is homotopic to $c^k$ for some $k\ge 1$.
\end{enumerate}
Note that the second item implies that either $\mf F(d) = \mf F(c^k)$ or $\mf F(d)$ and $\mf F(c^k)$ are homotopic unstable prong curves corresponding to the same end of $N$ (\Cref{th:closed_orbits}).

Let us show how \Cref{claim:big component} follows from these two subclaims.  By claim
$(1)$ above, $\gamma_\Phi$ lies in some component of $\Phi \ssm S$ and so some subray
$\gamma_\Phi^+$ lies in some component $J$ of $\Phi|S$.  If $c$ is any directed cycle of
$\Phi|N$, then  by claim $(2)$ there is some directed cycle of $J$ whose image under $\mf F$ is an
element of $\mc O^+|N$ that is homotopic to a multiple of $\mf F(c)$. In
particular, the cone in $H_1(J ; \R)$ positively generated by directed cycles maps onto
the cone in $H_1(N; \R)$ positively generated by $\mc O^+|N$ (see
\Cref{lem:decomposing_flow}). This gives the second statement in the claim.
Moreover, item $(3)$ of \Cref{th:closed_orbits} gives that
for all but finitely many primitive directed cycles $c$ of $\Phi$ (i.e. the ones for which
$\mf F(c)$ is not a prong curve and not homotopic to an AB cycle), $c^k = \mf F^{-1}(\mf F(c^k))$
for all $k\ge 1$ and in fact no other directed cycles of $\Phi$ are homotopic
to $c^k$ in $M$. It follows that, outside finitely many exceptions, every primitive
directed cycle of $\Phi|N$ is actually in $J$, 
so all components of $\Phi|N$ other than
$J$ are cycles. If $J$ were a cycle too, then $\ent_{\phi|N}$ would be identically $0$. This
proves the Claim.

It remains to establish the two subclaims. For the first, suppose that $\gamma_\Phi$ intersects some component $S'$ of $S$. Then $\wt \gamma_\Phi$ intersects some lift $\wt S'$ of $S'$ to $\wt M$ in some face $\wt f$ of the triangulation on $\wt S'$ induced by $\wt \tau$. Let $f\subset \wh S'$ be the corresponding triangle in $\mc P$ 
and let $R_f$ be the face rectangle determined by $f$. Since $\wt \gamma_\Phi$ is a $p$-line, $R_f$ contains the point $p$. 
If the regular point $p$ is not contained in $\wh S'$ then by \Cref{lem:boundaryflow} either the vertical or horizontal leaf through $p$ is also disjoint from $\wh S'$. But each side of $R_f$ contains a singular vertex of $f$ and so in this case, the vertical or horizontal leaf through $p$ would have to cut through $f$, giving a contradiction. This implies that $p \in f$ and so the orbit $\wt \gamma$ also intersects $\wt S'$. This, however, contradicts the assumption that $\gamma$ is contained in $N$ where $N$ is a component of $M|S$.

For the second subclaim, fix a directed subray $\wt \gamma_\Phi^+$ of $\wt \gamma_\Phi$ with initial maximal $p$-rectangle $R_0$ and let $c$ be any directed cycle of $\Phi|N$.
Let $\wt c$ be a lift of $c$ to $\wt N$ and choose $g \in \pi_1 (N)$ to generate its stabilizer so that it translates $\wt c$ in its positive direction. We set $q = \wt{\mf F}(\wt c)$,
set $\gamma_c = \mf F(c)$ , and let $\wt \gamma_c$ be the lift to $\wt N$ that is also stabilized by $g$. Note that the projection of $\wt \gamma_c$ to $\mc P$ is $q$, which is also stabilized by $g$.

To complete the proof, we first assume that $q$ is a regular point.
Fix a maximal $q$-rectangle $R$ along the $\wt \Phi$-line $\wt c$ and let $D_q$ be the
dynamic plane for $q$, which contains $\wt c$ by \Cref{prop:planes_for_leaves}. Also let
$n$  be the number of vertices in the chains of sectors associated to the sectors immediately above $\sigma(R)$ in
$D_q$, as in \Cref{rmk:fastconverge}. Here we recall that $\sigma(R)$ is the sector
immediately below the vertex in $D_q$ corresponding to $R$.

The fact that $\gamma$ accumulates on $\gamma_c$ in its positive direction translates into
the statement that there is a sequence $h_i \in \pi_1(N)$ such that $h_i p \to q$ and that
$h_i R_0$ eventually lies below the $q$-rectangle $g^{-(n+1)}(R)$. To see this, fix an
equivariant family of rectangle sections $\{s_R\}$ as in \Cref{lem:flow and rectangles} and let $x$
be the intersection point of $\wt\gamma_c$ with the section over $g^{-(n+1)}(R)$. 
Let $y$ be the intersection point of $\wt\gamma$ with the section over $R_0$ and let $r_0$ be
the positive subray of $\wt\gamma$ starting at $y$. The positive accumulation of $\gamma$
on $\gamma_c$ implies there exist $h_i\in\pi_1(N)$ and $t_i\to\infty$ such that
$h_i(r_0(t_i))$ converges to $x$. We may choose $t_i$ so that the flow segment
$h_i r_0([0,t_i])$ begins at the section over $h_i(R_0)$ and ends at the section over $g^{-(n+1)}(R)$. 
Applying \Cref{lem:flow and rectangles} now  tells us that,  for $i$ sufficiently large,
$g^{-(n+1)}(R)$ lies above $h_i R_0$. This is the desired statement. 

Further choose 
 $i$ sufficiently large  that $h_i  p$ lies in $R \cap g^{-(n+1)}(R)$. Hence, the descending set $\Delta(\sigma(R)) \subset D_q$ is also contained in the dynamic plane $D_{h_ip}$ (\Cref{prop:planes_for_leaves}).
 Moreover, $D_{h_ip} = h_i D_p$ contains $h_i \wt \gamma_\Phi^+$ 
 whose initial maximal rectangle $h_i R_0$ corresponds to a vertex contained in  $\Delta(\sigma(g^{-(n+1)}(R)))$.
Then by \Cref{lem:speed_converge} and \Cref{rmk:fastconverge},
$h_i \wt \gamma_\Phi^+$ must pass through a vertex in the chain of one of the sectors above $g^{-i}(\sigma(R))$ for each $i=1,\dots, n+1$. Since for each $i$ there are $n$ of these vertices, $h_i \wt \gamma_\Phi^+$ must pass through two vertices of $D_p$ which lie in the same $\langle g\rangle$-orbit. Hence, there is a subpath $\wt d$ of $h_i \wt \gamma_\Phi^+$ such that $g^k$
 takes its initial vertex to its terminal vertex. In $\Phi|N$, this projects to a closed subpath $d$ of $\gamma_\Phi^+$ that is homotopic to $c^k$ as a loop in $N$, establishing the second claim when $q$ is regular.

 When $q$ is a singular point only minor modifications to the setup are needed. 
 In this case, $\gamma_c$ is an unstable prong curve and $\wt \gamma_c$ is its lift determined by an unstable singular leaf $\ell^u$ emanating from $q$. This time $g \in \pi_1(N)$ stabilizes $\ell^u$ 
and therefore it stabilizes each orthant based at $q$. The fact that $\gamma$ meets each neighborhood of the end of $N$ corresponding to $\gamma_c$ implies that there is a sequence $h_i \in \pi_1(N)$ such that $h_i p \to q$. Since the stabilizer of $q$ acts cofinitely on the orthants at $q$, we can also assume that the $h_ip$ all lie in a single half-plane (i.e. union of two adjecent orthants) cobounded by two consecutive singular stable leaves $\ell_1, \ell_2$ emanating from $q$. There is a unique dynamic plane $D_q$ containing the dynamic half-planes $D_{\ell_1}, D_{\ell_2}$, which can be characterized as the union of descending sets $\Delta(\sigma(R'))$ where $R'$ is a maximal rectangle with $q$ in its vertical boundary that is contained in the half-space at $q$ cobounded by $\ell_1,\ell_2$ (see \Cref{rmk:singularplanes}). Note that $g$ stabilizes $D_q$ and so there is a $g$-periodic $\wt\Phi$-line $\wt c'$ in $D_q$ whose image $c'$ in $N$ is homotopic to $c$.
The rest of the proof now goes through as above after replacing $c$ with $c'$.
\end{proof}

\begin{remark}
Our definition of essentially transitive concerns an orbit which ``sees" every orbit and every unstable prong curve in the forward direction. In fact the conclusions of \Cref{th:entropy3} also hold if $\phi|N$ has an orbit which in the \emph{backward} direction accumulates on every closed orbit and meets every neighborhood of every end of $N$ containing a \emph{stable} prong curve. Indeed, after reversing the orientation of $\phi$ and the coorientation of $\tau^{(2)}$, we can apply the argument from above to conclude that the corresponding entropy function is real analytic, strictly convex, and tends to infinity at the boundary of the positive cone. This implies that the original entropy function has the same properties.
\end{remark}

\subsection{Suspension flows and fibered cones}
\label{sec:layered_flow}
\label{sec:boundary_cone}

Again returning to \Cref{th:entropy-veering-nonlayered} (or \Cref{th:growth_rates_closed} in the closed case), if we let $S = \emptyset$, then there exists a (strongly\footnote{When $S=\emptyset$, all positive classes are strongly positive.}) positive class $\xi$ in $H^1(M)$ (or $H^1(\ol M)$)
 if and only if the flow
$\phi$ is isotopic to the suspension flow of a pseudo-Anosov homeomorphism and $\xi$ lies in the interior of the associated \emph{fibered cone} $\R_+{\bf F}$.
This follows from either Fried's criterion for the existence of cross sections \cite[Theorem D]{fried1982geometry} or a combinatorial analogue proven in \cite[Theorem E]{LMT20}. Hence, we conclude that the growth rate $\mathrm{gr}_\phi(\xi)$ of $\phi$'s closed orbits with respect to $\xi$ is given by the reciprocal of the smallest positive root of the specialization $V_\tau^\xi$ of the veering polynomial.

\begin{remark}[Teichm\"uller polynomial]
Applying \Cref{th:entropy-veering-nonlayered} in this setting to the primitive integral points in the interior of $\R_+{\bf F}$, and using the connection to the Teichm\"uller polynomial established in \cite[Theorem B]{LMT20},
we recover McMullen's theorem \cite[Theorem 5.1]{mcmullen2000polynomial} that the Teichm\"uller polynomial computes stretch factors of monodromies associated to the fibered cone $\R_+{\bf F}$.
\end{remark}

\smallskip

Combining \cite[Theorem E]{LMT20} with the above discussion, $\phi$ is circular (i.e. admits a cross section) if and only if the associated veering triangulation $\tau$ is \define{layered} (i.e. admits a fully carried surface) and this occurs if and only if the associated cone $\cone_2(\tau)= \R_+{\bf F}$ is fibered (see \Cref{th:cones}). In this setting, we call ${\bf F}$ a \define{fully punctured} fibered face.

We next focus on the case in which $S$ represents a class in the boundary of the fibered cone $\R_+{\bf F}$. 
To this end, let $\tau$ be a layered veering triangulation with dual flow $\varphi$ and let $S$ be a connected surface carried by $\tau^{(2)}$ that is \emph{not} a fiber. We remark that every primitive integral class in $\partial(\R_+{\bf F})$ is represented by such a surface.
Then $M|S$ is connected and any
$\xi \in H^1(M)$ dual to a class in $\mathrm{int}(\R_+{\bf F})$ 
pulls back under $M|S \to M$ to a 
positive class in $H^1(M|S)$.
Hence, \Cref{th:entropy_nonlayered} and \Cref{th:entropy-veering-nonlayered}
give the growth rate of the closed orbits missing $S$, and in this case more can be said.

Let $\bf F$
be the fibered face associated to $\tau$ 
and fix a subface $\bf S \subset \bf F$. By the relative interior of the cone $\RR_+\bf S$ we mean the cone on ${\bf S} \ssm \partial {\bf S}$ (i.e. the interior of $\RR_+\bf S$ within the subspace it spans). 
If $V = V_\tau$ is the veering polynomial of $\tau$, let $V|\bf S$ be the polynomial obtained by {deleting} the terms that pair positively with $\bf S$.
 In more details, if $V = \sum a_g g$, the 
\[
V|{\bf S}  = \sum a'_g g,
\]
where $a'_g = a_g$ if $\eta(g) = 0$ for some $\eta$ in the relative interior of $\RR_+\bf S$ and $a'_g = 0$ otherwise (c.f. \Cref{prop:polys}).
We note that this definition does not depend on the choice of $\eta$ in the relative interior of $\RR_+\bf S$. This follows from the fact that the cone of homology directions $\cone_1(\Gamma)$ and $\cone_2(\tau)=\R_+{\bf F}$ are dual (\Cref{th:cones}(2)). 
Indeed, the basic theory of convex polyhedral cones in finite-dimensional vector spaces (see e.g \cite[Section 1.2]{Ful93}) gives that if $\eta,\eta'$ are two classes lying in the relative interior of $\R_+{\bf S}$ then $\ker(\eta)\cap \cone_1(\Gamma)=\ker(\eta')\cap \cone_1(\Gamma)$.

In this setting \Cref{th:entropy_nonlayered} and \Cref{th:entropy-veering-nonlayered} easily imply the following:

\begin{corollary}[Counting orbits missing transverse surfaces]
\label{cor:boundary_fibered}
Suppose that $M$ has a fully punctured fibered face $\bf F$. Let $\tau$ be the associated veering triangulation and $\phi$ the associated suspension flow. Finally, fix a subface $\bf S$ of $\bf F$ and let $\eta \in \intr (\RR_+\bf S)$.

For any $\xi \in \intr(\RR_+\bf F)$, the growth rate
\begin{align}\label{eq:relrate}
\mathrm{gr}_\phi(\xi; {\bf S}) = \lim_{L\to \infty}  \# \{ \gamma \in \mc O_\phi : \eta(\gamma) =0 \text{ and }  \xi(\gamma) \le L   \} ^{\frac{1}{L}}.
\end{align}
exists and equals
the reciprocal of the smallest root of the specialization $V|{\bf S}^\xi$. 

Moreover,
\begin{enumerate}
\item The growth rates $\mathrm{gr}_\phi(\xi; {\bf S})$ depends only on the face $\bf{S}$ and not the chosen $\eta$.
\item If $S$ is any surface carried by $\tau$ dual to a class in $\intr (\RR_+\bf S)$, then $\mathrm{gr}_\phi(\xi; {\bf S})$ computes the growth rate (with respect to $\xi$) of closed orbits that miss the surface $S$ and is equal to $\mathrm{gr}_{\phi|S}(\xi)$ from \Cref{eq:entropy_semiflow_non}.
\item The growth rate $\mathrm{gr}_\phi(\xi; {\bf S})$ is strictly larger than $1$ if and only if there are infinitely many primitive closed orbits that are $\eta$-null.
\end{enumerate}
\end{corollary}

We remark that a straightforward calculation shows that $V|S\in \Z[H_1(M)/\text{torsion}]$ is equal to the image of $P_{\Phi|S}$ under the map induced by the inclusion $\Phi|S\to M|S\to M$, regardless of whether ${\bf F}$ is fibered. However, in order for $V|S$ to output interesting dynamical information as in the above result, the fibered hypothesis is essential: there exists a class $\xi\in H_1(M)$ which pulls back to a positive class on $M| S$ if and only if ${\bf F}$ is fibered. For the less trivial direction of this statement, note that if $\xi\in H^1(M)$ is a class pairing positively with every closed orbit that has zero pairing with $\eta$, then $\xi+k\eta$ 
pairs positively with every closed orbit of $\phi$ for sufficiently large $k$. As a consequence, $\xi+k\eta$ is dual to a cross section to $\phi$.

Also, we again emphasize that \Cref{cor:boundary_fibered} has a natural generalization to closed manifolds by first puncturing along singular orbits of the suspension flow.

\begin{remark}[Depth one foliations and stretch factors of endperiodic monodromies]
\label{rmk:endperiodic}
The growth rates appearing in \Cref{cor:boundary_fibered} when $\xi$ is integral can be naturally interpreted as stretch factors of endperiodic homeomorphisms associated to depth one foliations of $M$ (or more precisely its compact model as in \Cref{rmk:cmpt}). 
 Indeed, if $S$ is a surface carried by $\tau$ that is not a fiber, then any primitive integral class $\xi$ in the interior of the associated fibered cone gives rise to a depth one taut oriented foliation on $M|S$ that is positively transverse to flow lines of $\phi$ (see for example \cite[Theorem 3.7]{agol2008criteria}). The foliation restricted to the complement of the boundary (depth zero) leaves is a fibration over the circle and the first return map to a fiber (i.e. a depth one leaf) is a endperiodic homeomorphism \cite[Lemma 4.1, 4.2]{fenley1992asymptotic}.
The growth rate of periodic points of the first return map is equal to $\mathrm{gr}_{\phi|S}(\xi)$, giving a direct generalization of the stretch factor of a pseudo-Anosov homeomorphism.
These stretch factors will be the subject of future work \cite{LMTspA}.
\end{remark}

We can use these tools to answer the following question of Chris Leininger:

\begin{question}[Leininger] \label{q:Lein}
Given a fibered face $\bf F$ of a hyperbolic $3$--manifold $M$, what is the limit set of stretch factors arising from monodromies whose fibers correspond to integral points in $\RR_+ \bf F$?
\end{question}

It is clear that $1$ is such an accumulation point, but in unpublished work Leininger and Shixuan Li have produced examples where there are accumulation points greater than $1$.

To answer \Cref{q:Lein}, we introduce the following notation: 
For each subface $\bf S$ of $\bf F$ define
\[
\Lambda( {\bf S}) = \{\mathrm{gr}_\phi(\alpha; {\bf S}) : \alpha \text{ is an integral point of } \mathrm{int}(\RR_+ \bf F) \}.
\]
where $\mathrm{gr}_\phi(\alpha; {\bf S})$ is as in \Cref{eq:relrate}.
Also set $\Lambda = \Lambda( \emptyset)$, which is exactly the set of stretch factors of the monodromies of fibrations corresponding to integral points in $\RR_+ \bf F$. Our goal is to understand its closure $\ol \Lambda$.

Recall that $X^\prime$ denotes the derived set of $X$, i.e. its set of accumulation points. Also inductively set $X^0 = X$ and $X^{n+1} = (X^n)^\prime$.

\begin{theorem}[Structure of stretch factors]
\label{th:structure_stretch}
Let $\Lambda \subset (1,\infty)$ be the set of stretch factors of the monodromies of fibrations corresponding to integral points in $\RR_+ \bf F$. Then its closure $\ol \Lambda$ is compact, well-ordered under $\ge$,
and ${\ol \Lambda}^n = \{1\}$ for some $1\le n \le \mathrm{dim} (H^1(M; \RR))$. 

Moreover,
\begin{itemize}
\item each number in $\ol \Lambda^\prime \ssm \{1\}$ is itself a growth rate in the sense of \Cref{cor:boundary_fibered} and an infinite type stretch factor in the sense of \Cref{rmk:endperiodic}, 
\item the accumulation set $\ol \Lambda^\prime$ is infinite if and only if 
 there are infinitely many primitive orbits in $\mc O_\phi$
 that are null with respect to some class in $\partial (\RR_+\bf F)$, and
\item the derived length is maximal (i.e. $n = \mathrm{dim} (H^1(M; \RR))$) if and only if there are infinitely many primitive orbits in $\mc O_\phi$
that represent a multiple of a vertex class in the cone of homology directions in $H_1(M; \RR)$.
\end{itemize}
\end{theorem}

\begin{proof}
In the proof, we assume that the fibered face $\bf F$ is fully punctured and associated to the veering triangulation $\tau$ of $M$. The general case then follows from puncturing along singular orbits and considering only cohomology pulled back from the original manifold.

We begin by establishing a more technical claim. 

\begin{claim}[going up]
\label{cl:going}
Let $\bf S$ be a face of $\bf F$. Then
\[
\ol {\Lambda( {\bf S})}^\prime \ssm \{1\} = \bigcup_{{\bf T\supset S}} \Lambda( {\bf T}) \ssm\{1\},
\]
where the union is over proper faces ${\bf T}$ of $\bf F$ that \emph{properly} contain $\bf S$.

Moreover, if $(\lambda_k)$ is a sequence in $\Lambda( {\bf S})$ converging to $\lambda \in
\ol {\Lambda( {\bf S})}^\prime$, then $\lambda \le \lambda_k$ for sufficiently large $k$.
\end{claim}

\begin{proof}[Proof of claim]
Any integral $\alpha$ in the interior of  $\RR_+ \bf F$ can be realized as a (multiple of a) fiber surface $S_\alpha$ carried by $\tau$. We note that while the isotopy class of $S_\alpha$ is unique, its carried position is not, but this will not matter here. Since the image of $\Phi$ in $M$ is positively transverse to $\tau^{(2)}$, the nonnegative integral cocycle $m_\alpha$ on $\Phi$ given by mapping each directed edge to its intersection number with $S_\alpha$ represents the pullback of $\alpha$ to $\Phi$.  Obviously, the restriction of $m_\alpha$ to any subgraph of $\Phi$ represents the pullback of $\alpha$ to that subgraph.

We first prove the containment $\Lambda( {\bf T}) \ssm \{1\} \subset \ol {\Lambda( {\bf S})}^\prime
\ssm \{1\}$ for each ${\bf T} \supset {\bf S}$. Fix an integral class $\eta$
in the relative interior of $\RR_+{\bf T}$,
and let $\Phi|{\bf T} = \Phi|\eta$ be the subgraph of $\Phi$ covered by cycles that are $\eta$--null (as in \Cref{sec:cutting_co}). 
Also fix $\alpha \in \mathrm{int}(\RR_+{\bf F})$ so that $\mathrm{gr}(\alpha; \Phi|{\bf T})>1$. 
For $i>0$ we note that $\alpha + i\eta$
and $\alpha$ agree on cycles of $\Phi|{\bf T}$, while the value of $\alpha + i\eta$
on any cycle of $\Phi|{\bf S}$ not contained in $\Phi|{\bf T}$ goes to $\infty$ with
$i$. At this point we use the following lemma about growth rates in graphs. It is probably
well-known but for completeness we will include a proof at the end. 

For a directed graph $D$ and positive class $\alpha \in H^1(D)$ (i.e. class that is positive on directed cycles), let $\gr(\alpha;D)$ denote the growth rate of directed cycles in $D$ with respect to $\alpha$ (as in \Cref{eq:entropy_flowgraph_non}).
\begin{lemma}\label{lem:rates_converge}
Let $D$ be a directed graph with subgraph $D'$. Let $\alpha_i$ be a sequence of positive classes in $H^1(D)$  that pull back to the same positive class $\alpha \in H^1(D')$. Then 
\[
\mathrm{gr}(\alpha_i; D) \ge \mathrm{gr}(\alpha; D').
\]
Suppose further that
\begin{enumerate}
\item the $\alpha_i$ blow up on the complement of $D'$ (i.e. $\alpha_i(\gamma) \to \infty$ for each directed cycle $\gamma$ of $D$ that is not contained in $D'$), and
\item $\liminf_{i \to \infty} \mathrm{gr}(\alpha_i; D )> 1$.
\end{enumerate}
Then
\[
\mathrm{gr}(\alpha_i; D) \to \mathrm{gr}(\alpha; D')
\]
as $i \to \infty$.
\end{lemma}

We apply \Cref{lem:rates_converge} 
to conclude that $\mathrm{gr}(\alpha + i \cdot \eta; \Phi|{\bf S}) \ge \mathrm{gr}(\alpha; \Phi|{\bf T}) > 1$ and that
\[
\mathrm{gr}(\alpha + i \cdot \eta; \Phi|{\bf S}) \to \mathrm{gr}(\alpha; \Phi|{\bf T})
\]
as $i\to \infty$, and note that this sequence is nonconstant exactly when the containment
${\bf T} \supset {\bf S}$ is proper.
By \Cref{th:entropy_nonlayered} (and \Cref{cor:boundary_fibered}(2)) this gives us
\[
\mathrm{gr}_\varphi(\alpha + i \cdot \eta; {\bf S}) \to \mathrm{gr}_\varphi(\alpha; {\bf T})
\]
Thus any point of $\Lambda({\bf T})$ is a limit 
point of $\Lambda({\bf S})$.

\medskip

Conversely,  let $\alpha_i \in \mathrm{int}(\RR_+{\bf F})$ be a sequence of integral
classes so that the sequence of growth rates $\lambda_i = \gr_\varphi(\alpha_i; {\bf S})$
is pairwise distinct and 
converges to $\lambda >1$, and let us show that $\lambda\in \Lambda({\bf T})$ for some
face ${\bf T \supsetneq S}$.

Now let $\eta$ denote
an integral class in the relative interior of $\RR_+{\bf S}$
(if ${\bf S} = \emptyset$, then by convention $\eta = 0$ and $\Phi|\eta=\Phi|{\bf S} = \Phi$). 
Replace each $\alpha_i$ in this sequence with $\alpha_i + i \cdot \eta$. This does not
change $\lambda_i = \mathrm{gr}(\alpha_i; {\bf S})$, but it ensures that $\alpha_i
(\gamma) \to \infty$ for any directed cycle of $\Phi$ that is not in $\Phi|{\bf S}$.

After passing to a subsequence, we may assume that for each edge $e$ of $\Phi|{\bf S}$ 
either $m_{\alpha_i}(e)$ stays bounded for all $i$ or $m_{\alpha_i}(e) \to \infty$. Let $E$ be the set of edges whose lengths stay bounded. 
Because each $m_{\alpha_i}$ is integral, we may pass to a further subsequence
and assume that $m_{\alpha_i}(e) =: m(e)$ is constant 
for each edge $e$ of $E$.

We can again apply \Cref{lem:rates_converge} to the pullback of $\alpha_i$ on the graphs
$E$ and $\Phi|{\bf S}$, concluding
$$
\gr(\alpha_i; \Phi|{\bf S}) \to \gr(\alpha;E).
$$
This limit is then $\lambda$ since $\gr(\alpha_i; \Phi|{\bf S}) = \gr_\varphi(\alpha_i;{\bf   S})$ 
by \Cref{th:entropy_nonlayered} and \Cref{cor:boundary_fibered}. Since $\lambda >1$, $E$ has a nontrivial recurrent subgraph. 
It remains to find a face ${\bf T}$ strictly containing ${\bf S}$ such that
$\gr(\alpha;E) = \gr(\alpha;\Phi|{\bf T})$.

Set $\eta_i = \alpha_i - \alpha_1$. We claim that for $i$ sufficiently large, $\eta_i$ is
contained in the relative interior of $\RR_+\bf T$ for some face ${\bf T}$ that properly
contains ${\bf S}$. Let $\gamma$ be a directed cycle in $\Phi$.
If $\gamma$ is not contained in $\Phi|\bf S$, then $\eta(\gamma) >0$ and so
$\alpha_i(\gamma) \to \infty$ by definition of $\eta$. Hence, $\eta_i(\gamma) >0$ for large $i$. 
If $\gamma$ is contained in $\Phi| \bf S$ but not in $E$, then again
$\alpha_i(\gamma)\to\infty$ by definition of $E$ so $\eta_i(\gamma) > 0$
for large $i$. If $\gamma$ is contained in $E$ then $\alpha_i(\gamma) = \alpha_1(\gamma)$
so $\eta_i(\gamma)=0$. At any rate $\eta_i(\gamma) \ge 0$ and since 
the cone of homology directions is spanned by finitely many cycles in $\Phi$
(\Cref{th:cones}),
we may fix $i$ so that $\eta_i \in \RR_+ {\bf F}$. 

Let ${\bf T}$ be the face of $\mathbf F$ such that $\eta_i$ is in the relative interior of
$\RR_+{\mathbf T}$. Then from the previous paragraph we see that $E$ and $\Phi|{\mathbf
  T}$ have the same directed cycles, namely those where $\eta_i$ vanishes. 
Thus $\gr(\alpha;E) = \gr(\alpha;\Phi|{\bf T})$, and the latter equals
$\gr_\varphi(\alpha;\mathbf T)$ by \Cref{th:entropy_nonlayered}.
Note that $\mathbf S$ is a proper subface of $\mathbf T$ because we have assumed the
$\lambda_i$ are not eventually constant.

Applying this to all limit points $\lambda$ we obtain the containment
\[
\ol {\Lambda( {\bf S})}^\prime \ssm \{1\} \subset \bigcup_{{\bf T}\supsetneq {\bf S}}
\Lambda( {\bf T}) \ssm\{1\}. 
\]

The final statement, that eventually $\lambda_i \ge \lambda$,
follows from the first conclusion of \Cref{lem:rates_converge}, and the fact that
$\lambda_i = \gr(\alpha_i;\Phi|\mathbf S)$
again by \Cref{th:entropy_nonlayered}. This concludes the proof of
\Cref{cl:going}. 
\end{proof}

The claim now immediately implies that $\ol \Lambda$ is well-ordered by $\geq$ and that the length of the derived sequence is bounded above by $\mathrm{dim} (H^1(M; \RR))$. Compactness of $\ol \Lambda$ was previously observed by Leininger (see also \cite[Theorem A]{fried1982flow}), but it also follows from our setup.
First recall that as in the proof of \Cref{cl:going}, the pullback to $\Phi$ of each integral class $\alpha$ in $\R_+\bf F$ can be represented by a nonnegative, integral cocycle $m_\alpha$ that is positive on directed cycles of $\Phi$. By \Cref{th:entropy_nonlayered}, to show that $\Lambda$ is bounded above, it suffices to show that $\mathrm{gr}_\Phi([m])$ is uniformly bounded over all nonnegative, integral cocycles $m$ representing a positive class $[m]\in H^1(\Phi)$. This is straightforward: if $m$ is such a cocycle, then we obtain another such cocycle $m'$ by declaring that $m'(e) =0$ if $m(e)=0$ and $m'(e)=1$ otherwise, for each directed edge $e$ of $\Phi$. Since $m$ is integral and nonnegative, $m'(e) \le m(e)$ for all directed edges $e$ of $\Phi$. This implies that $\mathrm{gr}_\Phi([m]) \le \mathrm{gr}_\Phi([m'])$. But since there are only finitely many cocycles taking values in $\{0,1\}$, there is a maximum to their growth rates (after restricting to the ones that are positive on directed cycles). Hence, $\Lambda$ is bounded and so $\ol \Lambda$ is compact.

It only remains to prove the additional items.
The first item follows from the proof of \Cref{cl:going}. The second item follows from \Cref{cor:boundary_fibered}(3), since if $(\alpha_i)$ is a sequence of classes with $\mathrm{gr}(\alpha_i)\to x>1$, then for all $n\in \mathbb{N}$ we have $\mathrm{gr}(n\alpha_i)=\mathrm{gr}(\alpha_i)^{\frac{1}{n}}\to x^{\frac{1}{n}}$.
Finally, for the third item, it is easy to see (again by \Cref{cor:boundary_fibered}(3)) that the derived length is maximal if and only if for some $\eta$ in the relative interior of a top dimensional face $\bf S$ of $\bf F$, there are infinitely many closed primitive orbits that are $\eta$--null. All such orbits must represent a multiple of the vertex of the cone of homology directions that is dual to $\mathbb{R}_+\bf S$. This completes the proof of \Cref{th:structure_stretch}.
\end{proof}

We conclude with a proof of \Cref{lem:rates_converge}. 
Instead of assuming that the $\alpha_i$ pull back to the same class on $D'$ it in fact suffices to assume that the pullbacks to $H^1(D')$ converge, but we will only need the weaker statement. Also, condition $(2)$ could be replaced by the condition that $\mathrm{gr}(\alpha; D') \ge 1$, i.e. that $D'$ contains a directed cycle, but we have chosen to state \Cref{lem:rates_converge} so that it can be directly applied in the proof of \Cref{th:structure_stretch}.

\begin{proof}[Proof of \Cref{lem:rates_converge}]
Set $\lambda_i = \mathrm{gr}(\alpha_i; D)$ and $\lambda = \mathrm{gr}(\alpha;
D')$. Clearly, $\lambda \le \lambda_i$ since $D' \subset D$.

Now assume items $(1)$ and $(2)$ from the lemma statement.
We claim that $\lambda_i$ are bounded above: 
for $i$ sufficiently large $\alpha_i(\gamma) \ge \alpha_1(\gamma)$ for all directed cycles
$\gamma$ of $D$ and this implies that $\mathrm{gr}(\alpha_i; D) \le \mathrm{gr}(\alpha_1;
D)$.  Thus it suffices to show that any accumulation point $\mu\ge \lambda$ of
$(\lambda_i)$ is equal to $\lambda$. 

Let $P_D$ be the Perron polynomial of $D$. From \cref{eq:cliquepoly}, we see that this is a sum 
\[
P_D = P_{D'} + N,
\]
where $P_{D'}$ is the Perron polynomial of $D'$ consisting of the terms of $P_D$ that correspond to cycles contained in $D'$ and where $N$ has terms corresponding to cycles that are not contained in $D'$. Specializing (as in \Cref{sec:poly}), we get
\[
P_D^{\alpha_i}(t^{-1}) = P_{D'}^{\alpha}(t^{-1}) + N^{\alpha_i}(t^{-1}),
\]
where the largest real root of $P_D^{\alpha_i}(t^{-1})$ is $\lambda_i$ and the largest real root of $P_{D'}^{\alpha}(t^{-1})$ is $\lambda$ (see \cite[Theorem 1.2]{mcmullen2015entropy}), unless $D'$ contains no directed cycles. In this last case, we would have that $P_{D'} = 1$ and $\lambda = 0$.

Since the $\alpha_i$ blow up on loops not in $D'$, $N^{\alpha_i}(t^{-1})$ is a finite sum of terms of the form $a t^{-x_i}$ where $x_i \to \infty$ as $i \to \infty$. 

Let $\mu$ be an accumulation point of $(\lambda_i)$. Then $\mu\ge \lambda > 1$. Passing to
a subsequence we may assume $\lambda_i \to \mu$, and plugging into the specializations we obtain
\[
0 = P_{D'}^{\alpha}(\lambda_i^{-1}) + N^{\alpha_i}(\lambda_i^{-1}).
\]
Then using the above description of $N^{\alpha_i}$ and that fact that $\lambda_i \to \mu >1$, we see that $N^{\alpha_i}(\lambda_i^{-1}) \to 0$ as $i \to \infty$. So
by continuity of $P_{D'}^{\alpha}(t^{-1})$, we get that $\mu$ is a root of
$P_{D'}^{\alpha}(t^{-1})$. Since $\lambda$ is the largest root, we conclude $\mu =
\lambda$. This completes the proof of \Cref{lem:rates_converge}. 
\end{proof}

\bibliography{veering_poly2.bib}
\bibliographystyle{amsalpha}

\end{document}